\documentclass[12pt]{article}
\usepackage{mathptmx}

\usepackage[affil-it]{authblk}
      \usepackage{latexsym}
         \usepackage[reqno, namelimits, sumlimits]{amsmath}
         \usepackage{amssymb, amsfonts}
         \usepackage{amsthm}
	 \usepackage{marginnote}
	 \usepackage{comment}
	 \usepackage{enumerate}

	 \usepackage{color}
	 \usepackage{esint} %This is for fint command
	 \usepackage{graphicx}
	 \usepackage[margin=1in]{geometry}

	 \theoremstyle{plain}

 \newtheorem{theorem}{Theorem}[section]
 
 \newtheorem{definition}[theorem]{Definition}
 
 \newtheorem{lemma}[theorem]{Lemma}

 \newtheorem{cor}[theorem]{Corollary}
 \newtheorem{pro}[theorem]{Proposition}

 \theoremstyle{definition}
  \newtheorem{remark}[theorem]{Remark}
 \newcommand{\R}{\mathbb{R}}
  
  \newcommand{\N}{\mathbb{N}}
    \newcommand{\Z}{\mathbb{Z}}
  \newcommand{\cK}{\mathcal K}

    \newcommand{\cF}{\mathcal F}
    
     \newcommand{\cS}{\mathcal S}
     
     \newcommand{\cX}{\mathcal X}
     \newcommand{\cY}{\mathcal Y}

		\newcommand{\cZ}{\mathcal Z}
		
	\newcommand{\cD}{\mathcal D}

	\newcommand{\la}{\langle}
	\newcommand{\ra}{\rangle}
 \newcommand{\Bp}{\dot{B}^{s_p}_{p,\infty}}
 \newcommand{\norm}[1]{\lVert #1 \rVert}
  \renewcommand{\:}{\colon}
\newcommand{\bnorm}[1]{\left\lVert #1 \right\rVert}
\newcommand{\upto}{\uparrow}
\newcommand{\dto}{\downarrow}

\newcommand{\wto}{\rightharpoonup} %weak convergence
\newcommand{\wstar}{\overset{\ast}{\rightharpoonup}}
\newcommand{\into}{\hookrightarrow}

 \newcommand{\loc}{{\rm loc}}

  \newcommand{\p}{\partial}
 \let\div\relax
 \DeclareMathOperator{\div}{div}
  \DeclareMathOperator{\curl}{curl}
  
 \let\tilde\relas
 \newcommand{\tilde}[1]{\widetilde{#1}}
 \newcommand{\bP}{\mathbb{P}}
\DeclareMathOperator*{\esssup}{ess\,sup}
\renewcommand{\mathring}[1]{\overset{\circ}{#1}}

\newcommand{\I}{{\rm I}}
\newcommand{\II}{{\rm II}}
\newcommand{\III}{{\rm III}}
\newcommand{\IV}{{\rm IV}}
\newcommand{\BMO}{{\rm BMO}}
\newcommand{\VMO}{{\rm VMO}}

\begin{document}

\title{Global weak Besov solutions of the Navier-Stokes equations and applications}

\author{Dallas Albritton\footnote{School of Mathematics, University of Minnesota, Minneapolis, MN. \\Email address: \texttt{albri050@umn.edu}}, Tobias Barker\footnote{OxPDE, Mathematical Institute, University of Oxford, Oxford, UK. \\Email address: \texttt{tobiasbarker5@gmail.com}}}

\date{\today}
\maketitle

\begin{abstract}
	We introduce a notion of global weak solution to the Navier-Stokes equations in three dimensions with initial values in the critical homogeneous Besov spaces $\dot B^{-1+\frac{3}{p}}_{p,\infty}$, $p > 3$. These solutions satisfy a certain stability property with respect to the weak-$\ast$ convergence of initial conditions. To illustrate this property, we provide applications to blow-up criteria, minimal blow-up initial data, and forward self-similar solutions. Our proof relies on a new splitting result in homogeneous Besov spaces that may be of independent interest.
\end{abstract}

\tableofcontents

\setcounter{equation}{0}
\numberwithin{equation}{section}
\numberwithin{figure}{section}

\section{Introduction}

In this paper, we investigate certain classes of global-in-time weak solutions of the incompressible Navier-Stokes equations in three dimensions:
\begin{equation}\label{NSEintro}\tag{NSE}
	\left.
	\begin{aligned}
	\partial_{t}v- \Delta v+v\cdot\nabla v &=-\nabla q \\
	\quad \div v &= 0
\end{aligned}
\right\rbrace \text{ in } \R^3 \times \R_+.
\end{equation}
In the recent paper~\cite{sereginsverakweaksols}, G. Seregin and V. {\v S}ver{\'a}k introduced a notion of \emph{global weak $L^3$ solution} to the Navier-Stokes equations which enjoys the following property. Given a sequence of global weak $L^3$ solutions with initial data $u_0^{(n)} \rightharpoonup u_0$ in $L^3$, there exists a subsequence converging in the sense of distributions to a global weak $L^3$ solution with initial data~$u_0$. This property, known as  \emph{weak-$\ast$ stability}, plays a distinguished role in the regularity theory of the Navier-Stokes equations. For example, such sequences of solutions arise naturally when zooming in on a potential singularity of the Navier-Stokes equations, as in the papers~\cite{sereginl3,sereginh1/2}\footnote{In these papers, Seregin also investigated weak-$\ast$ stability in the context of local Leray solutions, which were discovered by Lemari{\'e}-Rieusset~\cite{lemarie2002}.} by Seregin and~\cite{escauriazasereginsverak} by Escauriaza, Seregin, and {\v S}ver{\'a}k.

The main idea in \cite{sereginsverakweaksols} is to decompose a solution of the Navier-Stokes equations as
\begin{equation}
	v = V + u,
	\label{vdecompintro}
\end{equation}
where $V$ is the linear evolution of the initial data $u_0 \in L^3$,
\begin{equation}
	V(x,t) := \int_{\R^3} \Gamma(x-y,t) u_0(y) \, dy,
	\label{Vdefintro}
\end{equation}
and $u$ is a perturbation belonging to the global energy space
\begin{equation}
	u \in L^\infty_t L^2_x \cap L^2_t \dot H^1_x(Q_T) \text{ for all } T > 0.
	\label{}
\end{equation}
Here, $\Gamma \: \R^3 \times \R_+ \to \R$ denotes the heat kernel in three dimensions, and 
\begin{equation}
	Q_T := \R^3 \times ]0,T[, \quad 0 < T \leq \infty,
\label{}
\end{equation} denotes a parabolic cylinder. It is reasonable to expect that solutions of the form $v = V + u$ enjoy weak-$\ast$ stability, since the linear evolution $V$ is continuous in many nice topologies with respect to weak convergence of initial data, while the correction term $u$ is ``merely a perturbation''. %The analogue of~\eqref{vdecompintro} in the half-space also appears in the paper~\cite{barkersereginhalfspace} of G. Seregin and the second author. %We will illustrate the details below.
%\textbf{Possible mention of decomposition?}

In the paper~\cite{barkersereginsverakstability}, Barker \emph{et al.} created a notion of \emph{global weak $L^{3,\infty}$ solution} that contains the solutions in~\cite{sereginsverakweaksols} as well as the scale-invariant solutions investigated by Jia and {\v S}ver{\'a}k in~\cite{jiasverakselfsim} as a special case. These solutions exhibit some interesting phenomena. For instance, global weak $L^{3,\infty}$ solutions exist even when a local-in-time mild solution is not known to exist\footnote{We mention that for divergence-free initial data in $L^{3,\infty}$ , there exists an associated global-in-time Lemari{\'e}-Rieusset local energy solution of the Navier-Stokes equations~\cite{lemarie2002}.} (unlike in the $L^3$ case). It appears that such solutions may be non-unique even from the initial time, see the examples of the forward self-similar solutions computed by Guillod and {\v S}ver{\'a}k in~\cite{guillodsverak}. On the other hand, weak-$\ast$ stability continues to hold in spite of the conjectured non-uniqueness. The authors of~\cite{barkersereginsverakstability} also showed that global weak $L^{3,\infty}$ solutions provide a natural class in which to investigate \emph{minimal blow-up initial data}. {\v S}ver{\'a}k also mentioned the possibility of investigating the \emph{radius of smoothness} (resp. \emph{uniqueness}) associated to each initial data $u_0 \in L^{3,\infty}$. This is the maximal time such that each global weak $L^3$ solution with prescribed initial data $u_0$ is smooth (resp. unique).

Recently, the second author proposed in the paper~\cite{barkerB4} to investigate notions of solution in critical spaces $X$ that generalize the solutions described above. Namely, one desires a notion of \emph{global $X$ solution} that satisfies a weak-$\ast$ stability property when, for example,
\begin{equation}
	X = \dot H^{\frac{1}{2}}, L^3, L^{3,\infty}, \dot B^{-1+\frac{3}{p}}_{p,\infty}, \BMO^{-1}, \dot B^{-1}_{\infty,\infty}, \quad 3 < p < \infty.
	\label{}
\end{equation}
The second author established the existence of \emph{global $\dot B^{-\frac{1}{4}}_{4,\infty}$ solutions} with the decomposition $v = V + u$ utilized in previous works. Moreover, he proved that under natural hypotheses, $\dot B^{-\frac{1}{4}}_{4,\infty}$ is the largest critical space in which such a decomposition is viable. Therefore, a notion of global $X$ solution for the critical homogeneous Besov spaces $X = \dot B^{-1+\frac{3}{p}}_{p,\infty}$ with $4 < p < \infty$ must be based on a new structure.\footnote{While critical spaces are not strictly necessary for weak-$\ast$ stability (see p. 5 of the second author's paper \cite{barkerweakstrong}, for example), they are convenient for the applications we have in mind.}

\textit{In this paper, we develop a notion of \emph{global weak Besov solution} of the Navier-Stokes equations associated to initial data in the critical homogeneous Besov spaces $\dot B^{-1+\frac{3}{p}}_{p,\infty}(\R^3)$, $3 < p < \infty$.}

%The precise definition of global weak Besov solution is Definition~\ref{weaksol}.

%We postpone the definition of global weak Besov solution until after we state our main theorems.

In Section~\ref{sec:weakbesovsols}, we prove the following results.
Let $3 < q \leq p < \infty$, and $0 < T \leq \infty$. We include forcing terms of the form $\div F$ with $F \in \cF_q(Q_T)$, defined as the space of locally integrable functions $F \: Q_T \to \R^{3\times3}$ such that
	\begin{equation}
	\norm{F}_{\cF_q(Q_T)} := \sup_{t \in ]0,T[} t^{1-\frac{3}{2q}} \norm{F(\cdot,t)}_{L^q(\R^3)} < \infty.
		\label{Fqdefintro}
	\end{equation}

	%Our notion of solution includes the previous works \cite{sereginsverakweaksols,barkersereginsverakstability,barkerB4} as a special case.

\begin{theorem}[Existence]
	\label{existenceintro}
	Let $u_0 \in \dot B^{-1+\frac{3}{p}}_{p,\infty}(\R^3)$ be a divergence-free vector field and $F \in \cF_q(Q_\infty)$. There exists a global weak Besov solution $v$ with initial data $u_0$ and forcing term $\div F$.
\end{theorem}
%In the following discussion, $p \in ]3,\infty[$ and $q \in ]3,p]$. %$T \in ]0,\infty]$, 
\begin{theorem}[Weak--$\ast$ stability]
	\label{thm:stabilityintro}
	Suppose that $(v^{(n)})_{n \in \N}$ is a sequence of global weak Besov solutions with initial data $u_0^{(n)}$ and forcing terms $\div F^{(n)}$, respectively. Furthermore, suppose that 
	\begin{equation}
		u_0^{(n)} \wstar u_0 \text{ in } \dot B^{-1+\frac{3}{p}}_{p,\infty}(\R^3), \quad F^{(n)} \wstar F \text{ in } \cF_q(Q_\infty).
		\label{}
	\end{equation} Then there exists a subsequence converging strongly in $L^3_\loc(Q_\infty)$ to a global weak Besov solution $v$ with initial data $u_0$ and forcing term $\div F$.%\footnote{The convergence occurs in other sense as well, see Proposition \ref{pro:stability}.}
\end{theorem}

\begin{theorem}[Weak-strong uniqueness]
	\label{weakstrongintro}
	There exists a constant $\varepsilon_0 := \varepsilon_0(p,q) > 0$ such that for all $u_0 \in \dot B^{-1+\frac{3}{p}}_{p,\infty}(\R^3)$ divergence-free and $F \in \cF_q(Q_T)$ satisfying
	\begin{equation}
		\norm{u_0}_{\dot B^{-1+\frac{3}{p}}_{p,\infty}(\R^3)} + \norm{F}_{\cF_q(Q_T)} \leq \varepsilon_0,
	\end{equation}
	there exists a unique weak Besov solution on $Q_T$ with initial data $u_0$ and forcing term $\div F$. This solution belongs to $L^\infty_\loc(\R^3 \times \R_+)$.
\end{theorem}

 % The Section~\ref{sec:applications} for the detailed proofs.
%demonstrate how global weak Besov solutions provide a unified framework in which to answer 

\emph{The second half of this paper is dedicated to applications of global weak Besov solutions.} Namely, we provide applications to certain critical problems concerning blow-up criteria, minimal blow-up initial data, and forward self-similar solutions. We present these results at the end of the introduction. The reader interested only in applications is invited to skip to~Section~\ref{sec:applicationsintro}.

To motivate our notion of solution, it is instructive to write the perturbed Navier-Stokes system satisfied by the correction term in the decomposition $v = V + u$ used in the previous works~\cite{sereginsverakweaksols,barkersereginsverakstability,barkerB4}:
	\begin{equation}
		\left.
		\begin{aligned}
		\p_t u - \Delta u + (u+V) \cdot \nabla u + u \cdot \nabla V &= - \nabla q - \div V \otimes V \\
		\div u &= 0
	\end{aligned}
	\right\rbrace \text{ in } \R^3 \times \R_+
		\label{equationforu}
	\end{equation}
	with zero initial condition. The associated global energy inequality is
	\begin{equation}\label{globalL3energyinequalityintro}
 \|u(\cdot,t)\|_{L^{2}}^2+2\int\limits_0^t\int\limits_{\mathbb{R}^3} |\nabla u|^2\,dx\, dt' \leq 2\int\limits_{0}^t\int\limits_{\mathbb{R}^3} (V\otimes u+V\otimes V):\nabla u \,dx\, dt'.
 \end{equation}
 In order for the RHS of \eqref{globalL3energyinequalityintro} to make sense, we require that
 \begin{equation}
	 V \in L^4_{t,\loc} L^4_x(\R^3 \times \R_+).
	 \label{Vrequirementintro}
 \end{equation}
As demonstrated by the second author in~\cite{barkerB4}, the quantitative scale-invariant version of \eqref{Vrequirementintro} is
$\norm{u_0}_{\dot B^{-\frac{1}{4}}_{4,\infty}} \leq M$, due to the caloric characterization of Besov spaces. Roughly speaking, the forcing term should belong to an $L^2$-based space, whereas $V \otimes V$ may only belong to spaces with integrability $\geq \frac{p}{2}$ for initial data $u_0 \in \dot B^{-1+\frac{3}{p}}_{p,\infty}$. When $p \gg 1$, the obstacle is sometimes interpreted as ``slow decay at spatial infinity.''

	The notion of global weak Besov solution developed in this paper is based on the decomposition
	\begin{equation}
	 v = P_k(u_0) + u,
	 \label{besovdecompintro}
 \end{equation}
 where $P_k(u_0)$ is the \emph{$k$th Picard iterate}, $k \geq 0$, defined by
 \begin{equation}
	P_0(u_0)(\cdot,t) := S(t)u_0,
 \end{equation}
 \begin{equation}
	 P_{k+1}(u_0)(\cdot,t) := S(t)u_0 - B(P_{k},P_{k}), \quad k \geq 0,
	 \label{}
 \end{equation}
 and $B$ is the bilinear term in the integral formulation of the Navier-Stokes equations (see \eqref{StPdivdef} for the precise definition):
		\begin{equation}\label{Bdefintro}
			B(u,v)(\cdot,t):= \int_0^t S(t-s) \bP \div u \otimes v(\cdot,s) \, ds \, dx.
 \end{equation}
  %To simplify notation, we will often omit the dependence on $u_0$.
 The papers \cite{sereginsverakweaksols,barkersereginsverakstability,barkerB4} utilized the decomposition \eqref{besovdecompintro} with $k=0$. Observe that if $v$ solves \eqref{NSEintro}, then $u = v-P_{k}$ solves
\begin{equation}\label{vminuspicard1}
	\left.
	\begin{aligned}
\partial_{t} u-\Delta u+P_{k} \cdot\nabla u+u\cdot\nabla P_{k}+u\cdot\nabla u&=-\nabla p-\textrm{div}\,F_{k} \\
\div u &= 0
\end{aligned}
\right\rbrace \text{ in } \R^3 \times \R_+
\end{equation}
with initial condition $u(\cdot,0) = 0$, where the forcing term $F_k(u_0)$, $k \geq 0$, is defined by
\begin{equation}\label{Fkdefinition}
F_{k}(u_0) := P_{k} \otimes P_{k} -P_{k-1} \otimes P_{k-1},
\end{equation}
and we use the convention that $P_{-1}(u_0) = 0$. One expects the correction $u$ to belong to the energy class if $F_k$ belongs to $L^2(\R^3 \times ]0,T[)$.
	
	Here is our key observation:

%Our key observation is that for $u_0\in \dot{B}^{-1+\frac{3}{p}}_{p,\infty}(\mathbb{R}^3)$ with $p \in ]3,\infty[$, the forcing term $F_k$ belongs to $L^2(\R^3 \times ]0,T[)$ for all $k \gg 1$ depending on $p$.
	\begin{lemma}[Finite energy forcing]
		\label{finiteenergyforcingintro}
Let $p \in ]3,\infty[$ and $u_0 \in \dot B^{-1+\frac{3}{p}}_{p,\infty}(\R^3)$ be a divergence-free vector field with $\norm{u_0}_{\dot B^{-1+\frac{3}{p}}_{p,\infty}(\R^3)} \leq M$. Then for all integers $k \geq k(p) := \lceil \frac{p}{2} \rceil - 2$, the forcing term $F_k(u_0)$ satisfies
	\begin{equation}
	\norm{F_k(u_0)}_{L^2(\R^3 \times ]0,T[)} \leq T^{\frac{1}{4}} C(k,M,p).
		\label{}
	\end{equation}
\end{lemma}

The proof of Lemma~\ref{finiteenergyforcingintro} is based on a self-improvement property of the bilinear term $B$. Heuristically, if a vector field $V$ belongs to an $L^p$-based space, then $B(V,V)$ belongs to an $L^{\frac{p}{2}}$-based space (as well as the original space). For instance, let $u_0 \in \dot B^{-\frac{1}{2}}_{6,\infty}(\R^3)$. Then $V := Su_0$ belongs to an $L^6$-based space, and $F_1(u_0)$ satisfies
\begin{equation}
	F_1 = -B(V,V) \otimes V - V \otimes B(V,V) + B(V,V) \otimes B(V,V).
	\label{}
\end{equation}
Since $B(V,V)$ belongs to an $L^3$-based space (and an $L^6$-based space), an application of H{\"o}lder's inequality implies that $F_1$ belongs to an $L^2$-based space. The same reasoning applies \emph{mutatis mutandis} with the inclusion of a forcing term $\div F$ with $F$ belonging to $\cF_q(Q_T)$ with $q \in ]3,p]$.\footnote{See~\eqref{Fqdefintro} or \eqref{katoforcing1}-\eqref{katoforcing2} for the relevant definition.}
%~\eqref{cFQTdef}. 
The self-improvement property of $B$ was already exploited in the papers \cite{gallagherasymptotics,gallagherkochplanchonbesov}. The phenomenon that $F_1$ is a higher order term is already present in the Picard iterates for the ODE $\dot x = ax^2$, $x(0) = x_0$, where $a, x_0 \in \R$.  % and mentioned in the book \cite{bahourichemindanchin}.

 %A key part of our conception of global $\dot{B}^{-1+\frac{3}{p}}_{p,\infty}(\mathbb{R}^3)$ solutions is the use of Picard iterates.
% Before further discussion, we define the Picard iterates $P_{k}(u_0,F)$ ($k=-10,1,\ldots$)  by 
 %\begin{equation}\label{Picarditeratredefintro-10}
 %P_{-1}(u_0,F):=0,\,\,\,\,\, P_{0}(u_0,F):= S(t)u_0+L(F)(\cdot,t)
 %\end{equation}
 %and
 %\begin{equation}\label{Picarditeratesk}
 %P_{k}(u_0,F)(\cdot,t):= P_{0}(u_0,F)(\cdot,t)+ B(P_{k-1}(u_0,F), P_{k-1}(u_0,F))(\cdot,t)
 %\end{equation}
 %for $l=1,\ldots$. In the above,
 %\begin{equation}\label{Ldefintro}
 %L(F)(\cdot,t):=\int\limits_{0}^{t} S(t-s)\mathbb{P}\textrm{div}\,F\,\,ds
 %\end{equation}
 %and
 %with $F$ being matrix valued, and $u$ and $v$ being vector valued., see Chapter 11 of Lemari{\'e}'s \emph{Recent Advances}.

%with $\div u = 0$, $u(\cdot,0) = 0$, and
%\begin{equation}\label{vminuspicard2}
%\end{equation}

Here is our main definition:

%$$F\in \mathcal{F}_{p}(Q_{T}):=\{(f_{ij})_{1\leq i,j\leq 3}: f_{ij}\in \mathcal{S}^{'}(\mathbb{R}^3 \times]0,T[)\,\,\,\textrm{and}\,\,\, \max_{1\leq i,j\leq 3}\sup_{0<t<T} t^{1-\frac{3}{2p}} \|f_{ij}(\cdot,t)\|_{L_{p}(\mathbb{R}^3)}<\infty\},$$
  
\begin{definition}[Weak Besov solution]\label{weaksol}
	Let $T > 0$, $u_0 \in \BMO^{-1}(\R^3)$ be a divergence-free vector field, and $F \in \cF(Q_T)$.\footnote{The requirements $u_0 \in \BMO^{-1}(\R^3)$ and $F \in \cF(Q_T)$ ensure that the Picard iterates $P_k(u_0,F)$ are well-defined. We refer the reader to~\eqref{BMO-1def}, \eqref{cFQTdef}, and~\eqref{Picarditeratedef1}-\eqref{Picarditeratedef2} for the respective definitions.}

We say that a distributional vector field $v$ on $Q_T:=\mathbb{R}^3\times ]0,T[$ is a \emph{weak Besov solution} to the Navier-Stokes equations on $Q_T$ with initial data $u_0$ and forcing term $\div F$ if there exists an integer $k \geq 0$ such that the following requirements are satisfied.
	
	First, there exists a pressure $q \in L^{\frac{3}{2}}_\loc(Q_T)$ such that $v$ satisfies the Navier-Stokes equations on $Q_T$ in the sense of distributions:
	\begin{equation}
		\p_t v - \Delta v + v \cdot \nabla v  = - \nabla q + \div F, \qquad \div v = 0.
	\end{equation}
Second, $v$ may be decomposed as
	\begin{equation}
		v = u + P_k(u_0,F),
	\end{equation}
	where $u \in L^\infty_t L^2_x \cap L^2_t \dot H^1_x(Q_T)$ and $u(\cdot,t)$ is weakly $L^2$-continuous on $[0,T]$. Furthermore, we require that $\lim_{t \dto 0} \norm{u(\cdot,t)}_{L^2(\R^3)} = 0$ and $F_\ell \in L^2(Q_T)$ for all integers $\ell \geq k$. Finally, we require that $(v,q)$ satisfies the following local energy inequality for every $t \in ]0,T[$ and all non-negative test functions $\varphi \in C^\infty_0(Q_\infty)$:
		\begin{equation}\label{vlocalenergyineq}
		\int_{\R^3} \varphi(x,t) |v(x,t)|^2 \, dx + 2 \int_0^t \int_{\R^3} \varphi |\nabla v|^2 \, dx \, dt'  $$ $$ \leq \int_0^t \int_{\R^3} |v|^2 (\p_t \varphi + \Delta \varphi) + v (|v|^2+2q) \cdot \nabla \varphi - 2F : \nabla (\varphi v) \, dx \, dt'.
	\end{equation}

We say that the weak Besov solution is \emph{based on the $k$th Picard iterate} if the above properties are satisfied for a given integer $k \geq 0$.

We say that $v$ is a \emph{global weak Besov solution} (or weak Besov solution on $Q_\infty := \R^3 \times \R_+$) with initial data $u_0$ and forcing term $\div F$ if there exists an integer $k \geq 0$ such that for all $T>0$, $v$ is a weak Besov solution on $Q_T$ based on the $k$th Picard iterate with initial data $u_0$ and forcing term $\div F$.
\end{definition}

Let us explain the requirement that $F_\ell \in L^2(Q_T)$ for all $\ell \geq k$. Its role is to ensure that $v$ is also a weak Besov solution based on the $\ell$th Picard iterate for all $\ell \geq k$, see Proposition~\ref{pro:orderofpicarditerate}. In other words, one may always raise the order of the Picard iterate. Similarly, one may lower the Picard iterate depending on the regularity of the initial data. %ensures that $v$ is a weak Besov solution based on the $k(p)$th Picard iterate, $k(p) := \lceil \frac{p}{2} \rceil - 2$, provided that $u_0 \in \dot B^{-1+\frac{3}{p}}_{p,\infty}(\R^3)$ with $p \in ]3,\infty[$ and $F \in \cF_q(Q_T)$ with $q \in ]3,p]$.
Hence, our notion of weak Besov solution in Definition~\ref{weaksol} is not overly sensitive to the order of the Picard iterate.

%\begin{remark}[Order of the Picard iterate]

%\end{remark}

Before turning to applications, we present another key ingredient in our arguments: a decay property for the correction term near the initial time.
%Theorems~\ref{existenceintro}--\ref{weakstrongintro} %for global weak Besov solutions
%Theorem~\ref{thm:stabilityintro}

\begin{pro}[Decay property]\label{improveddecaypropintro}
	Let $p \in ]3,\infty[$, $q \in ]3,p]$, and $k \geq k(p) := \lceil \frac{p}{2} \rceil - 2$. Assume that $v$ is a global weak Besov solution based on the $k$th Picard iterate with initial data $u_0$ and forcing term $\div F$. Let $\norm{u_0}_{\dot B^{-1+\frac{3}{p}}_{p,\infty}(\R^3)} + \norm{F}_{\cF_q(Q_\infty)} \leq M$. Then
\begin{equation}\label{improveddecayeqintro}
	\norm{u(\cdot,t)}_{L^2(\R^3)} \leq C(k,M,p,q) t^{\frac{1}{4}}.
	\end{equation}
	%(Recall that $\gamma_1$,$\gamma_2$,$\delta_1$,$\delta_2$ were the constants depending on $p$ in Lemma \ref{splitlem}.)
\end{pro}
%Proposition~\ref{improveddecaypropintro}

The proof is given in Section~\ref{uniformdecayestimates}. Notably,~\eqref{improveddecayeqintro} is used to obtain the global energy inequality for the correction term starting from the initial time, see Corollary~\ref{cor:uglobalenergyineqinitialtime}. A similar issue is already present in Seregin's paper~\cite{sereginlerayhopfweaklyconverging} for sequences of weak Leray-Hopf solutions with initial data converging to zero weakly in $L^2$.%, see Seregin's paper.
 
To illustrate the main issue in Proposition~\ref{improveddecaypropintro}, we consider the special case $u_0 \in L^{3,\infty}$ with zero forcing term and decompose the solution as $v = V + u$ as in the paper~\cite{barkersereginsverakstability}. That is, $k=0$ in Definition~\ref{weaksol}.
Heuristically, the energy of the correction term should originate entirely in
 the forcing term $- \div V \otimes V$. Let $\norm{u_0}_{L^{3,\infty}} \leq M$. % in \eqref{equationforu}--\eqref{globalL3energyinequalityintro}. 
Since
	\begin{equation}
		\int_{\R^3} \int_0^t |V \otimes V|^2 \, dx \,dt \leq C t^{\frac{1}{2}} M^2 \text{ for all } t > 0,
\label{VotimesVest}
\end{equation}
we expect the following \emph{a priori} estimate:
 \begin{equation}\label{globalL3aprioriintro}
 \|u(\cdot,t)\|_{L^{2}}^2+2\int\limits_0^t\int\limits_{\mathbb{R}^3} |\nabla u|^2 \, dx \, dt'\leq t^{\frac 1 2} C(M).
 \end{equation}
When $u_0 \in L^3$, the proof of \eqref{globalL3aprioriintro} is via a Gronwall-type argument that does not extend to the more general case. For instance, consider the following estimate for the lower order term in the global energy inequality:
 \begin{equation}
	 \left\lvert \int_0^T \int_{\R^3} V \otimes u : \nabla u \, dx \,dt' \right\rvert \leq C \norm{V}_{L^\infty_t L^{3,\infty}_x(Q_T)} \norm{u}_{L^2_t L^{6,2}_x(Q_T)} \norm{\nabla u}_{L^2(Q_T)},
	 \label{globalL3inftydemonstration}
 \end{equation}
 where the quantity $\norm{V}_{L^\infty_t L^{3,\infty}_x(Q_T)}$ is not ``locally small'' unless $M \ll 1$.

 In the paper~\cite{barkersereginsverakstability}, this issue is overcome using splitting arguments\footnote{Related splitting arguments have also previously been used by Jia and \v{S}ver\'{a}k in~\cite{jiasverakminimal}, in
order to show estimates near the initial time for Lemari\'e-Rieusset local energy solutions of the Navier-Stokes equations with $L^{3}$ initial data.} inspired by C. P. Calder\'{o}n~\cite{calderon90} and a decomposition of initial data in Lorentz spaces. In the present work, we require the following new decomposition of initial data in Besov spaces.\footnote{This lemma was obtained by the second author in \cite[Proposition 1.5]{barkerB4}, which will not be submitted for journal publication.} %\footnote{Lemma \ref{splitlem} was already proven in~\cite[Proposition 1.5]{barkerB4}. Similar but less sophisticated splittings were utilized in~\cite{barkerweakstrong,albrittonblowupcriteria}.} %We also prove a decomposition result for the forcing term in~Lemma~\ref{splitforcing}.

\begin{lemma}[Splitting of initial data]\label{splitlem}
Let $p \in ]d,\infty[$. There exist $p_2 \in ]p,\infty[$, $\delta_{2} \in ]0,-s_{p_2}[$, $\gamma_1, \gamma_2 > 0$, and $C > 0$, each depending only on $p$, such that for each divergence-free vector field $g \in \Bp(\R^d)$ and $N>0$, there exist divergence-free vector fields
$\bar{g}^{N}\in \dot{B}^{s_{p_2}+\delta_{2}}_{p_2,p_2}(\mathbb{R}^d)\cap \dot{B}^{s_{p}}_{p,\infty}(\mathbb{R}^d)$ and $\widetilde{g}^{N}\in L^2(\mathbb{R}^d)\cap \dot{B}^{s_{p}}_{p,\infty}(\mathbb{R}^d) $ with the following properties:
\begin{equation}\label{Vdecomp1}
g= \widetilde{g}^{N} + \bar{g}^{N},
\end{equation}
\begin{equation}\label{barV_2est}
\|\widetilde{g}^{N}\|_{L^2(\R^d)}\leq C N^{-\gamma_{2}} \norm{g}_{\Bp(\R^d)},
\end{equation} 
\begin{equation}\label{barV_1est}
\|\bar{g}^{N}\|_{\dot{B}^{s_{p_2}+\delta_{2}}_{p_2,p_2}(\R^d)}\leq C N^{\gamma_{1}} \norm{g}_{\Bp(\R^d)}.
\end{equation}
Furthermore, 
\begin{equation}\label{barV_1est.1}
\|\widetilde{g}^{N}\|_{\dot{B}^{s_p}_{p,\infty}(\R^d)}, \|\bar{g}^{N}\|_{\dot{B}^{s_p}_{p,\infty}(\R^d)}\leq C \norm{g}_{\Bp(\R^d)}.
\end{equation}
%\begin{equation}\label{barV_2est.1}
%\leq C(p) M
%\end{equation} 
\end{lemma}

The most notable feature of Lemma~\ref{splitlem} is that the summability index $q$ is reduced from $\infty$ in both terms of the decomposition. Therefore, the splitting is not a simple ``diagonal splitting" that could be obtained via complex interpolation of Besov spaces. Moreover, it does not appear to obviously follow from the abstract real interpolation theory, since Besov spaces are not real interpolation spaces of Besov spaces except in special cases (such as when the integrability index $p$ is kept constant), see \cite[Section 4]{krepkogorskii} for an example. Further discussion and the proof of a general splitting result, Proposition~\ref{besovsplittingpro}, are contained in Section~\ref{sec:appendixsplitting}. %Our strategy is similar, see Section~\ref{uniformdecayestimates}.

\setcounter{figure}{\value{theorem}}
%%%%%%%%%%%%%%%%%%%%%%%%%%%%%%%%%%%%%%%%%%%%%%%%%%%%%%%%%%%%%%%%%%%%%%%%%%%%
	\begin{figure}[h!]
		\centering
		\includegraphics[width=3.25in]{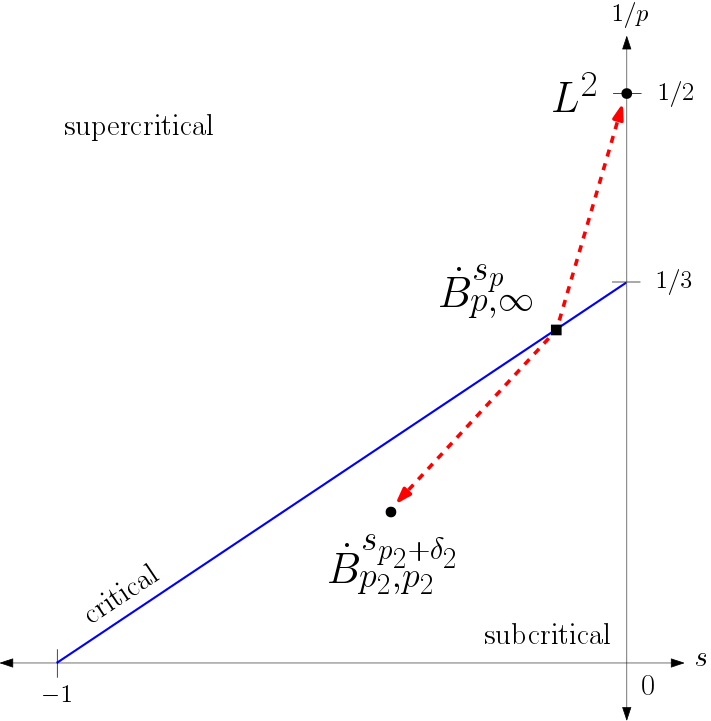}
	\caption{Illustration of Lemma~\ref{splitlem} with $d=3$. The initial data $g \in \dot B^{s_p}_{p,\infty}$ is split along the dashed red lines into $\tilde{g} \in L^2$ and $\bar{g} \in \dot B^{s_{p_2}+\delta_2}_{p_2,p_2}$.}
\label{fig:splittingfigure}
\end{figure}
\addtocounter{theorem}{1}

\subsection{Applications}\label{sec:applicationsintro}
The second part of the paper is focused on applications of global weak Besov solutions %, especially the weak-$\ast$ stability result Theorem \ref{thm:stabilityintro1},
to three problems concerning the three-dimensional Navier-Stokes equations.
\subsubsection{Blow-up criteria}

%In general, Seregin's strategy \cite{sereginl3} may be adapted to show that Leray-Hopf solutions satisfy $\lim_{t \uparrow T^*} \norm{v(\cdot,t)}_{\cX} = \infty$ in other critical spaces $\cX$ when $T^*$ is a finite blow-up time. The key ingredient is to prove \textbf{weak-$\ast$ stability} for a sequence of solutions obtained from magnifying a potential singularity. When $\cX = \dot B^{-1+\frac{3}{p}}_{p,\infty}(\R^3)$, global weak Besov solutions provide a natural framework for the ``zooming in'' procedure. %because they satisfy \textbf{weak-$\ast$ stability} and agree with the mild solutions and Leray-Hopf solutions under general assumptions.

Our first application is an improved blow-up criterion for the Navier-Stokes equations in critical spaces:

\begin{cor}[$\dot B^{-1+\frac{3}{p}}_{p,\infty}$ blow-up criterion]\label{blowupBesovintro}
Let $T^* > 0$, $u_0 \in L^\infty(\R^3)$ be a divergence-free vector field, and $F \in L^\infty_t L^q_x(\R^3 \times ]0,T^*[)$ for some $q \in ]3,\infty[$. Suppose that $v \in L^\infty(\R^3 \times ]0,T[)$ is a mild
solution of the Navier-Stokes equations on $\R^3 \times ]0,T[$ with initial data $u_0$ and forcing term $\div F$ for
all $T \in ]0,T^*[$. Suppose that there exists a
	sequence of times $t_n \upto T^*$ such that
	\begin{equation}\label{controlonsequenceintro}
		\sup_{n \in \N} \norm{v(\cdot,t_n)}_{\dot B^{-1+\frac{3}{p}}_{p,\infty}(\R^3)} < \infty
	\end{equation}
for some $p \in ]3,\infty[$.
	Finally, assume that $v(\cdot,T^*)$ satisfies
	\begin{equation}
		\sqrt{T^* - t_n} v(\sqrt{T^* - t_n} (\cdot - x^*),T^*) \wstar 0 \text{ in } \cD'(\R^3),
		\label{blowupassumptionintro}
	\end{equation}
for some $x^* \in \R^3$. Then $v$ is regular at $(x^*,T^*)$.\footnote{We say that $v$ is \emph{regular} at $(x^*,T^*) \in \R^{3+1}$ if $v \in L^\infty(B(x^*,R) \times ]T^*-R^2,T^*[)$ for some $R > 0$. Otherwise, it is \emph{singular} at $(x^*,T^*)$.} If~\eqref{blowupassumptionintro} is verified for all $x^* \in \R^3$, then $v \in L^\infty(\R^3 \times ]0,T^*[)$.
\end{cor}

Corollary~\ref{blowupBesovintro} is a special case of Theorem~\ref{blowupBesov}, which may be regarded as a quantitative version of the corollary. For comparison, the following weaker criterion was obtained without forcing term by the first author in~\cite{albrittonblowupcriteria}:
\begin{equation}
\lim_{t \upto T^*} \norm{v(\cdot,t)}_{\dot B^{-1+\frac{3}{p}}_{p,q}(\R^3)} = \infty \text{ with } p,q \in ]3,\infty[.
\label{}
\end{equation}
We also mention the preceding works \cite{escauriazasereginsverak,sereginh1/2,sereginl3,phuclorentz,barkersereginhalfspace,kenigkoch,gallagherkochplanchonl3,gallagherkochplanchonbesov} in this direction.\footnote{Corollary~\ref{blowupBesovintro} without forcing term appeared in the recent preprint \cite{barkerB4} of the second author which is not intended to be submitted for publication.} Specifically, our methods are based on the rescaling procedure and backward uniqueness arguments introduced by Escauriaza, Seregin, and {\v S}ver{\'a}k in~\cite{escauriazasereginsverak} and further developed by Seregin in~\cite{sereginh1/2,sereginl3}.

The requirement \eqref{blowupassumptionintro} states that the blow-up profile $v(\cdot,T^*)$ vanishes in the limit of the rescaling procedure. This assumption excludes, for example, the situation that $v(\cdot,T^*)$ is scale-invariant, in which case zooming on the singularity would provide no new information. If $v(\cdot,T^*)$ belongs to the closure of Schwartz functions in $\dot B^{-1}_{\infty,\infty}(\R^3)$, then \eqref{blowupassumptionintro} is automatically satisfied. See Section~\ref{sec:blowupBesov} for further remarks.
	 
	 The reason for the restriction $q > 3$ on the forcing term is to ensure that the maximal time of existence is indicated by the formation of a singularity. Note, for instance, that solutions of the equation $\Delta u = \div F$ in $\R^d$ with $F$ belonging to $L^d(\R^d)$
	 %belong to $\BMO(\R^d)$ but
	 may not be locally bounded when $d \geq 2$.
	 
Finally, let us mention that the ``concentration+rigidity'' roadmap of Kenig and Merle~\cite{kenigmerleradial} was utilised by Koch and Kenig in~\cite{kenigkoch}, and subsequently by Koch, Gallagher and Planchon in \cite{gallagherkochplanchonl3}-\cite{gallagherkochplanchonbesov} to show the following. Namely, if the maximal time of existence $T_{\rm max}(u_0)$ is finite, then
\begin{equation}\label{limsupinfinity}
\limsup_{t\uparrow T_{\rm max}(u_0)} \norm{v(\cdot,t)}_{X}=\infty
\end{equation}
for $X= \dot{H}^{\frac{1}{2}}(\mathbb{R}^3)$ \cite{kenigkoch}, $L^{3}(\mathbb{R}^3)$\cite{gallagherkochplanchonl3}\footnote{The result for $L_{3}(\mathbb{R}^3)$ and $\dot{H}^{\frac{1}{2}}(\mathbb{R}^3)$ was first obtained in  \cite{escauriazasereginsverak}.}, and $\dot{B}^{-1+\frac{3}{p}}_{p,q}(\mathbb{R}^3)$ ($3<p,q<\infty$) \cite{gallagherkochplanchonbesov}. The approach of the aforementioned papers relies on profile decompositions of sequences of elements bounded in the above spaces $X$. In \cite[Remark 3.1]{bahouricohenkoch}, it is conjectured that profile decompositions fail for the continuous embedding $\dot{B}^{-1+\frac{3}{p}}_{p,\infty}(\mathbb{R}^3)\hookrightarrow \dot{B}^{-1+\frac{3}{p'}}_{p',\infty}(\mathbb{R}^3)$ ($3<p\leq p^{'}$). Thus, it seems challenging to use the approach in \cite{kenigkoch} and \cite{gallagherkochplanchonbesov}-\cite{gallagherkochplanchonl3} to obtain Corollary~\ref{blowupBesovintro}.

	 \subsubsection{Minimal blow-up problems}\label{minimalblowupintro}

%Our second application involves showing the existence of minimal blow-up data in $\dot{B}^{-1+\frac{3}{p}}_{p,\infty}(\mathbb{R}^3)$ for the Navier-Stokes equations under the assumption that singularities do occur.

Our second application is to minimal blow-up questions in the context of global weak Besov solutions. The existence of minimal blow-up initial data was first proven by Rusin and {\v S}ver{\'a}k in~\cite{rusinsverakminimal} in the class of mild solutions with initial data belonging to $\dot{H}^{\frac{1}{2}}$ (provided that such solutions may form singularities in finite time). Analogous results were established for~$L^{3}$ by Jia and {\v S}ver{\'a}k in~\cite{jiasverakminimal} and for~$\dot{B}^{-1+\frac{3}{p}}_{p,q}$, $p, q \in ]3,\infty[$ by Gallagher, G.~Koch, and Planchon in~\cite{gallagherkochplanchonbesov}.

	While local-in-time mild solutions are not known to exist for each solenoidal initial data in $L^{3,\infty}$ or $\dot B^{-1+\frac{3}{p}}_{p,\infty}$, the minimal blow-up data problem in these spaces may be reformulated for certain classes of weak solutions. This was originally observed by the second author, Seregin, and {\v S}ver{\'ak} in~\cite{barkersereginsverakstability} for global weak $L^{3,\infty}$ solutions. We now formulate the problem for global weak Besov solutions:
%First, let us give a definition of critical space suitable for our purposes:
		\begin{definition}[Critical space]\label{criticalspacedef}
		Let $(\cX, \norm{\cdot}_{\cX})$ be a Banach space consisting of divergence-free distributional vector fields on $\R^3$. We say that $\cX$ is a \emph{critical space} provided that
		\begin{enumerate}[(i)]
			\item $\cX$ is continuously embedded in $\dot B^{-1}_{\infty,\infty}(\R^3)$,
			\item $\cX$ and $\norm{\cdot}_{\cX}$ are invariant under spatial translation and the scaling symmetry of the Navier-Stokes equations, and
			\item $\bar{B}^{\cX} := \{ u_0 \in \cX : \norm{u_0}_{\cX} \leq 1 \}$ is sequentially compact in the topology of distributions.
	\end{enumerate}
\end{definition}

%\footnote{In this case,  }
%As demonstrated in (cite Barker Seregin Sverak 2017) for $\mathcal{X}(\mathbb{R}^3)=L^{3,\infty}(\mathbb{R}^3)$, global $\mathcal{X}$ solutions provide a natural framework for showing existence of minimal blow-up data in $\mathcal{X}$. Let us now explain this in more detail, along with the concept of `minimal blow-up initial data'  first introduced in \cite{rusinsverakminimal}.
Let $\cX$ be a critical space which is embedded into $\dot B^{-1+\frac{3}{p}}_{p,\infty}(\R^3)$ for some $p \in ]3,\infty[$. By Theorem~\ref{weakstrongintro}, there exists $\rho>0$ satisfying
	 \begin{itemize}
\item[] \textit{(small data implies smooth)} $u_0 \in \cX$ and $\|u_0\|_{\mathcal{X}}<\rho$ implies that any global weak Besov solution with initial data $u_0$ has no singular points
\end{itemize}
Then the following quantity is well defined:
\begin{itemize}
	\item[] $\rho_{\cX} := \sup(\lbrace \rho > 0 : $ for all $u_0 \in \cX$ satisfying $\norm{u_0}_{\cX} \leq \rho$, any global weak Besov solution with initial condition $u_0$ has no singular points$\rbrace )$.
\end{itemize}
Under the assumption that $\rho_{\cX} < \infty$, one may ask whether the above supremum is attained: Does there exist a global weak Besov solution $\tilde{v}$ with initial data $\tilde{u}_0 \in \cX$ such that $\tilde{v}$ has a singular point and $\|\tilde{u}_0\|_{\mathcal{X}}=\rho_{\cX}$? Such $\widetilde{u_0}$ is referred to as \emph{minimal blow-up initial data}. We answer this question in the affirmative below:

\begin{cor}[Minimal blow-up data]\label{cor:minimalblowupdataintro}
Let $\cX$ be a critical space continuously embedded into $\dot B^{-1+\frac{3}{p}}_{p,\infty}(\R^3)$ for some $p \in ]3,\infty[$. Suppose that $\rho_{\cX} < \infty$. Then there exists a solenoidal vector field $\widetilde{u_0} \in \cX$ with $\norm{\widetilde{u_0}}_{\cX} = \rho_{\cX}$ such that $\widetilde{u_0}$ is initial data for a singular global weak Besov solution. The set of such $\widetilde{u_0}$ is sequentially compact (modulo spatial translation and the scaling symmetry of the Navier-Stokes equations) in the topology of distributions.
\end{cor}

Our more general result is Theorem \ref{thm:minimalblowupperturbations}, which treats the problem of minimal blow-up perturbations of global solutions, thus generalizing the work~\cite{rusinminimalperturbations} of Rusin
%and {\u S}ver{\'a}k 
for $\dot H^{\frac{1}{2}}$ initial data. On the other hand, Corollary~\ref{cor:minimalblowupdataintro} already contains the previously known results for $\cX = \dot H^{\frac{1}{2}}, L^3$, and $\dot B^{-1+\frac{3}{p}}_{p,q}$ with $p, q \in ]3,\infty[$ due to weak-strong uniqueness for global weak Besov solutions. While Corollary~\ref{cor:minimalblowupdataintro} only asserts the sequential compactness in the topology of distributions, we may upgrade to convergence in norm if the critical space is uniformly convex, as in the examples above. For $\cX = \dot B^{-1+\frac{3}{p}}_{p,\infty}$, compactness is in the weak-$\ast$ topology. A minor point is that our approach also accounts for any possible changes in the set of minimal blow-up initial data under renormings of the critical space.

	\subsubsection{Self-similar solutions}\label{selfsimintro}

%Our final application is the construction of forward self-similar solutions associated to arbitrary divergence-free initial data belonging to $\dot B^{-1+\frac{3}{p}}_{p,\infty}(\R^3)$.
Our final application concerns forward-in-time self-similar solutions of the Navier-Stokes equations. A locally integrable vector field $v \: \R^3 \times \R_+ \to \R^3$ is \emph{discretely self-similar} with scaling factor $\lambda > 1$ ($\lambda$-DSS) provided that
	\begin{equation}\label{eq:ssintro}
	v(x,t) = \lambda v(\lambda x,\lambda^2 t) \text{ for a.e. } (x,t) \in \R^3 \times \R_+.
	\end{equation}
	The vector field is \emph{self-similar} (scale-invariant) if the relation \eqref{eq:ssintro} is verified for all $\lambda > 0$. We consider also the analogous definition for vector fields on~$\R^3$ and for distributional vector fields.

Although self-similar solutions of the Navier-Stokes equations have a rich history going back to Leray~\cite{leray}\footnote{In~\cite{leray}, Leray posed the question of whether backward self-similar solutions of the Navier-Stokes equations exist. These were subsequently ruled out in~\cite{necas} and~\cite{tsai}.}, the existence of large-data forward self-similar solutions was settled only recently by Jia and {\v S}ver{\'a}k in~\cite{jiasverakselfsim}. These solutions have important consequences for the potential non-uniqueness of weak Leray-Hopf solutions, as investigated in~\cite{jiasverakillposed,guillodsverak}. While the solutions in~\cite{jiasverakselfsim} correspond to scale-invariant data in $C^\alpha_\loc(\R^3 \setminus \{0\})$, there is also an abundant literature on the existence of (discretely) self-similar solutions evolving from rough initial data~\cite{tsaidiscretely,korobkovtsai,bradshawtsaiII,lemarie2016,bradshawtsairot,bradshawtsaibesov,wolfchael2loc}. In particular, we are interested in the paper \cite{bradshawtsaibesov} of Bradshaw and Tsai, which established the existence of discretely self-similar solutions associated to initial data $u_0 \in \dot B^{-1+\frac{3}{p}}_{p,\infty}(\R^3)$, $p \in ]3,6[$. Our final application is the following extension of their work:

	\begin{theorem}[Existence of (discretely) self-similar solutions]\label{ssexistintro}
	Suppose $u_0 \in \dot B^{-1+\frac{3}{p}}_{p,\infty}(\R^3)$ is a divergence-free vector field for some $p \in ]3,\infty[$. If $u_0$ is $\lambda$--DSS for some scaling factor $\lambda > 1$, then there exists a $\lambda$--DSS global weak Besov solution with initial data $u_0$. If $u_0$ is scale-invariant, then there exists a scale-invariant global weak Besov solution with initial data $u_0$.
\end{theorem}

To prove Theorem~\ref{ssexistintro}, we approximate $u_0 \in \dot B^{-1+\frac{3}{p}}_{p,\infty}(\R^3)$ by a sequence of (discretely) self-similar initial data belonging to the Lorentz space~$L^{3,\infty}(\R^3)$. The proof is completed by applying the weak-$\ast$ stability property to an associated sequence of (discretely) self-similar solutions whose existence was established in~\cite{bradshawtsaiII}.

%\subsection{Weak Besov Solutions}

\section{Preliminaries}\label{sec:preliminaries}

\subsection{Function spaces}\label{sec:functionspaces}

Let $d, m \in \N$. We begin by recalling the definition of the \emph{homogeneous Besov spaces} $\dot B^s_{p,q}(\R^d;\R^m)$. Our treatment follows \cite[Chapter 2]{bahourichemindanchin}. There exists a non-negative radial function $\varphi \in C^\infty(\R^d)$ supported on the annulus $\{ \xi \in \R^d : 3/4 \leq |\xi| \leq 8/3 \}$ such that
\begin{equation}
	\sum_{j \in \Z} \varphi(2^{-j} \xi) = 1, \quad \xi \in \R^3 \setminus \{0\}.
	\label{}
\end{equation}
The homogeneous Littlewood-Paley projectors $\dot \Delta_j$ are defined by
\begin{equation}
	\dot \Delta_j f = \varphi(2^{-j} D) f, \quad j \in \Z,
	\label{}
\end{equation}
for all tempered distributions $f$ on $\R^d$ with values in $\R^m$. The notation $\varphi(2^{-j}D) f$ denotes convolution with the inverse Fourier transform of $\varphi(2^{-j}\cdot)$ with $f$.

 Let $p,q \in [1,\infty]$ and $s \in ]-\infty,d/p[$.\footnote{The choice $s=d/p$, $q=1$ is also valid.} The homogeneous Besov space $\dot B^s_{p,q}(\R^d;\R^m)$ consists of all tempered distributions $f$ on $\R^d$ with values in $\R^m$ satisfying
	\begin{equation}
		\norm{f}_{\dot B^s_{p,q}(\R^d;\R^m)} := \bnorm{ \big( 2^{js} \norm{ \dot \Delta_j f}_{L^p} \big)_{j \in \Z} }_{\ell^q} < \infty
		\label{}
	\end{equation}
	and such that $\sum_{j \in \Z} \dot \Delta_j f$ converges to $f$ in the sense of tempered distributions on $\R^d$ with values in $\R^m$. In this range of indices, $\dot B^{s}_{p,q}(\R^d;\R^m)$ is a Banach space. When $s \geq 3/p$ and $q > 1$, the spaces must be considered \emph{modulo polynomials}, see Section~\ref{sec:appendixsplitting}. Note that other reasonable choices of the function $\varphi$ defining $\dot \Delta_j$ lead to equivalent norms. In general, Besov spaces may also be characterized as real interpolation spaces of Bessel potential spaces, see~\cite{berghlofstrom,lemarie2002}. For now, we only consider $d=m=3$.

	%To accomodate initial data in critical spaces, we define $\cB_p$ space of divergence-free vector fields on $\R^3$ with components in $\dot B^{s_p}_{p,\infty}(\R^3)$. Recall the critical exponent $s_p := -1+3/p$. 

We now recall a particularly useful property of Besov spaces, i.e., their characterization in terms of the heat kernel. For all $s \in ]-\infty,0[$, there exists a constant $c := c(s) > 0$ such that for all tempered distributions $f$ on $\R^3$,
		\begin{equation}
			c^{-1} \sup_{t > 0} t^{-\frac{s}{2}} \norm{S(t) f}_{L^p(\R^3)} \leq \norm{f}_{\dot B^s_{p,\infty}(\R^3)} \leq c \sup_{t > 0} t^{-\frac{s}{2}} \norm{S(t) f}_{L^p(\R^3)},
			\label{besovequivalentnorm}
		\end{equation}
		where we use the notation 
		\begin{equation}
			(Sf)(\cdot,t) = S(t) f = \Gamma(\cdot,t) \ast f, \quad t > 0,
			\label{}
		\end{equation}
		and $\Gamma \: \R^3 \times \R_+ \to \R$ is the heat kernel in three dimensions.
		This motivates the definition of the \emph{Kato spaces} $\cK^s_p(Q_T)$ with parameters $s \in \R$, $p \in [1,\infty]$, and $0 < T \leq \infty$. The Kato spaces consist of all locally integrable vector fields $u \: Q_T \to \R^3$ satisfying
			\begin{equation}
			\norm{u}_{\cK^s_p(Q_T)} := \esssup_{t \in ]0,T[} t^{-\frac{s}{2}} \norm{u(\cdot,t)}_{L^p(\R^3)} < \infty.
				\label{}
			\end{equation}
We abbreviate 
\begin{equation}
	\cK_p(Q_T) := \cK^{s_p}_p(Q_T), \quad s_p := -1+3/p.
\label{}
\end{equation} Therefore, for all $p \in ]3,\infty]$, there exists a constant $c := c(p) > 0$ such that
			\begin{equation}
				c^{-1} \norm{Su_0}_{\cK_p(Q_\infty)} \leq \norm{u_0}_{\dot B^{s_p}_{p,\infty}(\R^3)} \leq c \norm{Su_0}_{\cK_p(Q_\infty)},
				\label{Bpequivalentnorm}
			\end{equation}
			for all vector fields $u_0 \in \dot B^{s_p}_{p,\infty}(\R^3)$.  As demonstrated in~\cite{cannone} and~\cite{planchon}, the characterization~\eqref{Bpequivalentnorm} is particularly well suited for constructing mild solutions of the Navier-Stokes equations.

Next, for all $0 < T \leq \infty$, consider the space $\cX_T$ consisting of all locally integrable functions $u$ on $Q_T$ such that
\begin{equation}
\norm{u}_{\cX_T} := \esssup_{t \in ]0,T[} t^{\frac{1}{2}} \norm{u(\cdot,t)}_{L^\infty(\R^3)} + \sup_{x \in \R^3} \sup_{R \in ]0,\sqrt{T}[} R^{-\frac{3}{2}} \norm{u}_{L^2(B(x,R) \times ]0,R^2[)} < \infty.
	\label{}
\end{equation}
This is the largest space on which the bilinear operator $B$ is known to be bounded, see the paper~\cite{kochtataru} of H. Koch and D. Tataru. We use the following Carleson measure 
%\footnote{The second term in the definition of $\norm{\cdot}_{\cX}$ is the Carleson-type norm. The first term is from the heat characterization of $\dot B^{-1}_{\infty,\infty}(\R^3)$.}
characterization of $\norm{\cdot}_{\BMO^{-1}(\R^3)}$. Namely, for all tempered distributions $f$ on $\R^3$, we define
\begin{equation}
	\norm{f}_{\BMO^{-1}(\R^3)} := \norm{Sf}_{\cX_\infty}.
	\label{BMO-1def}
\end{equation}
The space $\BMO^{-1}(\R^3)$ consists of all tempered distributions on $\R^3$ with finite $\BMO^{-1}$ norm.
%since the operator $L(F)(\cdot,t) := \int_0^t S(t-s) \bP \div F(\cdot,s) \, ds$ is well-defined on $\cF(Q_T)$ with values in $\cX_T$, see Section \ref{sec:preliminaries}.

		Let us clarify the relationships between various function spaces of initial data. The Lorentz space $L^{3,\infty}(\R^3)$ is continuously embedded into $\dot B^{s_p}_{p,\infty}(\R^3)$ for all $p \in ]3,\infty]$. This may be proven using \eqref{besovequivalentnorm} and Young's convolution inequality for Lorentz spaces. Next, the Bernstein inequalities for frequency-localized functions imply an analogue of the Sobolev embedding theorem for homogeneous Besov spaces. Finally, regarding $\BMO^{-1}(\R^3)$, H{\"o}lder's inequality gives
		\begin{equation}
			%substack
		\sup_{x \in \R^3} \sup_{R > 0} R^{-\frac{3}{2}} \norm{S(t) f}_{L^2(B(x,R) \times ]0,R^2[)} \leq c_p \sup_{t > 0} t^{-\frac{s_p}{2}} \norm{S(t) f}_{L^p(\R^3)},
			\label{}
		\end{equation}
	when $p \in ]2,\infty[$ and $f$ is tempered distribution on $\R^3$.
	These relationships are summarized below:
			\begin{equation}
			L^{3,\infty}(\R^3) \into \dot B^{s_{p_1}}_{p_1,\infty}(\R^3) \into \dot B^{s_{p_2}}_{p_2,\infty}(\R^3) \into \BMO^{-1}(\R^3) \into \dot B^{-1}_{\infty,\infty}(\R^3), $$ $$ 3 < p_1 < p_2 < \infty.
			\end{equation}
The above continuous embeddings are strict.			
	
	We now present a useful interpolation inequality for Kato spaces. Let $0 < T \leq \infty$ and $u \: Q_T \to \R^3$ be a locally integrable vector field. The interpolation inequality is
\begin{equation}
	\norm{u}_{\cK^s_p(Q_T)} \leq \norm{u}^\theta_{\cK^{s_1}_{p_1}(Q_T)} \norm{u}^{1-\theta}_{\cK^{s_2}_{p_2}(Q_T)},
	\label{Kinterp}
\end{equation}
provided that $s_1, s_2 \in \R$, $p_1,p_2 \in [1,\infty]$, $\theta \in ]0,1[$, and
\begin{equation}
	s := \theta s_1 + (1-\theta) s_2, \quad \frac{1}{p} := \frac{\theta}{p_1} + \frac{1-\theta}{p_2}.
	\label{}
\end{equation}
A common scenario is $u \in L^\infty_t L^2_x(Q_T) \cap \cK_{p_2}(Q_T)$ with $p_2 \in ]4,\infty[$. Observe that $L^\infty_t L^2_x(Q_T) = \cK^0_2(Q_T)$, so \eqref{Kinterp} implies $u \in \cK^s_4(Q_T)$ with $s = -1/2 + 1/(2(p_2-2))$. Hence, we obtain
	\begin{equation}
		\norm{u}_{L^4(Q_T)} \leq c(p_2) T^{\frac{1}{4(p_2-2)}} \norm{u}_{L^\infty_t L^2_x(Q_T)}^{\theta}\norm{u}_{\cK_{p_2}(Q_T)}^{1-\theta}.
%\left(\norm{u}_{L^\infty_t L^2_x(Q_T)} + \norm{u}_{\cK_{p_2}(Q_T)}\right).
		\label{}
	\end{equation}
This embedding fails when $p_2 = \infty$.
			
\subsection{Linear estimates}\label{linearestimatessec}
Our next goal is to present certain estimates for the time-dependent Stokes system in Kato spaces. The Leray projector $\bP$ onto divergence-free vector fields is the Fourier multiplier with matrix-valued symbol $I - \xi \otimes \xi/|\xi|^2$. The operators $\{ S(t) \bP \div \}_{t \geq 0}$ are convolution operators with the gradient of the Oseen kernel, see \cite[Chapter 11]{lemarie2002}. Specifically, there exists a smooth function $G \: \R^3 \to \R^{3\times 3}$ satisfying 
			\begin{equation}
				|\nabla^\ell G(x)| \leq \frac{c(\ell)}{(1+|x|^2)^{\frac{3+\ell}{2}}}, \quad x \in \R^3, \, 0 \leq \ell \in \Z,
				\label{}
			\end{equation}
			\begin{equation}
				(S(t) \bP \div F)_i := \sum_{j,k=1}^3 \frac{1}{t^2} \frac{\p G_{ij}}{\p x_k}\big( \frac{\cdot}{\sqrt{t}} \big) \ast F_{jk}, \quad t > 0, \, 1 \leq i \leq 3,
				\label{StPdivdef}
			\end{equation}
			for matrix-valued functions $F \: \R^3 \to \R^{3 \times 3}$.
			Let us define a space of forcing terms in analogy with the Kato spaces. For all $s \in \R$, $p \in [1,\infty]$, and $0 < T \leq \infty$, the space $\cF^s_p(Q_T)$ consists of all locally integrable matrix-valued functions $F \: Q_T \to \R^{3\times3}$ such that
			\begin{equation}
				\norm{F}_{\cF^s_p(Q_T)} := \esssup_{t \in ]0,T[} t^{-\frac{s}{2}}\norm{F(\cdot,t)}_{L^p(\R^3)} < \infty.
				\label{katoforcing1}
			\end{equation}
	We often abbreviate
			\begin{equation}
				\cF_p(Q_T) := \cF^{s_p'}_p(Q_T), \quad s_p' := -2+3/p.
				\label{katoforcing2}
			\end{equation}
In analogy with $\cX_T$, we also define the space $\cY_T$ consisting of all locally integrable $F \: Q_T \to \R^{3\times3}$ such that
\begin{equation}
\norm{F}_{\cY_T} := \esssup_{t\in]0,T[} t \norm{F(\cdot,t)}_{L^\infty(\R^3)} + \sup_{x \in \R^3} \sup_{R \in ]0,\sqrt{T}[} R^{-3} \norm{F}_{L^1(B(x,R) \times ]0,R^2[)} < \infty.
	\label{}
\end{equation}
Finally, our admissible class of forcing terms in Definition~\ref{weaksol} is
\begin{equation}
	\cF(Q_T) := \cY_T \cup \left[\cup_{q \in ]3,\infty[} \cF_q(Q_T)\right].
	\label{cFQTdef}
\end{equation}

The following estimates were proven by the first author in~\cite[Lemma 4.1]{albrittonblowupcriteria}:
\begin{lemma}[Estimates in Kato spaces]\label{lem:katoest}
	Let $0 < T \leq \infty$, $1 \leq p_1 \leq p_2 \leq \infty$, such that
	\begin{equation}\label{eq:Oseenscaling}
		s_2 - \frac{3}{p_2} = 1 + s_1 - \frac{3}{p_1}.
	\end{equation}
	In addition, assume the conditions
	\begin{equation}\label{eq:Oseenconds}
		s_1 > -2, \qquad \frac{3}{p_1} - \frac{3}{p_2} < 1.
	\end{equation}
	(For instance, if $p_2 = \infty$, then the latter condition is satisfied when $p_1 > 3$. If $p_1 = 2$, then the latter condition is satisfied when $p_2 < 6$.) Then
	\begin{equation}\label{eq:Oseenestimate}
		%\p_t^\alpha \nabla^\beta 
	\bnorm{\int_0^t  S(t-\tau) \bP \div F(\cdot,\tau) \, d\tau}_{\cK^{s_2}_{p_2}(Q_T)} \leq C(s_1,p_1,p_2) \norm{F}_{\cF^{s_1}_{p_1}(Q_T)},
\end{equation}
%\p_t^\alpha \nabla^\beta 
for all distributions $F \in \cF^{s_1}_{p_1}(Q_T)$, and the solution $u$ to the corresponding heat equation belongs to $C(]0,T];L^{p_2}(\R^3))$. Let $k, l \geq 0$ be integers. If we further require that
\begin{equation}
	t^{\alpha+\frac{|\beta|}{2}} \p_t^\alpha \nabla^\beta  F \in \cF^{s_1}_{p_1}(Q_T),
	\label{}
\end{equation}
for all integers $0 \leq \alpha \leq k$ and multi-indices $\beta \in (\N_0)^3$ with $|\beta| \leq l$,
then we have
\begin{equation}
	\bnorm{t^{k+\frac{l}{2}} \p_t^k \nabla^l \int_0^t  S(t-\tau) \bP \div F(\cdot,\tau) \, d\tau}_{\cK^{s_2}_{p_2}(Q_T)} \leq $$ $$
	\leq C(k,l,s_1,p_1,p_2) \Big( \sum_{\alpha=0}^k \sum_{\beta=0}^l \norm{t^{\alpha+\frac{|\beta|}{2}} F}_{\cF^{s_1}_{p_1}(Q_T)} \Big),
\end{equation}
and the spacetime derivatives $\p_t^k \nabla^l u$ of the solution $u$ belong to $C(]0,T];L^{p_2}(\R^3))$.
\end{lemma}

Let us define operators $L$, $B$ by
\begin{equation}
L(F)(\cdot,t) := \int_0^t S(t-s) \bP \div F(\cdot,s) \, ds, \quad B(U,V) := L(U \otimes V),
\label{}
\end{equation}
for certain matrix-valued functions $F$ and vector fields $U,V$ on $Q_T$. Lemma~\ref{lem:katoest} (see also p. 526 of \cite{cannone}, for example) asserts that for all $p \in ]3,\infty[$,
	\begin{equation}
		L \: \cF_p(Q_T) \to \cK_p(Q_T), \quad B \: \cK_p(Q_T) \times \cK_p(Q_T) \to \cK_p(Q_T),
		\label{}
	\end{equation}
boundedly and with operator norms independent of $0 < T \leq \infty$. As was demonstrated in~\cite{kochtataru}, $L$ and $B$ are also bounded as operators $L \: \cY_T \to \cX_T$ and $B \: \cX_T \times \cX_T \to \cX_T$ with norms independent of $T \in ]0,\infty]$. This leads to the following important consequence. If $u_0 \in \dot B^{s_p}_{p,\infty}(\R^3)$ is a divergence-free vector field and $F \in \cK_q(Q_T)$ with $p \in ]3,\infty[$ and $q \in ]3,p]$, then the Picard iterates $P_k(u_0,F)$ are well defined for all $k \geq 0$:
\begin{equation}
	P_0(u_0,F)(\cdot,t) := S(t) u_0 + L(F)(\cdot,t),
	\label{Picarditeratedef1}
\end{equation}
\begin{equation}
	P_{k+1}(u_0,F) := P_0(u_0,F) - B(P_{k},P_k), \quad k \geq 0,
	\label{Picarditeratedef2}
\end{equation}
and $P_{-1}(u_0,F) := 0$. For simplicity, we often omit the dependence on $u_0$ and $F$.
Here, $P_0$ belongs to $\cK_p(Q_T) \cap \cX_T$, and
	\begin{equation}
		\norm{P_0}_{\cK_p(Q_T)} + \norm{P_0}_{\cX_T} \leq C(p,q) (\norm{u_0}_{\dot B^{s_p}_{p,\infty}(\R^3)} + \norm{F}_{\cF_q(Q_T)}).
		\label{}
	\end{equation}
	Supposing now that $\norm{P_0(u_0,F)}_{\cK_p(Q_T)} + \norm{P_0(u_0,F)}_{\cX_T} \leq M$, we obtain
	\begin{equation}
		\norm{P_k}_{\cK_p(Q_T)} + \norm{P_k}_{\cX_T} \leq C(k,M,p,q),
		\label{Picardestimatefrequent}
	\end{equation}
where the constant is a polynomial in $M$ with no constant term and degree depending only on $k$. Lemma~\ref{lem:katoest} also has the following consequence regarding vector fields with finite kinetic energy. Namely, for all $p \in ]3,\infty]$,
\begin{equation}
	\label{bilinearenergyestimate}
	\norm{B(U,V)}_{L^\infty_t L^2_x(Q_T)} \leq c(p) \min( \norm{U}_{L^\infty_t L^2_x(Q_T)} \norm{V}_{\cK_p(Q_T)}, \norm{V}_{L^\infty_t L^2_x(Q_T)} \norm{U}_{\cK_p(Q_T)}).
\end{equation}
This is useful in our construction of strong solutions in Section~\ref{sec:strongsols}.

We now exploit the self-improvement property of the bilinear term described in the introduction to prove a version of Lemma~\ref{finiteenergyforcingintro}. Throughout the paper, we define for $k\geq 0$:
\begin{equation}
F_{k}(u_0,F):= P_{k} \otimes P_{k} -P_{k-1} \otimes P_{k-1}.
\end{equation}
 %Specifically, we prove estimates on the forcing term $F_k(u_0,F)$ arising in the splitting $u = v - P_k(u_0,F)$.
%Let us denote $p(k) := 2k + 4$ when $k \geq 0$ and $$ for all $p > 3$.
%Roughly, $p(k)$ is the highest value of $p \in ]3,\infty[$ such that the decomposition is effective whenever the data satisfies $u_0 \in \cB_p$ and $F \in \cF_p(Q_T)$. Similarly, we may define
		
		%\label{kpdef}
%	Here is the main lemma we will prove in this section:
	\begin{lemma}[Finite energy forcing]\label{forcinglem}
	Let $T>0$, $u_0 \in \BMO^{-1}(\R^3)$ be a divergence-free vector field, and $F \in \cF(Q_T)$. Suppose that
\begin{equation}
\norm{P_0(u_0,F)}_{\cK_{p}(Q_T)} + \norm{P_0(u_0,F)}_{\cX_T} \leq M
\label{}
\end{equation}
for some $p \in ]3,\infty[$. Then for all $k \geq k(p) := \lceil \frac{p}{2} \rceil - 2$, we have
	\begin{equation}
		F_k(u_0,F) \in L^2(Q_T).
	\end{equation}
	\begin{equation}
	P_{k+1}(u_0,F) - P_k(u_0,F) \in C([0,T];L^2(\R^3)) \cap L^2_t \dot H^1_x(Q_T),
		\label{Pkassertion}
	\end{equation}
	In addition, the following bounds are satisfied:
	\begin{equation}\label{Fkbound}
		 \norm{F_k}_{\cF_2(Q_T)} + \norm{P_{k+1} - P_k}_{\cK_2(Q_T)} \leq C(k,M,p),
	\end{equation}
\begin{equation}\label{Fkboundconsequence}
	\norm{F_k}_{L^2(Q_T)} + \norm{P_{k+1} - P_k}_{L^\infty_t L^2_x(Q_T)}  \leq T^{\frac{1}{4}} C(k,M,p).
\end{equation}
\end{lemma}

\begin{proof}[Proof of Lemma \ref{forcinglem}]
	
Define $p(k) := 2k+4$. By interpolation, we may assume without loss of generality that $p = p(k)$ in the statement. That is, $\norm{P_0(u_0,F)}_{\cK_{p(k)}(Q_T)} \leq M$. For all integers $0 \leq \ell \leq k$, we define $q(k,\ell) := \frac{p(k)}{\ell+2}$. We will prove inductively that
	\begin{equation}\label{Fkest}
		\norm{F_\ell}_{\cF_{q(k,\ell)}(Q_T)} + \norm{P_{\ell+1} - P_\ell}_{\cK_{q(k,\ell)}} \leq C(k,M).
	\end{equation}
	H{\"o}lder's inequality immediately implies that $\norm{F_0}_{\cF_{\frac{p(k)}{2}}(Q_T)} \leq M^2$. Since $P_1 - P_0 = -L(F_0)$, the estimate $\norm{P_1 - P_0}_{\cK_{\frac{p(k)}{2}}(Q_T)} \leq C(k) M^2$ follows from Lemma \ref{lem:katoest}. Let us now assume that we have the estimate \eqref{Fkest} for some $0 \leq \ell < k$. We observe the identity
	\begin{equation}
		F_{\ell+1} = P_{\ell+1} \otimes (P_{\ell+1} - P_{\ell}) + (P_{\ell+1} - P_{\ell}) \otimes P_\ell,
		\label{}
	\end{equation}
	so that $\norm{F_{\ell+1}}_{\cF_{q(k,\ell+1)}(Q_T)} \leq C(k,M)$ due to H{\"o}lder's inequality and \eqref{Picardestimatefrequent}. Now recall that $P_{\ell+2} - P_{\ell+1} = -L(F_{\ell+1})$. Lemma \ref{lem:katoest} implies
	\begin{equation}
		\norm{P_{\ell+2} - P_{\ell+1}}_{\cK_{q(k,\ell+1)}(Q_T)} \leq C(k,M).
		\label{}
	\end{equation}
	This completes the induction. It is clear that \eqref{Fkbound} follows from \eqref{Fkest} with $\ell=k$, and \eqref{Fkboundconsequence} is obtained from \eqref{Fkbound} by H{\"o}lder's inequality. Lastly, \eqref{Pkassertion} concerning $P_{k+1} - P_k$ follows from the classical energy estimate for the Stokes equations and that $F_k \in L^2(Q_T)$.
\end{proof}

Finally, we prove a linear estimate concerning solutions of the time-dependent Stokes equations in the space $L^\infty_t (\dot B^{s_p}_{p,\infty})_x(Q_T)$. By Young's convolution inequality, all tempered distributions $f$ on $\R^3$ satisfy
\begin{equation}
\norm{Sf}_{L^\infty_t (\dot B^{s_p}_{p,\infty})_x(Q_\infty)} \leq c \norm{f}_{\dot B^{s_p}_{p,\infty}(\R^3)}
	\label{}
\end{equation}
for a constant $c > 0$ independent of $p \in [1,\infty]$.
	Let us show that for all $q \in [1,\infty[$ and $p \in [q,\infty]$, there exists a constant $c := c(p,q) > 0$ such that for all $0 < T \leq \infty$ and $F \in \cF_q(Q_T)$,
\begin{equation}
\sup_{t \in ]0,T[} \bnorm{ L(F)(\cdot,t) }_{\dot B^{s_p}_{p,\infty}(\R^3)} \leq c(p,q) \norm{F}_{\cF_q(Q_T)}.
	\label{timespacebesovest}
\end{equation}
%To begin, we recall the following fact about the heat evolution of frequency-localized functions, see~\cite[Lemma 2.4]{bahourichemindanchin}. Let $\mathcal{C}$ be an annulus in $\R^3$ and $\lambda > 0$. There exists a constant $c > 0$ depending only on $\mathcal{C}$ such that for all tempered distributions $f$ on $\R^3$ with ${\rm supp} \widehat{f} \subset \lambda \mathcal{C}$,
%\begin{equation}
%\norm{S(t)}_{L^p(\R^3)} \leq Ce^{-c\lambda^2 t} \norm{f}_{L^p(\R^3)}, \quad t > 0, \; p \in [1,\infty].
	%\label{smoothingeffect}
%\end{equation}
After extending $F$ by zero, it suffices to consider $T=\infty$. By Sobolev embedding for Besov spaces, we need only consider the case $p=q$. For $t, \tau > 0$,
\begin{equation}
\label{katoestlinearzerodata}
\tau^{-\frac{s_p}{2}}\norm{S(\tau)L(F)(\cdot,t)}_{L^{p}(\mathbb{R}^3)}\leq \tau^{-\frac{s_p}{2}}\int\limits_{0}^{t+\tau}\norm{S(t+\tau-s)\mathbb{P}\textrm{div}\,F(\cdot,s)}_{L^{p}(\mathbb{R}^3)} ds$$$$
\leq \frac{c(p)\tau^{\frac{-s_p}{2}}}{(t+\tau)^{\frac{-s_p}{2}}}\norm{F}_{\cF_{p}(Q_{\infty})}\leq c(p)\norm{F}_{\cF_{p}(Q_{\infty})}.
\end{equation}
By the heat flow characterisation of Besov norms with negative upper index, we see that~\eqref{katoestlinearzerodata} implies~\eqref{timespacebesovest}.
 %By rescaling \eqref{timespacebesovest}, we only need estimate $\norm{L(F)(\cdot,1)}_{\dot B^{s_p}_{p,\infty}(\R^3)}$.
%Let $j \in \Z$ and $\lambda_j := 2^j$. Then for all $t > 0$, we use \eqref{smoothingeffect} and a change of variables to obtain
%\begin{equation}\label{LFbesovest}
%	\bnorm{\dot \Delta_j \int_0^t S(t-\tau) \div F(\cdot,\tau) \, d\tau}_{L^q(\R^3)} \leq C \lambda_j \int_0^\infty \tau^{\frac{s_q'}{2}} e^{-c\lambda_j^2 \tau} \, d\tau \, \norm{F}_{\cF_q(Q_\infty)} \leq $$ $$ \leq C(q) \lambda_j^{-s_q} \int_0^\infty \tau^{\frac{s_q'}{2}} e^{-\tau} \, d\tau \, \norm{F}_{\cF_q(Q_\infty)},
%\end{equation}
%which is equivalent to \eqref{timespacebesovest}.
We refer the reader to \cite[Lemma 8]{auscher} for the analogous estimate on $L \: \cY_T \to L^\infty_t \BMO^{-1}_x(Q_T)$.

%%%%%%%%%%%%%%%%%%%%%%%%%%%%%%%%%%%%%%%%%%%%%%%%%%%%%%%%%%%

\section{Weak Besov solutions}\label{sec:weakbesovsols}
This section contains the general theory of the weak Besov solutions introduced in Definition~\ref{weaksol}.

\subsection{Basic properties}\label{sec:basicproperties}

First, let us describe the singular integral representation of the pressure used throughout the paper.
\begin{remark}[Associated pressure]\label{pressurermk}
Let $v$ be as in Definition~\ref{weaksol} with pressure $q \in L^{\frac{3}{2}}_\loc(Q_T)$, $T > 0$. There exists a constant function of time $c \in L^1_\loc(]0,T[)$ such that for a.e. $(x,t) \in Q_T$,
	\begin{equation}
		q(x,t) = (-\Delta)^{-1} \div \div v \otimes v - F + c(t).
		\label{pressureequalityc}
	\end{equation}
	Since the Navier-Stokes equations and local energy inequality \eqref{vlocalenergyineq} do not depend on the choice of constant, we will assume in the sequel that $c \equiv 0$. The resulting pressure is known as the \emph{associated pressure}.
\end{remark}

\begin{proof}%[Proof of Remark~\ref{pressurermk}]
For now, assume that $u_0 \in \dot B^{s_p}_{p,\infty}(\R^3)$ is divergence-free and $F \in \cF_q(Q_T)$ for some $p \in ]3,\infty[$ and $q \in ]3,p]$. By the Calder{\'o}n-Zygmund estimates, certainly $\tilde{q} \in L^1_\loc(Q_T)$. If $(v,\tilde{q})$ solves the Navier-Stokes equations with forcing term $\div F$ in the sense of distributions, then $\nabla q = \nabla \tilde{q}$, which implies~\eqref{pressureequalityc}. Therefore, to complete the proof, we need only verify that $(v,\tilde{q})$ is a solution.

\emph{1. Picard iterate}.
Let $\pi_k$ denote the pressure associated to the $k$th Picard iterate,
	\begin{equation}
		\pi_k(u_0,F) := (-\Delta)^{-1} \div \div [P_{k-1} \otimes P_{k-1} - F].
		\label{}
	\end{equation}
	By the Calder{\'o}n-Zygmund estimates, $\pi_k \in L^\infty_{t,\loc} L^p_x(Q_T) + L^\infty_{t,\loc} L^q_x(Q_T)$. Recall that the $k$th Picard iterate is constructed as a solution of the heat equation
	\begin{equation}
		\p_t P_k - \Delta P_k = \bP \div [F-P_{k-1} \otimes P_{k-1}] \text{ in } Q_T.
		\label{Pkheateqnpressure}
	\end{equation}
	Since $-\nabla \pi_k = (I-\bP) \div [F-P_{k-1} \otimes P_{k-1}]$, we may add this term back into~\eqref{Pkheateqnpressure} to obtain the time-dependent Stokes equations with RHS $\div [F - P_{k-1} \otimes P_{k-1}]$. Hence, $\pi_k$ is a valid pressure for the $k$th Picard iterate.
	
	\textit{2. Correction term}. Next, let $p$ denote the pressure associated to the correction term,
	\begin{equation}
		p := (-\Delta)^{-1} \div \div [u \otimes u + P_k \otimes u + u \otimes P_k + F_k].
		\label{}
	\end{equation}
	The Calder{\'o}n-Zygmund estimates imply that $p \in L^{\frac{3}{2}}(Q_T) + L^2_{t,\loc} L^2_x(Q_T)$. Recall that $u \in L^\infty_t L^2_x \cap L^2_t \dot H^1_x(Q_T)$ is a distributional solution of
	\begin{equation}
		\left.
		\begin{aligned}
		\p_t u - \Delta u &= - \nabla \tilde{p} - \div \tilde{F} \\
		\div u &= 0
		\end{aligned}
		\right\rbrace \text{ in } Q_T,
		\label{equationsolvedbytildep}
	\end{equation}
	where $\tilde{F} = u \otimes u + P_k \otimes u + u \otimes P_k + F_k$ for some pressure $\tilde{p}$ in the class of tempered distributions. As in Step~1, if $u$ solves the heat equation with RHS $- \div \bP \tilde{F}$, then $p$ is a valid pressure for $u$, i.e., $(u,p)$ solves \eqref{equationsolvedbytildep}.
	
	The multiplier associated to the Leray projector $\bP$ is not smooth at the origin, so we truncate it. Let $\varphi \in C^\infty_0(\R^3)$ such that $\varphi \equiv 1$ outside $B(2)$ and $\varphi \equiv 0$ inside $B(1)$. Consider the operators $T_\varepsilon := \varphi(D/\varepsilon)$ and $\bP_\varepsilon := \bP T_\varepsilon$ for all $\varepsilon > 0$. Applying $\bP_\varepsilon$ to \eqref{equationsolvedbytildep}, we obtain
	\begin{equation}
		(\p_t - \Delta) (T_\varepsilon u) = - \bP_\varepsilon  \div \tilde{F} \text{ in } Q_T.
		\label{}
	\end{equation}
	In the limit $\varepsilon \dto 0$, we have $T_\varepsilon u \to u$ and $\bP_\varepsilon \div \tilde{F} \to \bP \div \tilde{F}$ in the sense of tempered distributions. Hence, the desired heat equation is satisfied.\footnote{This method is valid under quite mild assumptions on $u$ and $F$. Certainly some assumptions are necessary in order to exclude certain ``parasitic'' solutions $u = c(t)$, $p = -c'(t) \cdot x$, $F = 0$ of the time-dependent Stokes equations. For such solutions, $\widehat{u}$ is supported at the origin in frequency space, so $T_\varepsilon u \to 0$ as $\varepsilon \dto 0$.

A different method is to take the curl of the time-dependent Stokes equations with RHS $\div F$ and initial data $u_0$ and compare it to the curl of the solution of the heat equation with RHS $\bP \div F$ and initial data $u_0$. By well-posedness for the heat equation, the two vorticities are equal, and hence their velocities are equal according to the Biot-Savart law (that is, under mild assumptions).}

	%(see, for instance, Lemari{\'e}-Rieusset monograph \cite[Chapter 14]{lemarie2002}).
	
	\emph{3. Conclusion}. Let $\tilde{q} = \pi_k + p$. Combining Steps~1 and~2 for $\pi_k$ and $p$ gives that $(v,\tilde{q})$ is a distributional solution of the Navier-Stokes equations on $Q_T$ with forcing term $\div F$.

In the general case $u_0 \in \BMO^{-1}(\R^3)$ divergence-free and $F \in \cF(Q_T)$, the above analysis remains valid with the following caveat. Since the $k$th Picard iterate only belongs to the $L^\infty$-based space $\cX_T$, $(-\Delta)^{-1} \div \div P_k \otimes P_k$ may only belong to $L^\infty_{t,\loc} \BMO_x(Q_T)$. Hence, $\pi_k(u_0,F)$ and $q$ may only be well defined up to the addition of a constant function of time $c \in L^1_\loc(]0,T[)$. Even so, we still refer to ``the associated pressure'' with the understanding that our computations will not rely on the particular choice of representative.
\end{proof}

Our next order of business is the following proposition:
\begin{pro}[Energy inequalities for $u$]\label{pro:vlocalglobal}
	Let $T>0$ and $v$ be a weak Besov solution on $Q_T$ as in Definition \ref{weaksol}. Let $p := q - \pi_k$. Then $u$ obeys the \emph{local energy inequality} for
every $t \in ]0,T]$ and all non-negative test functions $0 \leq \varphi \in C^\infty_0(Q_\infty)$:
\begin{equation}\label{ulocalenergyineq}
\int_{\R^3} \varphi(x,t) |u(x,t)|^2 \, dx + 2 \int_0^t \int_{\R^3} \varphi |\nabla u|^2 \, dx \, dt' \leq $$ $$ \leq \int_0^t \int_{\R^3} |u|^2 (\p_t \varphi + \Delta \varphi) + [|u|^2(u+ P_k) + 2pu] \cdot \nabla \varphi + 2 [P_k \otimes u + F_k]: \nabla (\varphi u) \, dx \, dt'.
\end{equation}
Hence, $u$ satisfies the \emph{global energy inequality}
	\begin{equation}\label{uglobalenergyineq}
		\norm{u(\cdot,t_2)}_{L^2(\R^3)}^2 + 2 \int_{t_1}^{t_2} \int_{\R^3} |\nabla u(x,t')|^2 \, dx \, dt' \leq $$ $$ \leq \norm{u(\cdot,t_1)}_{L^2(\R^3)}^2 + 2 \int_{t_1}^{t_2} \int_{\R^3} [(P_k \otimes u) + F_k] : \nabla u \, dx \, dt'
	\end{equation}
for almost every $t_1 \in ]0,T[$ and all $t_2 \in ]t_1,T]$.
\end{pro}
\begin{remark}\label{otherdirectionlocalenergy}
	By adapting the proof of Proposition~\ref{pro:vlocalglobal}, it is also possible to show the following. Let $v$ be as in Definition~\ref{weaksol} \emph{except} that $v$ is not assumed to satisfy the local energy inequality \eqref{vlocalenergyineq}. Instead, assume that $u$ satisfies its corresponding local energy inequality \eqref{ulocalenergyineq}. Then $v$ satisfies \eqref{vlocalenergyineq}. In particular, $v$ is a weak Besov solution on $Q_T$. This fact will be useful in Proposition~\ref{pro:stability}.
\end{remark}

The proof is based on an identity that appears in the classical proof of weak-strong uniqueness and is useful in obtaining an energy inequality for the difference of two solutions of the Navier-Stokes equations. Let $f,g \in C^\infty_0(\overline{Q_\infty})$ and $0 \leq t_1 < t_2 < \infty$. Then
\begin{equation}\label{eq:weakstrongid}
	\int_{\R^3} f(x,t_2)\cdot g(x,t_2) \, dx + 2 \int_{t_1}^{t_2} \int_{\R^3} \nabla f : \nabla g \, dx \, dt - \int_{\R^3} f(x,t_1)\cdot g(x,t_1) \, dx = $$ $$ = \int_{t_1}^{t_2} \int_{\R^3} (\p_t f - \Delta f)\cdot g + f\cdot (\p_t g - \Delta g) \, dx \, dt.
\end{equation}
We also need an analogous identity for the local energy inequality:
\begin{equation}\label{eq:weakstronglocalid}
	\int_{\R^3} \varphi(x,t) f \cdot g(x,t) \, dx + 2 \int_0^t \int_{\R^3} \varphi \nabla f : \nabla g \, dx \, dt' = $$ $$ = \int_0^t \int_{\R^3}  (\p_t \varphi + \Delta \varphi) f\cdot g + \varphi (\p_t f - \Delta f )\cdot g + \varphi f\cdot (\p_tg - \Delta g) \, dx \,dt',
\end{equation}
for all $\varphi \in C^\infty_0(Q_\infty)$ and $t > 0$.
These identities may be obtained for a larger class of functions by approximation.
\begin{proof}
	\textit{1. Local energy inequality for $u$}.
Recall from Definition \ref{weaksol} that $v$ is assumed to satisfy the local energy inequality \eqref{vlocalenergyineq} for every $t \in ]0,T]$ and all $0 \leq \varphi \in C^\infty_0(Q_\infty)$.
Using aforementioned properties of $P_{k}$ and $F$, together with the fact that Riesz transforms are Calder\'{o}n-Zygmund singular integral operators, we see that $P_{k}$, $P_{k-1}\otimes P_{k-1}$, $F$ and $\pi_{k}$ all belong to $L^{2}_{\loc}(Q_{T})$. Using a mollification argument in the same spirit as in \cite[p. 160-161]{sereginnotes}, one can show that $P_k$ satisfies the local energy equality
\begin{equation}\label{Pklocalenergyineq}
	\int_{\R^3} \varphi(x,t) |P_k(x,t)|^2 \, dx + 2 \int_0^t \int_{\R^3} \varphi |\nabla P_k|^2 \, dx \,dt' = $$ $$
	= \int_0^t \int_{\R^3} |P_k|^2 (\p_t \varphi + \Delta \varphi) + 2 \pi_k P_k \cdot \nabla \varphi + 2 [P_{k-1} \otimes P_{k-1} - F] : \nabla (P_k \varphi) \, dx \, dt'.
\end{equation}
Next, one combines the identities $|u|^2 = |v|^2 - 2 v \cdot P_k + |P_k|^2$ and $|\nabla u|^2 = |\nabla v|^2 - 2 \nabla v : \nabla P_k + |\nabla P_k|^2$ with the local energy estimates \eqref{vlocalenergyineq} and \eqref{Pklocalenergyineq} to obtain
\begin{equation}\label{uintermediatelocalenergy}
	\int_{\R^3} \varphi(x,t) |u(x,t)|^2 \, dx + 2 \int_{0}^t \int_{\R^3} \varphi |\nabla u|^2 \, dx \, dt' \leq $$ $$ \leq \int_0^t \int_{\R^3} [|v|^2 + |P_k|^2] (\p_t \varphi + \Delta \varphi) + 2 [vq + P_k \pi_k] \cdot \nabla \varphi \, dx \,dt' + $$ $$ + \int_0^t \int_{\R^3} |v|^2 v \cdot \nabla \varphi + 2 P_{k-1} \otimes P_{k-1} : \nabla (P_k \varphi) - 2 F : \nabla ((v+P_{k})\varphi) \, dx \,dt' + $$ $$
	- 2 \int_{\R^3} \varphi(x,t) v \cdot P_k(x,t) \, dx - 4 \int_{0}^t \int_{\R^3} \varphi \nabla v : \nabla P_k \, dx \, dt'.
\end{equation}
According to the weak-strong identity \eqref{eq:weakstronglocalid}, we may write
\begin{equation}\label{weakstrongidconsequence}
	\int_{\R^3} \varphi(x,t) v \cdot P_k(x,t) \, dx + 2 \int_{0}^t \int_{\R^3} \varphi \nabla v : \nabla P_k \, dx \, dt' = $$ $$ = \int_0^t \int_{\R^3} (\p_t \varphi + \Delta \varphi) v \cdot P_k - \nabla q \cdot P_k \varphi - \nabla \pi_k \cdot \varphi v \, dx \,dt' + $$ $$ + \int_0^t \int_{\R^3} v \otimes v : \nabla (P_k \varphi) + P_{k-1} \otimes P_{k-1} : \nabla (\varphi v) - F : \nabla ((v+P_{k})\varphi ) \, dx \, dt'.
\end{equation}
Substitute \eqref{weakstrongidconsequence} into the final line of \eqref{uintermediatelocalenergy} and collect various terms:
\begin{equation}
	\int_{\R^3} \varphi(x,t) |u(x,t)|^2 \, dx + 2 \int_{0}^t \int_{\R^3} \varphi |\nabla u|^2 \, dx \, dt' \leq $$ $$ \leq \int_0^t \int_{\R^3} |u|^2 (\p_t \varphi+ \Delta \varphi) + 2 pu \cdot \nabla \varphi  + \I \,dx \,dt',
\end{equation}
\begin{equation}
	\I :=  |v|^2 v \cdot \nabla \varphi - 2 v \otimes v : \nabla (P_k \varphi) - 2 P_{k-1} \otimes P_{k-1} : \nabla (\varphi u).
\end{equation}
We now add and subtract $2 P_k \otimes P_k : \nabla (\varphi u) = 2 P_k \otimes (v-u) : \nabla (\varphi u)$ in the expression for $\I$ in order to introduce the forcing term $F_k$:
\begin{equation}
	\I = 2 [ P_k \otimes u + F_k] : \nabla (\varphi u) + \II,
	\label{}
\end{equation}
\begin{equation}\label{IIearlydef}
	\II := |v|^2 v \cdot \nabla \varphi - 2 v \otimes v : \nabla (P_k \varphi) - 2 P_k \otimes v : \nabla (\varphi u).
\end{equation}
Expanding $|v|^2 = |u|^2 + 2 u \cdot P_k + |P_k|^2$ in \eqref{IIearlydef} gives
\begin{equation}
	\II = |u|^2 (u + P_k) \cdot \nabla \varphi + \III,
	\label{}
\end{equation}
\begin{equation}
	\III := [2u \cdot P_k + |P_k|^2] v \cdot \nabla \varphi - 2 v \otimes v : \nabla (P_k \varphi) - 2 P_k \otimes v : \nabla (\varphi u).
	\label{IIIearlydef}
\end{equation}
Clearly,  the first and third terms of $\III$ are integrable. Let us now demonstrate that  $\III$ is integrable by showing that the second term $v \otimes v : \nabla (P_k \varphi)$ is integrable. Recall that\begin{equation}\label{Fspacerecall}
F\in\cF_{q}(Q_{T})\subset L^{\infty}_{t,\textrm{loc}}L^{q}_{x}(Q_{T}).
\end{equation}
Second, from (\ref{Picardestimatefrequent}) we infer that
\begin{equation}\label{Pkspace}
P_{k}, \, P_{k}\otimes P_{k}\in L^{\infty}_{t,\textrm{loc}}L^{p}_{x}(Q_{T}).
\end{equation}
Using (\ref{Picarditeratedef1})-(\ref{Picarditeratedef2}) and (\ref{Fspacerecall})-(\ref{Pkspace}), together with maximal regularity results for the heat equation, we infer that \begin{equation}\label{gradPkspace}
\nabla P_{k}\in \bigcap_{1<s<\infty}(L^{s}_{t,\textrm{loc}}L^{q}_{x}(Q_{T})+L^{s}_{t,\textrm{loc}}L^{p}_{x}(Q_{T})).
\end{equation}
Using that $v\in (L^\infty_t L^2_x)_\loc(Q_{T}),$ $\nabla v\in L^{2}_{\textrm{loc}}(Q_{T})$, and multiplicative inequalities, we see that $v\in L^{\frac{10}{3}}_{\textrm{loc}}(Q_{T})$. From this fact and (\ref{gradPkspace}), we infer that $v \otimes v : \nabla (P_k \varphi)$ is integrable.

It remains to prove that $\III$ integrates to zero. Expanding $v \otimes v = P_k \otimes v + u \otimes v$ in \eqref{IIIearlydef} and rearranging, we obtain
\begin{equation}
	\int_0^t \int_{\R^3} \III \, dx \,dt'  = \int_0^t \int_{\R^3}  |P_k|^2 v \cdot \nabla \varphi - 2 P_k \otimes v : \nabla (P_k \varphi) \, dx \,dt ' + $$ $$ \int_0^t \int_{\R^3} [2u \cdot P_k] v \cdot \nabla \varphi - 2 u \otimes v : \nabla (P_k \varphi) - 2 P_k \otimes v : \nabla (\varphi u) \, dx \,dt'.
\end{equation}
This last expression vanishes, so we have verified that $u$ satisfies the local energy inequality \eqref{ulocalenergyineq}.

\textit{2. Global energy inequality for $u$}. The global energy inequality \eqref{uglobalenergyineq} will follow from the local energy inequality \eqref{ulocalenergyineq} with a special choice of test function. Let $0 \leq \psi \in C^\infty_0(\R^3)$ such that $\psi \equiv 1$ in $B(1)$ and ${\rm supp}(\psi) \subset B(2)$. Fix $0 < t_1 < t_2 \leq T$. For each $\varepsilon, R > 0$, we define Lipschitz functions
\begin{equation}
\eta_{\varepsilon}(t) := \frac{1}{2\varepsilon} \int_{-\infty}^t \chi_{]t_1-\varepsilon,t_1+\varepsilon[}(t') \, dt', \quad t \in \R,
	\label{}
\end{equation}
\begin{equation}
	\Phi_{\varepsilon,R}(x,t) := \eta_\varepsilon(t) \psi(x/R), \quad (x,t) \in \R^{3+1}.
\end{equation}
Technically, $\Phi_{\varepsilon,R}$ is neither smooth nor compactly supported in $Q_T$, but by approximation we may use it as a test function in the local energy inequality \eqref{ulocalenergyineq} when $0 < \varepsilon < \min(t_1,T-t_1)/2$. For $0 < \varepsilon < \min(t_1,T-t_1, t_2-t_1)/2$, this gives
\begin{equation}\label{globalenergywroteout}
	\int_{\R^3} \psi(x/R) |u(x,t_2)|^2 \, dx + 2 \int_0^{t_2} \int_{\R^3} \Phi_{\varepsilon,R} |\nabla u|^2 \, dx \,dt' \leq $$ $$
	\leq \frac{1}{2\varepsilon} \int_{t_1-\varepsilon}^{t_1+\varepsilon} \int_{\R^3}  \psi(x/R) |u|^2 \,dx \,dt' + 2 \int_0^{t_2} \int_{\R^3} \Phi_{\varepsilon,R} [P_k \otimes u + F_k] : \nabla u  \, dx \,dt' + $$ $$
	+ \frac{1}{R^2}\int_0^{t_2} \eta_\varepsilon\int_{\R^3} |u|^2 \Delta \psi(x/R) \, dx \,dt' + $$ $$ + \frac{1}{R} \int_0^{t_2} \eta_\varepsilon \int_{\R^3}  [|u|^2 (u + P_k) + 2pu] \cdot (\nabla \psi)(x/R) \, dx \,dt' + $$ $$
	+ \frac{2}{R} \int_0^{t_2} \eta_\varepsilon \int_{\R^3} [P_k \otimes u + F_k] : u \otimes (\nabla \psi)(x/R) \, dx \,dt'.
\end{equation}
Since $u \in L^\infty_t L^2_x \cap L^2_t \dot H^1_x(Q_T)$, $p \in L^{\frac{3}{2}}(Q_T)+L_{t,\loc}^2 L^{2}_x(Q_{T})$ (see Remark~\ref{pressurermk}), $P_k \in L^\infty(\R^3 \times ]t_1-\varepsilon,T[)$, and $F_k \in L^2(Q_T)$, the last three lines of \eqref{globalenergywroteout} vanish as $R \upto \infty$. Hence, we obtain
	\begin{equation}
		\label{globalenergyrlim}
		\norm{u(\cdot,t_2)}_{L^2(\R^3)}^2 + 2 \int_{0}^{t_2} \eta_\varepsilon \int_{\R^3} |\nabla u|^2 \, dx \, dt' \leq $$ $$ \leq \frac{1}{2\varepsilon} \int_{t_1-\varepsilon}^{t_1+\varepsilon} \norm{u(\cdot,t')}_{L^2(\R^3)}^2 \, dt' + 2 \int_{0}^{t_2} \eta_\varepsilon \int_{\R^3} [P_k \otimes u + F_k] : \nabla u \, dx \,dt'.
	\end{equation}
	Using that $u(\cdot,t)$ is weakly $L^{2}$-continuous on $[0,T]$, we see that (\ref{globalenergyrlim}) in fact holds for all $0<t_{1}<t_{2}\leq T$ and  $0 < \varepsilon < \min(t_1,T-t_1, t_2-t_1)/2$. Recall from the Lebesgue differentiation theorem that for a.e. $t_1 \in ]0,T[$,
	\begin{equation}
		\frac{1}{2\varepsilon} \int_{t_1-\varepsilon}^{t_1+\varepsilon} \norm{u(\cdot,t')}_{L^2(\R^3)}^2 \, dt' \to \norm{u(\cdot,t_1)}_{L^2(\R^3)}^2 \text{ as } \varepsilon \dto 0.
		\label{}
	\end{equation}
Finally, the global energy inequality \eqref{uglobalenergyineq} follows from taking $\varepsilon \dto 0$ in \eqref{globalenergyrlim}. This completes the proof.
\end{proof}

The next proposition asserts that under mild hypotheses, weak Besov solutions are not highly sensitive to the order of the Picard iterate used in the splitting.

\begin{pro}[Raising and lowering]\label{pro:orderofpicarditerate}
	%\label{raiselowerkprop}
Let $0 < T \leq \infty$, $u_0 \in \BMO^{-1}(\R^3)$ be a divergence-free vector field, $F \in \cF(Q_T)$, and $0 \leq k \in \Z$. Suppose that $v$ is a weak Besov solution on $Q_T$ based on the $k$th Picard iterate with initial data $u_0$ and forcing term $\div F$.
\begin{enumerate}[(i)]
	\item Then $v$ is a weak Besov solution based on the $(k+1)$th Picard iterate.
	\item If $k \geq 1$ and $F_{k-1}(u_0,F) \in L^2(Q_S)$ for all finite $0 < S \leq T$, then $v$ is a weak Besov solution based on the $(k-1)$th Picard iterate.
	\item If $u_0 \in \dot B^{s_p}_{p,\infty}(\R^3)$ and $F \in \cF_q(Q_T)$ for some $3 < q \leq p < \infty$, then $v$ is a weak Besov solution based on the $k(p)$th Picard iterate, where $k(p) := \lceil \frac{p}{2} \rceil - 2$.
\end{enumerate}
\end{pro}
\begin{proof}
\textit{Proof of (i)}. We need only consider $T < \infty$. We must show that $\tilde{u} := u + (P_k - P_{k+1})$ belongs to the energy space, $\tilde{u}(\cdot,t)$ is weakly continuous as an $L^2(\R^3)$-valued function on $[0,T]$, and $\lim_{t \dto 0} \norm{\tilde{u}(\cdot,t)}_{L^2(\R^3)} = 0$. These conditions are already satisfied for the correction term $u$, so it remains to show them for $P_{k+1} - P_k$. Recall now that $F_k \in L^2(Q_T)$.\footnote{Definition~\ref{weaksol} requires that $F_{\ell}(u_0,F) \in L^2(Q_T)$ for all $\ell \geq k$.} Since $P_k - P_{k+1} = L(F_k)$,
we obtain $P_k - P_{k+1} \in C([0,T];L^2(\R^3)) \cap L^2_t \dot H^1_x(Q_T)$
and $\lim_{t \dto 0} \norm{P_k - P_{k+1}}_{L^2(\R^3)} = 0$. This completes the proof.

\textit{Proof of (ii)}. Once we further assume that $k \geq 1$ and $F_{k-1} \in L^2(Q_T)$, the proof is nearly identical to the proof of (i).

\textit{Proof of (iii)}. The proof follows from (i)--(ii) combined with the estimates on $F_k(u_0,F)$ proven in Lemma \ref{forcinglem}.
\end{proof}

\subsection{Uniform decay estimate}\label{uniformdecayestimates}

The goal of this section is to prove Proposition~\ref{improveddecaypropintro}, which we restate below:
\begin{pro}[Decay property]\label{improveddecayprop}
Let $T > 0$, $p \in ]3,\infty[$, $q \in ]3,p]$, and $k \geq k(p) := \lceil \frac{p}{2} \rceil - 2$. Assume that $v$ is a weak Besov solution on $Q_T$ based on the $k$th Picard iterate with initial data $u_0$ and forcing term $\div F$. Let $\norm{u_0}_{\dot B^{s_p}_{p,\infty}(\R^3)} + \norm{F}_{\cF_q(Q_T)} \leq M$.
	 Then
\begin{equation}\label{improveddecayeq}
\norm{u(\cdot,t)}_{L^2(\R^3)} \leq C(k,M,p,q) t^\frac{1}{4}.
	\end{equation}
	%(Recall that $\gamma_1$,$\gamma_2$,$\delta_1$,$\delta_2$ were the constants depending on $p$ in Lemma \ref{splitlem}.)
\end{pro}

Heuristically, the global energy inequality starting from the initial time should give a decay rate for $\norm{u(\cdot,t)}_{L^2}$ that depends on the decay rate for $\int_{0}^t \int_{\R^3} |F_k(u_0,F)|^2 \, dx \,dt'$. However, it is not obvious whether the global energy inequality even makes sense starting from the initial time without a decay rate for $\norm{u(\cdot,t)}_{L^2}$.\footnote{The problem is with integrating the lower order term $\int_0^t \int_{\R^3} P_k \otimes u : \nabla u \, dx \, dt'$.} This issue is overcome by decomposing $u$ into two parts, each of which satisfies a global energy inequality with no integrability issues, and estimating $u$ by its parts. The method involves splitting the initial data $u_0$ into a subcritical part $\bar{u_0}$ and a perturbation $\tilde{u_0}$ with finite energy as in Lemma~\ref{splitlem}. The idea is that only subcritical coefficients will enter into the energy inequality for the time-evolution of the perturbation $\tilde{u_0}$. See~\cite{barkersereginsverakstability} for similar arguments in the context of global weak $L^{3,\infty}$ solutions.% and~\cite{barkerweakstrong,barkerB4} for Besov spa

%The proof of Proposition \ref{improveddecayprop} follows closely the work of the second author in~\cite{barkerB4}. 

The hypotheses of Proposition~\ref{improveddecayprop} will be taken as standing assumptions for the remainder of the section.

For $N = 1$, we decompose $u_0$ according to Lemma~\ref{splitlem} and $F$ according to Lemma~\ref{splitforcing},
\begin{equation}
	u_0 = \bar{u_0} + \tilde{u_0}, \qquad
	F = \bar{F} + \tilde{F},
	\label{u_0Fdecomp}
\end{equation}
with $\bar{u_0},\tilde{u_0},\bar{F},\tilde{F}$ satisfying the following properties:
\begin{equation}
	\norm{\bar{u_0}}_{\dot B^{s_{p_2}+\delta_2}_{p_2,p_2}(\R^3)} \leq C(p)M, \quad \norm{\tilde{u_0}}_{L^2(\R^3)} \leq C(p)M, 
	\label{u0Nsplitrevised}
\end{equation}
\begin{equation}
	\norm{\bar{u_0}}_{\dot B^{s_p}_{p,\infty}(\R^3)}, \norm{\tilde{u_0}}_{\dot B^{s_p}_{p,\infty}(\R^3)} \leq C(p)M,
	\label{u_0splitboundBp}
\end{equation}
\begin{equation}
	\norm{\bar{F}}_{\cF^{s_{p_3}^{'}+\delta_3}_{p_3}(Q_T)} \leq C(q,T)M, \quad \norm{\tilde{F}}_{L^3_t L^2_x(Q_T)} \leq C(q,T) M,
	\label{FNsplitrevised}
\end{equation}
\begin{equation}
	\norm{\bar{F}^N}_{\cF_q(Q_T)}, \norm{\tilde{F}^N}_{\cF_q(Q_T)} \leq M.
	\label{FspliboundKp}
\end{equation}
We will use the following notation. For each $k \geq 0$, we define
\begin{equation}\label{barPkN}
	\bar{P_k}(u_0,F) := P_k(\bar{u_0},\bar{F}^N), \quad 	\tilde{P}_k(u_0,F) := P_k(\tilde{u_0},\tilde{F}),
\end{equation}
%\begin{equation}\label{tildePkN}
%
%\end{equation}
\begin{equation}\label{barFkN}
	\bar{F_k}(u_0,F) := \bar{P_k} \otimes \bar{P_k} - \bar{P}_{k-1} \otimes \bar{P}_{k-1},
\end{equation}
\begin{equation}\label{EkNdef}
	E_k(u_0,F) := P_k(u_0,F) - \bar{P_k}(u_0,F),
\end{equation}
\begin{equation}\label{GkNdef}
	G_k(u_0,F) := P_k \otimes P_k - \bar{P}_k \otimes \bar{P}_k.
	\end{equation}
We will frequently suppress dependence on the data $(u_0,F)$ in our notation.

In this section, we will also use the following subcritical estimates for $B$ and $L$, in addition to the properties discussed in Section~\ref{sec:preliminaries}. Namely,
for $\delta>0$,
\begin{equation}\label{katoinfinitysubcritest}
	\norm{B(U,V)}_{\cK^{-1+\delta}_\infty(Q_T)} \leq cT^{\frac{\delta}{2}} \norm{U}_{\cK^{-1+\delta}_\infty(Q_T)} \norm{V}_{\cK^{-1+\delta}_{\infty}(Q_T)},
\end{equation}
\begin{equation}\label{katoinfinitypersistency2}
	\norm{B(U,V)}_{\cK^{-1+\delta}_{\infty}(Q_T)} \leq c(p,\delta) \min(\norm{U}_{\cK^{-1+\delta}_{\infty}(Q_T)} \norm{V}_{\cK_{p}(Q_{T})},
	\norm{V}_{\cK^{-1+\delta}_{\infty}(Q_T)} \norm{U}_{\cK_{p}(Q_{T})}),
\end{equation}
\begin{equation}\label{katoinfinitysubcritestL}
	\norm{L(G)}_{\cK^{-1+\delta_2}_\infty(Q_T)} \leq c(p_2) \norm{G}_{\cF_{p_2}^{s_{p_2}^{'}+\delta_2}(Q_T)},
\end{equation}
which follow from Lemma~\ref{lem:katoest}. Let $\delta:=\min(\delta_2,\delta_3) > 0$. Then
\begin{equation}\label{barPNKsubcritest}
	\norm{\bar{P_{k}}}_{\cK^{-1+\delta}_\infty(Q_T)}\leq C(T,\delta,k) \, Q_k (\norm{\bar{P_{0}}}_{\cK^{-1+\delta}_\infty(Q_T)} )<\infty, %removed (\bar{u}_{0}, \bar{F})
\end{equation}
where $Q_k$ denotes a polynomial with no constant term and degree depending only on $k$.
By the heat characterisation of homogeneous Besov spaces, (\ref{u_0splitboundBp}), and (\ref{FspliboundKp}), we have
\begin{equation}\label{Picardkatoests}
	\norm{\bar{P_{k}}}_{\cK_{p}(Q_T)}+\norm{\bar{P_{k}}}_{\cX_T} \leq C(k,M,p,q). %removed (\bar{u}_{0}, \bar{F})
\end{equation}
The same estimate holds for $P_k$ (see~\eqref{Picardestimatefrequent}). Finally, using~(\ref{Picardkatoests}), along with~(\ref{katoinfinitypersistency2}) and an induction argument, we see that
\begin{equation}\label{barPNKsubcritest2}
\norm{\bar{P_{k}}}_{\cK^{-1+\delta}_\infty(Q_T)} \leq c(\delta,k,M,p,q) \norm{\bar{P_{0}}}_{\cK^{-1+\delta}_\infty(Q_T)}. %removed (\bar{u}_{0}, \bar{F})
\end{equation}
\begin{lemma}\label{ENkest}
In the above notation, for all integers $k \geq 0$, $E_k(u_0,F)$ and $G_k(u_0,F)$ obey the following properties:
\begin{equation}\label{eq:ENkbelongrevised}
	E_k \in C([0,T];L^2(\R^3)) \cap L^2_t \dot{H}^{1}_x(Q_T),
\end{equation}
	\begin{equation}\label{eq:ENkestrevised}
		\norm{E_k}_{L^\infty_t L^2_x(Q_T)} \leq C(k,M,p,q) \norm{\tilde{P}_0}_{L^\infty_t L^2_x(Q_T)},
	\end{equation}
	\begin{equation}\label{GNkspacerevised}
	G_k \in L^2(Q_T).
	\end{equation}
\end{lemma}
\begin{proof}
In the proof, we will make use of the following identities. In particular,
\begin{equation}\label{ENKidentity1}
E_{k}=\tilde{P_{0}} -B(E_{k-1},P_{k-1})-B(\bar{P}_{k-1}, E_{k-1}), %removed (\tilde{u}_{0}, \tilde{F})
\end{equation}
\begin{equation}\label{ENKidentity2}
	E_{k} = \tilde{P}_0-B(E_{k-1},E_{k-1}) -B(\bar{P}_{k-1}, E_{k-1})-B(E_{k-1},\bar{P}_{k-1}).
\end{equation}

\textit{1. Showing $E_{k}$ has finite kinetic energy}.
We proceed by induction. Clearly, $E_{0}=\tilde{P_{0}}$.
This, together with (\ref{u0Nsplitrevised}) and (\ref{FNsplitrevised}), implies that 
\begin{equation}\label{kineticenergyEN0}
E_{0} \in C([0,T]; L^2(\mathbb{R}^3)) \cap L^2_t \dot{H}^{1}_x(Q_T).
\end{equation}
For the inductive step assume $E_{k-1} \in L^\infty_t L^2_x(Q_T)$ ($k\geq 1$). Using \eqref{bilinearenergyestimate}, (\ref{Picardkatoests}), (\ref{ENKidentity1}), (\ref{kineticenergyEN0}), and the inductive assumption, we infer that $E_{k} \in L^\infty_t L^2_x(Q_{T})$ and
\begin{equation}\label{ENKinductiveestimate}
	\|E_{k}\|_{L^\infty_t L^2_x(Q_T)}\leq \|E_{0}\|_{L^\infty_t L^2_x(Q_T)}+ C(p)\|P_{k-1}\|_{\cK_p(Q_T)}\norm{E_{k-1}}_{L^\infty_t L^2_x(Q_T)}+ $$ $$+C(p)\|\bar{P}_{k-1}\|_{\cK_p(Q_T)}\norm{E_{k-1}}_{L^\infty_t L^2_x(Q_T)}\leq $$ $$\leq\|E_{0}\|_{L^\infty_t L^2_x(Q_T)}+ C(k,M,p,q)\norm{E_{k-1}}_{L^\infty_t L^2_x(Q_T)}.
\end{equation}
From (\ref{ENKinductiveestimate}), we can then immediately obtain (\ref{eq:ENkestrevised}).

\textit{2. Showing $G_{k}$ is in $L^{2}(Q_T)$}.
As  previously mentioned, we have
\begin{equation}\label{GNKidentity2}
G_{k} =E_{k}\otimes E_{k}+\bar{P}_{k}\otimes E_{k}+E_{k}\otimes\bar{P}_{k}.
\end{equation}
Using (\ref{barPNKsubcritest}) and Step 1, we see that
\begin{equation}\label{GNKperturbedtermsL2}
\bar{P}_{k}\otimes E_{k}+E_{k}\otimes\bar{P}_{k}\in L^{2}(Q_T).
\end{equation}
Next, we use the interpolation inequality \eqref{Kinterp} together with (\ref{Picardkatoests}) and Step 1 to obtain
\begin{equation}\label{L4nonlin}
E_{k}\otimes E_{k}\in L^{2}(Q_T).
\end{equation}
Combining this with (\ref{GNKperturbedtermsL2}) gives that $G_{k}^{N}\in L^{2}(Q_T)$.
Finally, we note that
\begin{equation}
	E_{k} =S(t)\tilde{u}_{0} + L(\tilde{F})(\cdot,t) - L(G_{k})(\cdot,t).
\end{equation}
This, together with (\ref{kineticenergyEN0}) and $G_{k}\in L^{2}(Q_T)$ implies that
\begin{equation}\label{energyEkN}
E_{k} \in C([0,T]; L^2(\R^3)) \cap L^2_t \dot{H}^{1}_x(Q_T),
\end{equation}
and furthermore, for all $t \in ]0,T]$,
\begin{equation}\label{energyequalityEKN}
\norm{E_{k}(\cdot,t)}_{L^{2}(\mathbb{R}^3)}^2 +2\int\limits_{0}^{t}\int\limits_{\mathbb{R}^3} |\nabla E_{k}|^2 \, dx \, dt' = 2\int\limits_{0}^{t}\int\limits_{\mathbb{R}^3} (G_{k}-\tilde{F}):\nabla E_{k} \, dx \, dt' + \norm{\tilde{u}_{0}}_{L^{2}(\mathbb{R}^3)}^2. %removed (x,t') to keep on one line
\end{equation}
\end{proof}

\begin{remark}\label{ENkrmk}
	Standard energy estimates for Stokes equations imply that 
	\begin{equation}
	\norm{\tilde{P}_0}_{L^\infty_t L^2_x(Q_T)} \leq C\norm{\tilde{u_0}}_{L^2(\R^3)} + CT^{\frac{1}{6}} \norm{\tilde{F}}_{L^3_t L^2_x(Q_T)}.
	\end{equation} We combine the above estimate with \eqref{u0Nsplitrevised}, \eqref{FNsplitrevised}, and Lemma~\ref{ENkest} to obtain
	\begin{equation}
		\norm{E_k}_{L^\infty_t L^2_x(Q_T)}^2 \leq C(k,M,p,q,T), %(N^{-2 \gamma_2} + T^{\frac{1}{3}} N^{2-q}). 
		\label{}
	\end{equation}
	with constant $C>0$ increasing in $T > 0$.
\end{remark}

From now on, we will assume that $v = u + P_k(u_0,F)$ is a weak Besov solution on $Q_T$ as in Definition~\ref{weaksol}. Moreover, we will assume that $k \geq k(p)$ in order that $F_k, \, \bar{F}_k \in \cF_2(Q_T)$ as guaranteed by Lemma~\ref{forcinglem}. We will denote
\begin{equation}
	w_k(u_0,F) := u + E_k(u_0,F) = v - \bar{P_k}(u_0,F).
\end{equation}
%We will typically omit the dependence on the data in our notation.
It is clear from Lemma~\ref{ENkest} that $w_k \in L^\infty_t L^2_x \cap L^2_t \dot H^1_x(Q_T)$ and $w(\cdot,t)$ is weakly continuous as an $L^2(\R^3)$-valued function on $[0,T]$.

\begin{lemma}[Energy inequality for $w_k$]\label{lem:wNkenergyineq}
	In the above notation, we have
	\begin{equation}\label{wenergyineq}
		\norm{w_k(\cdot,t)}_{L^2(\R^3)}^2 + 2 \int_0^t \int_{\R^3} |\nabla w_k|^2 \, dx \, dt' \leq  $$ $$ \leq \norm{\tilde{u}_0}_{L^2(\R^3)}^2 + 2 \int_0^t \int_{\R^3} (\bar{P_k} \otimes w_k + \bar{F_k} - \tilde{F}) : \nabla w_k \, dx \, dt',
	\end{equation}
for all $t \in ]0,T]$.
\end{lemma}
Note that the last integral in \eqref{wenergyineq} is convergent because $\bar{P}_k \in \cK^{-1+\delta}_{\infty}(Q_T)$ belongs to subcritical spaces and $\bar{F}_k, \tilde{F} \in L^2(Q_T)$. Here, $\delta := \min(\delta_2,\delta_3) > 0$.

We omit the proof of Lemma~\ref{lem:wNkenergyineq}, as it is similar to the proof of Proposition~\ref{pro:vlocalglobal}. The main idea is to ``transfer" the global energy inequality from $u$ and $E_k$ to $w_k$ by using the weak-strong identity \eqref{eq:weakstrongid} and the cancellation properties of the nonlinearity.

\begin{proof}[Proof of Proposition \ref{improveddecayprop}]
Below, we use the convention that the constants $C > 0$ depend only on $k,M,p,q$. By a scaling argument, it suffices to obtain an estimate of the form
\begin{equation}
\label{improveddecayesttime1}
	\norm{u(\cdot,1)}_{L^2(\R^3)} \leq C
\end{equation}
when $T \geq 1$. Since one may truncate the interval of existence, $T=1$, without loss of generality.

%For each $N > 0$, we s
Split $u_0 \in \dot B^{s_p}_{p,\infty}(\R^3)$ and $F \in \cF_q(Q_{T})$ as in the beginning of Section~\ref{uniformdecayestimates}. %Note that the splitting of the forcing term is done on the time interval $]0,T_\sharp[$ rather than $]0,T[$.
Using the identity $u = w_k - E_k$, we obtain the inequality
	\begin{equation}\label{usplitest}
		\norm{u(\cdot,t)}_{L^2(\R^3)}^2 \leq
		2\norm{E_k(\cdot,t)}_{L^2(\R^3)}^2 +
		2\norm{w_k(\cdot,t)}_{L^2(\R^3)}^2
	\end{equation}
for all $t \in ]0,T]$.
	Recall from Remark \ref{ENkrmk} that $E^N_k$ obeys the estimate
	\begin{equation}\label{ENkfineest}
		\norm{E_k(\cdot,t)}_{L^2(\R^3)}^2 \leq C. %( N^{-2\gamma_2} + N^{2-q} ).
	\end{equation}

It remains to estimate the energy of $w_k$. Denote $y(t) := \norm{w_k(\cdot,t)}_{L^2(\R^3)}^2$. By manipulating the energy inequality \eqref{wenergyineq} for $w_k$, one obtains
\begin{equation}\label{intermediateenergyest}
	y(t) + \int_0^t \int_{\R^3} |\nabla w_k|^2 \,dx \,d\tau \leq $$ $$ \leq C \norm{\tilde{u}_0}_{L^2(\R^3)}^2 + C \int_0^t \int_{\R^3} |\tilde{F}|^2 + |\bar{P_k} \otimes w_k|^2 + |\bar{F_k}|^2 \, dx \, d\tau
	\end{equation}
	for all $t \in ]0,T]$. Let us now analyze each of the terms.
		To begin, recall that $\norm{\tilde{u}_0}_{L^2(\R^3)}^2 \leq C$. %N^{-2\gamma_2}$.
	As a consequence of Lemma~\ref{forcinglem}, (\ref{u_0splitboundBp}), and (\ref{FspliboundKp}), we have that $\norm{\bar{F}_k}_{L^2(\R^3 \times ]0,t[)}^2 \leq Ct^{\frac{1}{2}}$.\footnote{This does not rely on $\bar{F}_k$ belonging to subcritical spaces.} Due to the splitting properties \eqref{u0Nsplitrevised} and \eqref{FNsplitrevised}, we have that $\norm{\tilde{F}}_{L^2(Q_{T})}^2 \leq C$. %N^{2-q}$.
	Using (\ref{barPNKsubcritest2}), it is not difficult to show that
		\begin{equation}
			\norm{\bar{P}_k}_{\cK^{-1+\delta}_{\infty}(Q_{T_\sharp})} \leq C (\norm{\bar{u_0}}_{\dot B^{s_{p_2}+\delta_2}_{p_2,p_2}(\R^3)} + \norm{\bar{F}}_{\cF^{s_{p_3}^{'}+\delta_3}_{p_3}(Q_{T})}) \leq C. % (N^{\gamma_1} + N^{\frac{1}{2}}).
			\label{}
		\end{equation}
Substituting all the estimates into \eqref{intermediateenergyest}, we obtain that
	\begin{equation}
		y(t) \leq C (1+t^{\frac{1}{2}}) + C \int_0^t \frac{y(\tau)}{\tau^{1-\delta}} \, d\tau.
	\end{equation}
	Now we apply Gronwall's lemma:
	\begin{equation}\label{yNest2}
		y(t)\leq C (1+t^{\frac 1 2}) \times \exp \big( C t^{\delta} \big).
\end{equation}

We combine \eqref{yNest2}, \eqref{ENkfineest}, and \eqref{usplitest} to obtain the following estimate for each $t \in ]0,T]$:
\begin{equation}
	\norm{u(\cdot,t)}_{L^2}^2 \leq C (1+t^{\frac 1 2}) \times \big[ \exp \big( C t^{\delta} \big) + 1 \big].
	\label{finaldecayest}
\end{equation}
Let $t=1$ to verify~\eqref{improveddecayesttime1} and complete the proof.
%Finally, we select $N=t^{-\kappa}$ with $\kappa := \delta/\max(2\gamma_1,1)$ to obtain the desired decay near the initial time.% This completes the proof.
\end{proof}

\begin{cor}[Global energy inequality, revised]\label{cor:uglobalenergyineqinitialtime}
	Under the hypotheses of Proposition~\ref{improveddecayprop}, we have that $P_k \otimes u \in L^2(Q_T)$ and, for all finite $t \in ]0,T]$,
	\begin{equation}\label{uglobalenergyineqinitialtime}
		\norm{u(\cdot,t)}_{L^2(\R^3)}^2 + 2 \int_{0}^{t} \int_{\R^3} |\nabla u|^2 \, dx \, dt' \leq 2 \int_{0}^{t} \int_{\R^3} [(P_k \otimes u) + F_k] : \nabla u \, dx \, dt'.
	\end{equation}
\end{cor}

\begin{remark}[On the constant in the decay estimate]
\label{ontheconstant}
From the proof of Proposition~\ref{improveddecayprop}, see~\eqref{finaldecayest}, we may take $C = Q(M) \exp( Q(M) )$ in the decay estimate \eqref{improveddecayeq}, where $Q$ is a polynomial with coefficients depending on $k,p,q$, zero constant term, and degree depending only on $k$. Therefore, plugging~\eqref{improveddecayeq} back into \eqref{uglobalenergyineqinitialtime}, we obtain
\begin{equation}
	\norm{u(\cdot,t)}_{L^2(\R^3)}^2 + 2 \int_{0}^{t} \int_{\R^3} |\nabla u|^2 \, dx \, dt' \leq Q(M) \exp( Q(M) ) t^{\frac{1}{2}}.
\end{equation}
\end{remark}

\subsection{Weak Leray-Hopf solutions}\label{sec:relationshiplerayhopf}

In this subsection, we prove the equivalence of suitable weak Leray-Hopf solutions and global weak Besov solutions under certain assumptions. Let $J(\R^3)$ denote the space of divergence-free vector fields in $L^2(\R^3)$.
\begin{pro}[Equivalence of suitable weak Leray-Hopf solutions and weak Besov solutions]\label{pro:lerayhopfisweakbesov}
Let $0 < T \leq \infty$, $u_0 \in J(\R^3) \cap \dot B^{s_p}_{p,\infty}(\R^3)$, and $F \in L^2(Q_T) \cap \cK_q(Q_T)$ for some $p \in ]3,\infty[$ and $q \in ]3,p]$. A distributional vector field $v$ on $Q_T$ is a suitable weak Leray-Hopf solution on $Q_T$ with initial data $u_0$ and forcing term $\div F$ if and only if $v$ is a weak Besov solution on $Q_T$ with the same initial data and forcing term.
\end{pro}
Later, we will use Proposition~\ref{pro:lerayhopfisweakbesov} to prove the existence of global weak Besov solutions in Corollary~\ref{existencecor}. First, we remind the reader of the definition of suitable weak Leray-Hopf solution. Recall that $C^\infty_{0,0}(Q_T)$ denotes the space of smooth vector fields $\varphi \: Q_T \to \R^3$ with compact support and $\div \varphi = 0$.
\begin{definition}[Suitable weak Leray-Hopf solution]
	\label{weaklerayhopfdef}
Let $T > 0$, $u_0 \in J(\R^3)$, and $F \in L^2(Q_T)$.
	
	We say that a distributional vector field $v$ on $Q_T$ is a \emph{weak Leray-Hopf solution} to the Navier-Stokes equations on $Q_T$ with initial data $u_0$ and forcing term $\div F$ if $v$ satisfies the following properties:
	
	First, $v \in L^\infty_t L^2_x \cap L^2_t \dot H^1_x(Q_T)$, and $v$ satisfies the Navier-Stokes equations on $Q_T$ in the sense of divergence-free distributions:
	\begin{equation}
		\int_0^T \int_{\R^3} v \cdot \p_t \varphi + v \otimes v : \nabla \varphi - \nabla v : \nabla \varphi - F : \nabla \varphi \, dx \, dt = 0
		\label{}
	\end{equation}
	for all $\varphi \in C^\infty_{0,0}(Q_T)$. In addition, $v(\cdot,t)$ is weakly continuous on $[0,T]$ as an $L^2(\R^3)$-valued function, and $v$ attains its initial data strongly in $L^2(\R^3)$:
	\begin{equation}
		\lim_{t \dto 0} \norm{v(\cdot,t) - u_0}_{L^2(\R^3)} = 0.
		\label{}
	\end{equation}
	Finally, it is required that $v$ satisfies the energy inequality
	\begin{equation}
		\norm{v(\cdot,t)}_{L^2(\R^3)}^2 + 2 \int_0^t \int_{\R^3} |\nabla v(x,t')|^2 \, dx \, dt' \leq \norm{u_0}_{L^2(\R^3)}^2 - \int_0^t \int_{\R^3} F : \nabla v \, dx \, dt',
		\label{globalenergyv}
	\end{equation}
	for all $t \in [0,T]$.
		
	We say that a distributional vector field $v$ on $Q_\infty$ is a weak Leray-Hopf solution on $Q_\infty$ if it is a weak Leray-Hopf solution on $Q_T$ for all $T > 0$. These solutions are known as \emph{global} weak Leray-Hopf solutions.

	Now let $0 < T \leq \infty$. We say that a weak Leray-Hopf solution $v$ on $Q_T$ is \emph{suitable} if there exists a pressure $q \in L^{\frac{3}{2}}_\loc(Q_T)$ such that $(v,q)$ is a distributional solution of the Navier-Stokes equations on $Q_T$ with forcing term $\div F$ and moreover satisfies the local energy inequality \eqref{vlocalenergyineq} for all $0 \leq \varphi \in C^\infty_0(Q_\infty)$.
\end{definition}

The following proposition concerning the existence of suitable weak Leray-Hopf solutions is well known (see, for instance, \cite{lemarie2002}).
\begin{pro}[Existence of suitable weak Leray-Hopf solutions]\label{pro:lerayhopfexist}
Let $u_0 \in J(\R^3)$ and $F \in L^2(Q_T)$ for all $T > 0$. There exists a global suitable weak Leray-Hopf solution with initial data $u_0$ and forcing term $\div F$.
\end{pro}

We now prove Proposition~\ref{pro:lerayhopfisweakbesov}.
\begin{proof}[Proof of Proposition~\ref{pro:lerayhopfisweakbesov}]
	 Assume the hypotheses of the proposition. It suffices to consider the case $T < \infty$. We now record a few properties. Namely,
\begin{equation}\label{P0energy}
	P_{0}(u_0,F) \in C([0,T]; L^2(\R^3)) \cap L^2_t \dot H^1_x(Q_T) \cap \cK_p(Q_T),
 \end{equation}
 \begin{equation}\label{P0initialdata}
	 \lim_{t \dto 0} \|P_{0}(u_0,F)(\cdot,t)-u_0\|_{L^{2}(\mathbb{R}^3)}=0.
 \end{equation}
Combining these observations with \eqref{Picardestimatefrequent} and \eqref{bilinearenergyestimate} gives that $P_k(u_0,F) \in L^\infty_t L^2_x(Q_T) \cap \cK_p(Q_T)$ for all $k \geq 0$. Next, the interpolation inequality~\eqref{Kinterp} implies $P_k \in L^4(Q_T)$. Finally, since $F_k(u_0,F) = P_k \otimes P_k - P_{k-1} \otimes P_{k-1}$, we obtain that $F_k \in L^2(Q_T)$ for all $k \geq 0$.

	%We will start with the forward implication.
\emph{1. Forward direction}. Suppose that $v$ is a suitable weak Leray-Hopf solution on $Q_{T}$
with initial data $u_0$ and forcing term $\div F$. 
From~\eqref{P0energy} and Definition~\ref{weaklerayhopfdef}, it is clear that $u = v - P_0(u_0,F) \in L^\infty_t L^2_x \cap L^2_t \dot H^1_x(Q_T)$
 %\end{equation}
  and $u(\cdot,t)$ is weakly continuous on $[0,T]$ with values in $L^2(\R^3)$.
 In addition, \eqref{P0initialdata} implies that
 %\begin{equation}
	$\lim_{t \dto 0}\|u(\cdot,t)\|_{L^2(\mathbb{R}^3)}=0$.
 %\end{equation}
Since $F_k(u_0,F) \in L^2(Q_T)$ for all $k\geq 0$, we conclude that $v$ is a weak Besov solution on $Q_T$ based on the zeroth Picard iterate.

 \emph{2. Reverse direction}. Suppose that $v$ is a weak Besov solution on $Q_{T}$
with initial data $u_0$ and forcing term $\div F$. As observed above, $F_k(u_0,F) \in L^2(Q_T)$ for all $k \geq 0$. Hence, Proposition~\ref{pro:orderofpicarditerate} implies that $v$ is a weak Besov solution on $Q_T$ based on the zeroth Picard iterate.
By~\eqref{P0energy}, \eqref{P0initialdata}, and the properties of $u = v - P_0(u_0,F)$ in Definition~\ref{weaksol} (with $k=0$), we have that
%\begin{equation}
	$v \in L^\infty_t L^2_x \cap L^2_t \dot H^1_x(Q_T)$,
%	\label{}
%\end{equation}
$v(\cdot,t)$ is weakly continuous on $[0,T]$ with values in $L^2(\R^3)$, and
%\begin{equation}
	$\lim_{t \dto 0} \norm{v(\cdot,t) - u_0}_{L^2(\R^3)} = 0$.
	%\label{}
%\end{equation}
It remains to verify the global energy inequality \eqref{globalenergyv} for weak Leray-Hopf solutions. This is obtained from the local energy inequality \eqref{vlocalenergyineq} by similar arguments as in Step~2 of Proposition~\ref{pro:vlocalglobal}. The proof is complete.% $v$ is a suitable Leray-Hopf solution on $Q_{T}$.
\end{proof}

\subsection{Weak--$\ast$ stability}\label{sec:stabilityexistence}

Here is our main result concerning weak-$\ast$ stability.

\begin{pro}[Weak--$\ast$ stability]\label{pro:stability}
	Let $0 < T \leq \infty$ and $(v^{(n)})_{n \in \N}$ be a sequence of weak Besov solutions on $Q_T$. For each $n \in \N$, denote by $u_0^{(n)}$ and $\div F^{(n)}$ the initial data and forcing term of $v^{(n)}$, respectively. Suppose that
	\begin{equation}
		u_0^{(n)} \wstar u_0 \text{ in } \dot B^{s_p}_{p,\infty}(\R^3), \quad F^{(n)} \wstar F \text{ in } \cF_q(Q_T)
		\label{}
	\end{equation}
for some $p \in ]3,\infty[$ and $q \in ]3,p]$. There exists a subsequence (still denoted by~$n$) converging in the following senses to a weak Besov solution $v$ on $Q_T$ with initial data $u_0$ and forcing term $\div F$. Namely,
%	\begin{equaation}\label{unenergyspaceconv}
%		u^{(n)} \wstar u \text{ in } L^\infty_t L^2_x(Q_{\tilde{T}}), \quad
%		\nabla u^{(n)} \wto \nabla u \text{ in } L^2(Q_{\tilde{T}}),
%	\end{equation}
%	for all $0 < \tilde{T} \leq T$,
	%\begin{equation}
	%	v^{(n)} \wstar v \text{ in } L^\infty_t L^2_x(Q_S), \quad \nabla v^{(n)} \wto v \text{ in } L^2(Q_S),
	%	\label{}
	%\end{equation}
	\begin{equation}
	v^{(n)} \wstar v \text{ in } L^\infty_t L^2_x(B(R) \times ]\delta,S[), \;
	\nabla v^{(n)} \wto \nabla v \text{ in } L^2(B(R) \times ]\delta,S[),
		\label{localconvofv}
	\end{equation}
	\begin{equation}\label{vqconvergence}
	v^{(n)} \to v \text{ in } L^3(B(R) \times ]\delta,S[), \quad q^{(n)} \wto q \text{ in } L^{\frac{3}{2}}(B(R) \times ]\delta,S[),
	\end{equation}
	\begin{equation}\label{vntimeconvergence}
		v^{(n)} \to v \text{ in } C([0,S];\cS'(\R^3)),
	\end{equation}
for each $R > 0$, $S \in ]0,T]$ finite, and $\delta \in ]0,S[$. Here, $q^{(n)}$ and $q$ denote the pressures associated to $v^{(n)}$ and $v$, respectively.
\end{pro}

First, we require an analogous result for the Picard iterates.

\begin{lemma}[Weak--$\ast$ stability of Picard iterates]\label{Picardconvlem}
	Under the hypotheses of Proposition~\ref{pro:stability}, there exists a subsequence (still denoted by~$n$) such that for each $0 \leq k \in \Z$, the Picard iterates $P_k^{(n)} := P_k(u_0^{(n)},F^{(n)})$ converge in the following senses to $P_k(u_0,F)$. Namely,
	\begin{equation}\label{Pkconvergence}
		P_k^{(n)} \wstar P_k \text{ in } \cK_p(Q_T), \quad
	P_k^{(n)} \to P_k \text{ in } L^q(B(R) \times ]\delta,S[),
	\end{equation}
	\begin{equation}
	\nabla P_k^{(n)} \wto \nabla P_k^{(n)} \text{ in } L^q(B(R) \times ]\delta,S[),
		\label{}
	\end{equation}
	\begin{equation}\label{Pktimeconvergence}
		P_k^{(n)} \to P_k \text{ in } C([0,S];\cS'(\R^3)),
	\end{equation}
for each $R > 0$, $S \in ]0,T]$ finite, and $\delta \in ]0,S[$.
\end{lemma}

In the proofs below, we allow the implicit constants $C > 0$ to depend on $k$, $M$, $p$, $q$. We will also not vary our notation when passing to subsequences.

\begin{proof}[Proof of Lemma \ref{Picardconvlem}]
	It suffices to consider the case when $T < \infty$. Due to weak-$\ast$ convergence, there exists a constant $M > 0$ such that
	\begin{equation}
		\sup_{n \in \N} [\norm{u_0^{(n)}}_{\dot B^{s_p}_{p,\infty}(\R^3)} + \norm{F^{(n)}}_{\cF_q(Q_T)}] \leq M.
		\label{}
	\end{equation}
	Let $k \geq 0$ be a fixed integer. From~\eqref{Picardestimatefrequent}, we obtain
	\begin{equation}
		\sup_{n \in \N} [\norm{P_k^{(n)}}_{\cK_p(Q_T)} + \norm{P_k^{(n)}}_{\cX_T}] \leq C(k,M,p,q).
		\label{Pnkbounds}
	\end{equation}
	 Therefore, there exists a subsequence such that $P_k^{(n)} \wstar \tilde{P}_k$ in $\cK_p(Q_T)$. Eventually, we will show that $\tilde{P}_k = P_k(u_0,F)$.
	
	\emph{1. Strong convergence in $L^q$}. Consider the heat equation satisfied by the Picard iterates:
	\begin{equation}\label{heateqnPk}
		\begin{aligned}
			\p_t P^{(n)}_k - \Delta P^{(n)}_k = \bP \div [F^{(n)} - P^{(n)}_{k-1} \otimes P^{(n)}_{k-1}] 
	\end{aligned} \text{ in } Q_T
	\end{equation}
in the sense of distributions.
Interior estimates for~\eqref{heateqnPk} give us the following gradient estimate for $P^{(n)}_k$ on domains $Q := B(R) \times ]\delta,T[$. For all $n \in \N$,
	\begin{equation}
		\norm{P^{(n)}_k}_{L^q_t W^{1,q}_x(Q)} + \norm{\p_t P^{(n)}_k}_{L^q_t W^{-1,q}_x(Q)} \leq C(Q),
		\label{}
	\end{equation}
for all $R > 0$ and $\delta \in ]0,T[$. Hence, we may assume that $\nabla P_k^{(n)} \wstar \nabla \tilde{P}_k$ in $L^q(B(R) \times ]\delta,T[)$.
	%The estimate on $\p_t P^{(n)}_k$ is obtained by combining the gradient estimate with \eqref{heateqnPk}.
	By the Aubin-Lions lemma (see, for example, Seregin's book~\cite[Proposition 1.1]{sereginnotes} or the paper~\cite{aubin}) in the function spaces
	\begin{equation}
		W^{1,q}(\Omega) \overset{{\rm cpt}}{\into} L^q(\Omega) \into W^{-1,q}(\Omega)
		\label{}
	\end{equation}
and a diagonal argument, there exists a subsequence such that $P^{(n)}_k \to \tilde{P}_k$ in $L^q(B(R) \times ]\delta,T[)$ for all $R > 0$ and $\delta \in ]0,T[$.
	\begin{comment}
	\emph{2. Pressure estimates}. For each $n \in \N$, we decompose the pressure associated to the Picard iterate as $\pi^{(n)}_k = \pi^{(n)}_{k,1} - \pi^{(n)}_{k,2}$, where
	\begin{equation}
		\pi^{(n)}_{k,1} := (-\Delta)^{-1} \div \div P^{(n)}_{k-1} \otimes P^{(n)}_{k-1}, \quad \pi^{(n)}_{k,2} := (-\Delta)^{-1} \div \div F^{(n)}.	
	\end{equation}
	By the Calder{\'o}n-Zygmund estimates,
	\begin{equation}
		\sup_{n \in \N} [\norm{\pi^{(n)}_{k,1}}_{\cF_p(Q_T)} + \norm{\pi^{(n)}_{k,2}}_{\cF_q(Q_T)}] \leq C.
	\end{equation}
	Hence, there exists a subsequence such that 
	\begin{equation}
		\pi^{(n)}_{k,1} \wstar \tilde{\pi}_{k,1} \text{ in } \cF_p(Q_T), \quad \pi^{(n)}_{k,2} \wstar \tilde{\pi}_{k,2} \text{ in } \cF_q(Q_T),
		\label{}
	\end{equation}
	We denote $\tilde{\pi}_k := \tilde{\pi}_{k,1} - \tilde{\pi}_{k,2}$.
	\end{comment}

	\emph{2. Weak continuity in time}.
Let $\varphi$ be a vector field belonging to the Schwartz class on $\R^3$. Since $P^{(n)}_k \in C([0,T];\cS'(\R^3))$ for all $n \in \N$, we consider the family $\cF_{k,\varphi} \subset C([0,T])$ consisting of the functionals
\begin{equation}
	[0,T] \to \R : t \mapsto \la P^{(n)}_k(\cdot,t), \varphi \ra, \quad \, n \in \N.
	\label{}
\end{equation}
Our goal is to apply the Arzel{\`a}-Ascoli theorem to the family $\cF_{k,\varphi}$.
Recall from \eqref{timespacebesovest} in Section \ref{sec:preliminaries} that $\sup_{n \in \N} \norm{P^{(n)}_k}_{L^\infty_t \dot B^{s_p}_{p,\infty}(Q_T)} \leq C$, so using the characterisation of dual spaces for homogeneous Besov spaces (see \cite[Chapter 2]{bahourichemindanchin}, for example) we obtain
	\begin{equation}
		\sup_{n \in \N} |\la P^{(n)}_k(\cdot,t), \varphi \ra| \leq C \norm{\varphi}_{\dot B^{-s_p}_{p',1}(\R^3)}, \quad t \in [0,T].
		\label{Pnkfunnybesovbound}
	\end{equation}
	Therefore, $\cF_{k,\varphi}$ is a bounded subset of $C([0,T])$. To prove equicontinuity, we estimate $\p_t P^{(n)}_k(\cdot,t)$ with values in the space $\cZ := W^{-2,p}(\R^3) + W^{-1,\frac{p}{2}}(\R^3)+W^{-1,q}(\R^3)$. The space $\cZ$ is motivated by the estimate
	\begin{equation}
	\sup_{t \in ]0,T[} [ t^{-\frac{s_p}{2}} \norm{\Delta P_k^{(n)}}_{W^{-2,p}(\R^3)} + t^{-s_p} \norm{\bP \div P_{k-1}^{(n)} \otimes P_{k-1}^{(n)}}_{W^{-1,\frac{p}{2}}(\R^3)} +$$$$+t^{-\frac{s_{q}^{'}}{2}}\norm{\bP\div F^{(n)}}_{W^{-1,q}(\R^3)}] \leq C.
		\label{timederivestpk}
	\end{equation}
	Let $r > 1$ such that $r\min{\big(s_p,\frac{s_{q}^{'}}{2}\big)} > -1$. The time derivative $\p_t P^{(n)}_k$ is estimated from the other terms in time-dependent Stokes equations \eqref{heateqnPk} and \eqref{timederivestpk} to obtain $\norm{\p_t P^{(n)}_k}_{L^r_t {\cZ}_x(Q_T)} \leq C(r)$. This gives us equicontinuity:
\begin{equation}
	\sup_{n \in \N} \lvert \la P^{(n)}_k(\cdot,t_2), \varphi \ra - \la P^{(n)}_k(\cdot,t_1), \varphi \ra \rvert \leq |t_2-t_1|^{1-\frac{1}{r}} C(r,\varphi),
	\label{}
\end{equation}
for all $t_1,t_2 \in [0,T]$. Hence, there exists a subsequence such that
\begin{equation}
	\la P_k^{(n)}(\cdot,t), \varphi \ra \to \la \tilde{P}_k(\cdot,t), \varphi \ra \text{ in } C([0,T]).
	\label{Pknweakconv}
\end{equation}
The above argument was for a single vector field $\varphi$. Let $(\varphi_m)_{m \in \N} \subset \cS(\R^3)$ be a dense sequence of vector fields in $\dot B^{-s_p}_{p',1}(\R^3)$. By the previous reasoning and a diagonal argument, there exists a subsequence such that
\begin{equation}
		\la P_k^{(n)}(\cdot,t), \varphi_m \ra \to \la \tilde{P}_k(\cdot,t), \varphi_m \ra \text{ in } C([0,T])
	\end{equation}
for all $m \in \N$. From the estimate \eqref{Pnkfunnybesovbound} and the density of $(\varphi_m)_{m \in \N}$ in $\dot B^{-s_p}_{p',1}$, one may show that \eqref{Pknweakconv} is valid for all Schwartz vector fields $\varphi$. Moreover, $\tilde{P}_k(\cdot,0) = u_0$.

\emph{3. Showing $\tilde{P}_k = P_k(u_0,F)$}.
First, note that while the convergence arguments up to now were for a fixed $k \geq 0$, we may assume they hold for all $k \geq 0$ simultaneously by a diagonalization argument. Let us proceed inductively. For the base case, we may write $\tilde{P}_{-1} = P_{-1}(u_0,F) = 0$. Next, suppose that $\tilde{P}_{k-1} = P_{k-1}(u_0,F)$ for a given $k \geq 0$. %Recall the time-dependent Stokes equations satisfied on $Q_T$ by the Picard iterates:
%\begin{equation}
%	\begin{aligned}
%		\p_t P^{(n)}_k - \Delta P^{(n)}_k + \nabla \pi^{(n)}_k &= \div F^{(n)} - \div P^{(n)}_{k-1} \otimes P^{(n)}_{k-1} \\
%		\div P^{(n)}_k &= 0.
%	\end{aligned}
%\end{equation}
Let $n \to \infty$ in~\eqref{heateqnPk} to obtain the following heat equation:
\begin{equation}
	\begin{aligned}
		\p_t \tilde{P}_k - \Delta \tilde{P}_k = \bP \div [F - P_{k-1} \otimes P_{k-1}]
\end{aligned}
\text{ in } Q_T
	\label{tildePkeqn}
\end{equation}
in the sense of distributions.\footnote{Here, we require the following fact concerning the Leray projector. For a sequence of vector fields $f_n \wto f$ in $L^l_t L^r_x(Q_T)$, $1 < l,r < \infty$, we also have $\bP f_n \wto \bP f$ in $L^l_t L^r_x(Q_T)$ due to the observation that $\int_{Q_T} \bP f_n \cdot \varphi \, dx \,dt = \int_{Q_T} f_n \cdot \bP \varphi \, dx \, dt$ for all vector fields $\varphi \in L^{l'}_t L^{r'}_x(Q_T)$.} Also, $\tilde{P}_k(\cdot,0) = u_0$. Therefore, $\tilde{P}_k \equiv P_k(u_0,F)$ on $Q_T$ due to the well-posedness of the heat equation in $C([0,T];\cS'(\R^3))$. This completes the induction and the proof.
\end{proof}
\begin{remark}
Using \eqref{Pkconvergence}, \eqref{Pnkbounds}, and interpolation, we have
\begin{equation}\label{convergPnkanyleb}
P^{(n)}_{k} \rightarrow P_{k} \text{ in } L^{l}(B(R)\times ]\delta,S[)
\end{equation}
for any $l \geq 1$, $R > 0$, $S\in ]0,T]$ finite, and $\delta\in ]0,S[$.
\end{remark}

We are now ready to prove Proposition~\ref{pro:stability}.

\begin{proof}[Proof of Proposition~\ref{pro:stability}]
	It suffices to consider the case when $T < \infty$. Let $k := k(p)$. As in Lemma~\ref{Picardconvlem}, exists a constant $M > 0$ such that
	\begin{equation}
		\sup_{n \in \N} [\norm{u_0^{(n)}}_{\dot B^{s_p}_{p,\infty}(\R^3)} + \norm{F^{(n)}}_{\cF_q(Q_T)}] \leq M,
		\label{}
	\end{equation}
		\begin{equation}
			\sup_{n \in \N} [\norm{P_k^{(n)}}_{\cK_p(Q_T)} + \norm{P_k^{(n)}}_{\cK_\infty(Q_T)}] \leq C(k,M,p,q).
		\label{Pnkuniformest}
	\end{equation}
	According to Proposition~\ref{pro:orderofpicarditerate}, $v^{(n)} = u^{(n)} + P_k(u_0,F)$ is a weak Besov solution on $Q_T$ based on the $k$th Picard iterate for each $n \in \N$.
	%From the remainder of the proof, the constants $C$ may be depend on $k,M,p,q$, and we will not vary notation when passing to subsequences. 
	
	\emph{1. Energy estimates}. Recall the uniform decay estimate from Proposition~\ref{improveddecayprop}. Namely, there exists $\alpha > 0$ such that
	\begin{equation}
	\sup_{n \in \N} \norm{u^{(n)}(\cdot,t)}_{L^2(\R^3)} \leq Ct^{\alpha}, \quad t \in ]0,\min(T,1)[.
		\label{improveddecayeststab}
	\end{equation}
	By combining \eqref{Pnkuniformest}, \eqref{improveddecayeststab}, and the global energy inequality \eqref{uglobalenergyineqinitialtime}, we obtain the following Gronwall-type estimate:
	\begin{equation}
\sup_{n \in \N} \sup_{t \in ]0,T[} \int_{\R^3} |u^{(n)}(x,t)|^2 \, dx + 2 \int_0^t \int_{\R^3} |\nabla u^{(n)}|^2 \, dx \, dt' \leq C.
		\label{}
	\end{equation}
For each $n \in \N$, we estimate the time derivative $\p_t u^{(n)}$ in a negative Sobolev space according to the Navier-Stokes equations:
\begin{equation}\label{untimederivest}
	\sup_{n \in \N} \norm{\p_t u^{(n)}}_{L^2_t H^{-\frac{3}{2}}_x(Q_T)} \leq \sup_{n \in \N} C [\lVert \Delta u^{(n)} - \div F^{(n)}_k - \div P^{(n)}_k \otimes u^{(n)} $$ $$ - u^{(n)} \otimes P^{(n)}_k \rVert_{L^2_t H^{-1}_x(Q_T)}
	+ \norm{\div u^{(n)} \otimes u^{(n)}}_{L^2_t H^{-\frac{3}{2}}_x(Q_T)}] \leq C.
\end{equation}
By the Banach-Alaoglu theorem, we obtain a subsequence
\begin{equation}
	u^{(n)} \wstar u \text{ in } L^\infty_t L^2_x(Q_T), \quad \nabla u^{(n)} \wto \nabla u \text{ in } L^2(Q_T),
	\label{}
\end{equation}
\begin{equation}
\esssup_{t \in ]0,\min(1,T)[} \frac{\norm{u(\cdot,t)}_{L^2(\R^3)}}{t^\alpha} \leq C.
	\label{uniformdecaylimitfunction}
\end{equation}
A standard application of the Aubin-Lions lemma %in spaces $H^1(B(R)) \overset{ {\rm cpt}}{\into} L^2(B(R)) \into H^{-\frac{3}{2}}(B(R))$ for $R > 0$
implies that $u^{(n)} \to u$ in $L^2(B(R) \times ]0,T[)$ for all $R > 0$. Moreover, since $\sup_{n \in \N} \norm{u^{(n)}}_{L^{\frac{10}{3}}(Q_T)} \leq C$, we obtain
\begin{equation}
u^{(n)} \to u \text{ in } L^{l}(B(R) \times ]0,T[)
	\label{}
\end{equation}
for all $R > 0$ and $l \in [1,\frac{10}{3}[$, by interpolation.
In addition, by the estimates \eqref{untimederivest} and arguments similar to those in Lemma \ref{Picardconvlem}, we have a subsequence such that
\begin{equation}
	\int_{\R^3} u^{(n)}(x,t) \varphi(x) \, dx \to \int_{\R^3} u(x,t) \varphi(x) \, dx \text{ in } C([0,T]).
	\label{}
\end{equation}
Hence, $u(\cdot,t)$ is weakly continuous as an $L^2(\R^3)$-valued function, and by \eqref{uniformdecaylimitfunction}, we have
\begin{equation}
	\lim_{t \dto 0} \norm{u(\cdot,t)}_{L^2(\R^3)} = 0.
	\label{}
\end{equation}

\emph{2. Pressure estimates}. As described in Remark~\ref{pressurermk}, we may take $q^{(n)}$ to be the associated pressure:
	\begin{equation}
		q^{(n)} = q_1^{(n)} + q_2^{(n)} + q_3^{(n)} + q^{(4)}, \quad n \in \N,
		\label{}
	\end{equation}
	where
	\begin{align}
		q_1^{(n)} &:= (-\Delta)^{-1} \div \div u^{(n)} \otimes u^{(n)}, \\
		q_2^{(n)} &:= (-\Delta)^{-1} \div \div P_k^{(n)} \otimes u^{(n)} + u^{(n)} \otimes P_k^{(n)}, \\
		q_3^{(n)} &:= (-\Delta)^{-1} \div \div P_k^{(n)} \otimes P_k^{(n)}, \\
		q_4^{(n)} &= (-\Delta)^{-1} \div \div F^{(n)}.
		\label{}
	\end{align}
By the Calder{\'o}n-Zygmund estimates, for all $\delta \in ]0,T[$,
	\begin{equation}
		\sup_{n \in \N} \norm{q_1^{(n)}}_{L^{\frac{3}{2}}(Q_T)} \leq C \norm{u}_{L^3(Q_T)}^2 \leq C,
\end{equation}
\begin{equation}
\sup_{n \in \N} \norm{q_2^{(n)}}_{L^3(\R^3 \times ]\delta,T[)} \leq C \norm{u}_{L^3(Q_T)} \norm{ P^{(n)}_k}_{L^\infty(\R^3 \times ]\delta,T[)} \leq C(\delta),
		\label{}
	\end{equation}
	\begin{equation}
		\sup_{n \in \N} \norm{q_3^{(n)}}_{\cF_{p}(Q_T)} \leq C \norm{P^{(n)}_k}_{\cK_p(Q_T)} \norm{P^{(n)}_k}_{\cK_\infty(Q_T)} \leq C,
		\label{}
	\end{equation}
	\begin{equation}
		\sup_{n \in \N} \norm{q_4^{(n)}}_{\cF_q(Q_T)} \leq C \norm{F^{(n)}}_{\cF_q(Q_T)} \leq C.
		\label{}
	\end{equation}
There exists a subsequence such that for all $R > 0$ and $\delta \in ]0,T[$,
	\begin{equation}
	q^{(n)} \rightharpoonup q \text{ in } L^{\frac{3}{2}}(B(R) \times ]\delta,T[),
		\label{}
	\end{equation}
and $q(\cdot,t) \in L^{\frac{3}{2}}(\R^3) + L^3(\R^3) + L^p(\R^3) + L^q(\R^3)$ for a.e. $t \in ]0,T[$. Hence, $c \equiv 0$ in Remark~\ref{pressurermk}. That is, $q$ is the pressure associated to $v$.
%We have already considered the convergence of the Picard iterates in Lemma \ref{Picardconvlem}, so \eqref{vqconvergence} in Proposition \ref{pro:stability} has been verified. In the limit $n \upto \infty$, we obtain $-\Delta q = \div \div (v \otimes v - F)$. Since $q(\cdot,t)$ belongs to a Lebesgue space for almost every $t \in ]0,T[$, we conclude from the Liouville theorem for harmonic functions that 
	%\begin{equation}
	%	q = (-\Delta)^{-1} \div \div (v \otimes v - F),
	%	\label{}
	%\end{equation} i.e., it is the pressure associated to $v$.

	\emph{3. Local energy inequality}. It remains to verify that $v$ satisfies the local energy inequality \eqref{vlocalenergyineq}. It will be more convenient\footnote{This way, we avoid the problematic term $\int_0^t \int_{\R^3} F^{(n)} : \nabla (\varphi v^{(n)}) \,dx \,dt'$ in the local energy inequality for $v$, since $F^{(n)}$ is only assumed to converge weakly--$\ast$ in $\cF_q(Q_T)$.} to examine the energy inequality satisfied by $u^{(n)}$:
	%see \eqref{unlocalenergyineq}:
	\begin{equation}\label{unlocalenergyineq}
		\int_{\R^3} \varphi(x,t) |u^{(n)}(x,t)|^2 \, dx + 2 \int_0^t \int_{\R^3} \varphi |\nabla u^{(n)}|^2 \, dx \, dt' \leq $$ $$ \leq \int_0^t \int_{\R^3} |u^{(n)}|^2 (\p_t \varphi + \Delta \varphi) + [|u^{(n)}|^2(u^{(n)}+ P_k^{(n)}) + 2p^{(n)}u^{(n)}] \cdot \nabla \varphi \,dx \,dt' + $$ $$ + 2 \int_0^t \int_{\R^3} [P_k^{(n)} \otimes u^{(n)} + F_k^{(n)}]: \nabla (\varphi u^{(n)}) \, dx \, dt',
	\end{equation}
for all $t \in ]0,T]$ and $0 \leq \varphi \in C^\infty_0(Q_\infty)$.
Each term in \eqref{unlocalenergyineq} converges to its corresponding term in \eqref{ulocalenergyineq} with $u,p,P_k,F_k$ replacing $u^{(n)},p^{(n)},P_k^{(n)},F_k^{(n)}$ \emph{except} for the term
	\begin{equation}
		\int_0^t \int_{\R^3} \varphi |\nabla u^{(n)}|^2 \, dx \, dt'.
		\label{}
	\end{equation}
Since $\int_0^T \int_{\R^3} |\nabla u^{(n)}|^2 \, dx \,dt' \leq C$, there exists a subsequence such that
\begin{equation}
	|\nabla u^{(n)}|^2 - |\nabla u|^2 \wstar \mu \text{ in } \mathcal{M}(Q_S),
	\label{}
\end{equation}
where $\mathcal{M}(Q_T)$ is the Banach space of all finite Radon measures on $Q_T$. Moreover, since $\nabla u^{(n)} \wto \nabla u$ in $L^2(Q_T)$, the lower semicontinuity of the $L^2$-norm implies that $\mu \geq 0$. Therefore,
\begin{equation}
		\int_0^t \int_{\R^3} \varphi |\nabla u^{(n)}|^2 \ dx \,dt' \to \int_0^t \int_{\R^3} \varphi |\nabla u|^2 \, dx \,dt' + \mu
	\label{}
\end{equation}
for all $0\leq \varphi \in C^\infty_0(Q_\infty)$ and $t \in ]0,T]$, and $u$ satisfies the local energy inequality \eqref{ulocalenergyineq}. By Remark~\ref{otherdirectionlocalenergy}, $v$ satisfies its corresponding local energy inequality \eqref{vlocalenergyineq}. This completes the proof.
\end{proof}

As a consequence of weak-$\ast$ stability, we obtain an existence result for global weak Besov solutions.

\begin{cor}[Existence]\label{existencecor}
Let $0 < T \leq \infty$, $u_0 \in \dot B^{s_p}_{p,\infty}(\R^3)$ be a divergence-free vector field, and $F \in \cF_q(Q_T)$ for some $p \in ]3,\infty[$ and $q \in ]3,p]$. There exists a weak Besov solution $v$ on $Q_T$ with initial data~$u_0$ and forcing term~$\div F$.
\end{cor}

First, we require the following lemma which we state without proof.
\begin{lemma}[Density]\label{lemma:density}
	Under the hypotheses of Corollary~\ref{existencecor}, there exist sequences $(u_0^{(n)})_{n \in \N} \subset \cS(\R^3;\R^3)$ and $(F^{(n)})_{n \in \N} \subset C^\infty_0(Q_T;\R^{3\times3})$ such that $u_0^{(n)} \wstar u_0$ in $\dot B^{s_p}_{p,\infty}(\R^3)$, $F^{(n)} \wstar F$ in $\cF_q(Q_T)$, and for each $n \in \N$, $\div u_0^{(n)} = 0$
\end{lemma}

\begin{proof}[Proof of Corollary \ref{existencecor}]
	Let $( u_0^{(n)} )_{n \in \N}$ and $(F^{(n)})_{n \in \N}$ be the approximating sequences from Lemma~\ref{lemma:density}.
By Proposition \ref{pro:lerayhopfexist}, there exists a sequence $(v^{(n)})_{n \in \N}$ of suitable weak Leray-Hopf solutions on $Q_T$ with respective initial data $u_0^{(n)}$ and forcing terms $\div F^{(n)}$ for each $n \in \N$. In Proposition \ref{pro:lerayhopfisweakbesov}, we proved that for each $n \in \N$, the suitable weak Leray-Hopf solution $v^{(n)}$ is also a weak Besov solution on $Q_T$ with initial data $u_0^{(n)}$ and forcing term $\div F^{(n)}$. Finally, recall Proposition~\ref{pro:stability} regarding weak-$\ast$ stability of weak Besov solutions. There exists a subsequence (still denoted by $n$) such that $v^{(n)} \to v$ in $L^3_\loc(Q_T)$, where $v$ is a weak Besov solution on $Q_T$ with initial data $u_0$ and forcing term $\div F$.
\end{proof}

\subsection{Weak-strong uniqueness}\label{sec:strongsols}

In this subsection, we are concerned with mild solutions of the Navier-Stokes equations and their relationship to weak Besov solutions.
\begin{definition}[Mild/strong solutions]
\label{strongsoldef}
	Let $T > 0$ and $F \in \cF(Q_T)$. Assume that $u_0 \in \cS'(\R^3)$ is divergence-free and $P_0(u_0,F) \in \cX_T$. (For instance, this is satisfied when $u_0 \in \BMO^{-1}(\R^3) \cup L^\infty(\R^3)$ is divergence free.)
	
A vector field $v \in \cX_T$ is a \emph{mild solution} of the Navier-Stokes equations on $Q_T$ with initial data $u_0$ and forcing term $\div F$ if for a.e. $t \in ]0,T]$, $v$ satisfies the integral equation
	\begin{equation}
		v(\cdot,t) = P_0(u_0,F)(\cdot,t) - B(v,v)(\cdot,t).
		\label{}
	\end{equation}
	A mild solution $v$ on $Q_T$ is a \emph{strong solution} if $v$ is also a weak Besov solution on $Q_T$ with initial data $u_0$ and forcing term $\div F$.

	We say that $v$ is a mild (resp. strong) solution on $Q_\infty$ with initial data $u_0$ and forcing term $\div F$ if for all $T > 0$, $v$ is a mild (resp. strong) solution on $Q_T$ with initial data $u_0$ and forcing term $\div F$.
\end{definition}

Our main goal is the following theorem:

\begin{theorem}[Weak-strong uniqueness]\label{weakstronguniqueness}
Let $0 < T \leq \infty$, $u_0 \in \dot B^{s_p}_{p,\infty}(\R^3)$ be a divergence-free vector field, and $F \in \cF_q(Q_T)$ for some $p \in ]3,\infty[$ and $q \in ]3,p]$.

There exists an absolute constant $\varepsilon_0 > 0$ with the following property. %Suppose that $v\in\cX_{T}$ is a strong solution on $Q_T$ with initial data $u_0$ and forcing term $\div F$.

Suppose that $v \in \cK_\infty(Q_T)$ is a weak Besov solution on $Q_T$ with initial data $u_0$ and forcing term $\div F$.
Moreover, assume that $v$ satisfies
\begin{equation}
	\esssup_{0<t<S} t^{\frac{1}{2}} \norm{v(\cdot,t)}_{L^\infty(\R^3)} < \varepsilon_0
	\label{smallness}
\end{equation}
for some $S \in ]0,T]$.
If $\tilde{v}$ is a weak Besov solution on $Q_T$ with the same initial data and forcing term, then $v \equiv \tilde{v}$ on $Q_T$.
\end{theorem}

Note that Theorem~\ref{weakstronguniqueness} proves weak-strong uniqueness until the maximal existence time of the solution in $\cK_\infty(Q_T)$, not merely on the initial interval $]0,S[$ where the strong solution is small.

We investigate the existence of strong solutions in Proposition~\ref{pro:strongexist}. In particular, strong solutions satisfying~\eqref{smallness} always exist when the initial data and forcing term are sufficiently small. This observation proves Theorem~\ref{weakstrongintro} in the introduction. %, which proves weak-strong uniqueness for small initial data and forcing term,
%may be considered as a corollary of these two results.

\begin{remark}[Alternative proof of small-data-uniqueness]
Let $\norm{u_0}_{\dot B^{s_p}_{p,\infty}(\R^3)} + \norm{F}_{\cF_q(Q_T)} \leq M$. When $M \ll 1$, one may prove the uniqueness for weak Besov solutions $v$ in the following way, which does not rely on the perturbation theory in Proposition~\ref{pro:strongexist}.

Without loss of generality, $T=1$. We will use Proposition~\ref{pro:epsilonregularity} ($\varepsilon$-regularity) with $f=0$, $q_2 = \infty$, and $p_2 = q$. Choose $0 < R \ll 1$ such that $c_0/R < \varepsilon_0$, where $\varepsilon_0$ is the constant in~\eqref{smallness} and $c_0$ from Proposition~\ref{pro:epsilonregularity}.

By using the energy inequality in Remark~\ref{ontheconstant}, estimates on the Picard iterates in Section~\ref{sec:preliminaries}, and Calder{\'o}n-Zygmund estimates for the pressure, one may show that
\begin{equation}
	\label{tildevqckn}
	\sup_{x_0 \in \R^3} \frac{1}{R^2} \int_{1-R^2}^1 \int_{B(x_0,R)} |v|^3 + |q|^{\frac{3}{2}} \, dx \, dt < \frac{\varepsilon_{\rm CKN}}{2}
\end{equation}
when $M \ll 1$. See the proof of Lemma~\ref{bddatinfinity} for similar arguments. Upon further reducing $M \ll 1$,
\begin{equation}
	\label{Fckn}
	\sup_{x_0 \in \R^3} \frac{1}{R^\delta} \norm{F}_{L^\infty_t L^q_x(B(x_0,R) \times ]1-R^2,1[)} < \frac{\varepsilon_{\rm CKN}}{2},
\end{equation}
with $\delta = 2-3/q$. Combining \eqref{tildevqckn}-\eqref{Fckn} and $\varepsilon$-regularity, we obtain
\begin{equation}
	\norm{v(\cdot,1)}_{L^\infty(\R^3)} < \epsilon_0.
\end{equation}
Using a scaling argument, one obtains \eqref{smallness}. Finally, Theorem~\ref{weakstronguniqueness} implies the uniqueness.
 %%%%%%%
\end{remark}

\begin{proof}[Proof of Theorem~\ref{weakstronguniqueness}]
	Let $v$, $\tilde{v}$ be as in the statement of the theorem with the constant $\varepsilon_0 > 0$ in \eqref{smallness} to be determined. Let $k:=k(p)$ and denote $u := v - P_k(u_0,F)$, $\tilde{u} := \tilde{v} - P_k(u_0,F)$, $w := \tilde{u} - u$.

\textit{0. Properties of $w$}.
Observe that $w \in L^\infty_t L^2_x \cap L^2_t \dot H^1_x(Q_S)$
for all finite $S \in ]0,T]$ solves the following Navier-Stokes-type system in the sense of distributions:
	\begin{equation}
		\left.
		\begin{aligned}
		\p_t w - \Delta w + \div w \otimes w + \div v \otimes w + \div w \otimes v &= - \nabla r \\ \div w &= 0
	\end{aligned}
	\right\rbrace \text{ on } Q_T.
		\label{}
	\end{equation}
	Also, $w(\cdot,t)$ is weakly continuous as an $L^2(\R^3)$-valued function. Due to the uniform decay estimate \eqref{improveddecayeq} satisfied by $u,\tilde{u}$, we have
	\begin{equation}
	\sup_{t \in ]0,T[} \frac{\norm{w(\cdot,t)}_{L^2(\R^3)}^2}{t^{\frac{1}{2}}} < \infty.
		\label{wdecay}
	\end{equation}

	\textit{1. Energy estimate for $w$}. Our goal is to demonstrate that $w \equiv 0$ on $Q_T$. Recall that $u,\tilde{u}$ satisfy the global energy inequality \eqref{uglobalenergyineqinitialtime} starting from the initial time, see Corollary \ref{cor:uglobalenergyineqinitialtime}. (In fact, $u$ satisfies the global energy \emph{equality}, compare with Step 1B in Proposition \ref{pro:strongexist}.) As is typical in weak-strong uniqueness arguments, we combine the two energy estimates using the weak-strong identity \eqref{eq:weakstrongid} to obtain the following energy inequality for $w$:
	\begin{equation}
	\norm{w(\cdot,t)}_{L^2(\R^3)}^2 + 2 \int_0^t \int_{\R^3} |\nabla w|^2 \, dx \,dt' \leq 2 \int_0^t \int_{\R^3} v \otimes w : \nabla w \, dx \,dt'
		\label{wweakstrongenergyineq}
	\end{equation}
for all $t \in ]0,T[$. The requirement $v \in \cK_\infty(Q_T)$ together with Proposition~\ref{improveddecayprop} are used to make certain calculations rigorous, in particular, to ensure that the RHS of~\eqref{wweakstrongenergyineq} is finite. See the proof of %Lemma~\ref{lem:wNkenergyineq} %and
Proposition~\ref{pro:vlocalglobal} for a similar argument.

	\textit{2. Showing $w \equiv 0$}. We will conclude with a Gronwall-type argument that crucially makes use of~\eqref{wdecay}. The connection between similar decay properties and weak-strong uniqueness was observed by Dong and Zhang in~\cite{dongzhangweakstrong} and was subsequently used by the second author in~\cite{barkerweakstrong}.

 Manipulating \eqref{wweakstrongenergyineq}, one obtains
\begin{equation}\label{complicatedwweakstrong}
\norm{w(\cdot,t)}^2_{L^2(\R^3)} \leq c \int_0^t \int_{\R^3} |v|^2 |w|^2 \, dx \, ds \leq $$ $$ \leq C_0 \sup_{s \in ]0,t[} [s^{\frac{1}{2}} \norm{v(\cdot,s)}_{L^\infty(\R^3)}]^2 \times t^{\frac{1}{2}} \sup_{s \in ]0,t[} [s^{-\frac{1}{2}} \norm{w(\cdot,s)}_{L^2(\R^3)}^2]
\end{equation}
%\bnorm{\frac{w}{s^\alpha}}_{L^\infty_s L^2_x(Q_t)}^2
for all finite $t \in ]0,T]$. We may choose $\varepsilon_0 := (2C_0)^{-\frac{1}{2}}$ in the statement of the theorem. Recall the assumption \eqref{smallness}, i.e., there exists $S \in ]0,T]$ such that 
\begin{equation}
\sup_{s \in ]0,S[} [s^{\frac{1}{2}} \norm{v(\cdot,s)}_{L^\infty(\R^3)}]^2 < \frac{1}{2C_0}.
\label{}
\end{equation}
Then \eqref{complicatedwweakstrong} gives us
	\begin{equation}
		\frac{\norm{w(\cdot,t)}_{L^2(\R^3)}^2}{t^{\frac{1}{2}}} \leq \frac{1}{2} \sup_{s \in ]0,t[} [s^{-\frac{1}{2}} \norm{w(\cdot,s)}_{L^2(\R^3)}^2]
		\label{}
	\end{equation}
for all $t \in ]0,S[$. Hence, $w \equiv 0$ on $Q_S$. Now, the original energy inequality \eqref{wweakstrongenergyineq} gives us
	\begin{equation}
	\norm{w(\cdot,t)}^2_{L^2(\R^3)} \leq C \norm{v}_{L^\infty(\R^3 \times ]S,t[)}^2 \int_S^t \norm{w}_{L^\infty_t L^2_x(Q_{t'})}^2 \, dx \, dt',
		\label{}
	\end{equation}
for all finite $t \in ]S,T]$. Finally, the standard Gronwall lemma implies that $w \equiv 0$ on $Q_T$. This completes the proof of weak-strong uniqueness.
\end{proof}

Finally, we consider the existence of strong solutions. First, we require some notation. Let $p \in ]3,\infty[$ and $0 < T \leq \infty$. Define $\mathring{\cK}_p(Q_T)$ to be the closed subspace of $\cK_p(Q_T)$ consisting of vector fields $v$ such that
%\begin{equation}
	$\lim_{S \dto 0} \norm{v}_{\cK_p(Q_S)} = 0$
%	\label{}
%\end{equation}
and satisfying the following additional requirement when $T = \infty$:
\begin{equation}
	\lim_{S \upto \infty} \esssup_{t > S} t^{-\frac{s_p}{2}}\norm{v(\cdot,t)}_{L^p(\R^3)} = 0.
	\label{}
\end{equation}
Similarly, define $\mathring{\cX}_T$ to be the closed subspace of $\cX_T$ consisting of vector fields $v$ such that $\lim_{S \dto 0} \norm{v}_{\cX_{S}}=0$ and such that the following additional requirements are satisfied when $T = \infty$. Namely, $v(\cdot,t_{1}+\cdot)\in \cX_{\infty}$ for all $t_{1} > 0$, and
\begin{equation}
	\lim_{t_1 \upto \infty} \|v(\cdot+t_{1})\|_{\cX_{\infty}}=0.
	\label{}
\end{equation}
The space $\mathring{\cY}_T$ is defined analogously for forcing terms.
Recall from Section~\ref{linearestimatessec} that when $q\in ]3,p]$ and $F \in\cF_{q}(Q_{T})$, $L(F)$ belongs to $\cK_{p}(Q_{T}) \cap \cX_{T}$.

Here is our main result concerning the existence of strong solutions:

\begin{pro}[Existence of strong solutions in perturbative regime]\label{pro:strongexist}
Let $0 < T \leq \infty$, $u_0 \in \dot B^{s_p}_{p,\infty}(\R^3)$ be a divergence-free vector field, and $F \in \cF_q(Q_T)$ for some $p \in ]3,\infty[$ and $q \in ]3,p]$. Suppose that $v \in \cK_p(Q_T) \cap \mathring{\cX}_T$ is a mild solution of the Navier-Stokes equations on $Q_T$ with initial data $u_0$ and forcing term $\div F$. There exists a constant $\varepsilon_0 := \varepsilon_0(v,p) > 0$ such that for all divergence-free $\tilde{u_0} \in \dot B^{s_p}_{p,\infty}(\R^3)$ and $\tilde{F} \in \cF_q(Q_T)$ satisfying
	\begin{equation}\label{eq:smallnesscond}
		\norm{P_0(\tilde{u_0},\tilde{F}) - P_0(u_0,F)}_{\cX_T} < \varepsilon_0,
	\end{equation}
	there exists a mild solution $\tilde{v} \in \cX_T$ with initial data $\tilde{u_0}$ and forcing term $\div \tilde{F}$ and such that
	\begin{equation}
		\norm{\tilde{v} - v}_{\cX_T} < 2 \varepsilon_0.
		\label{mildsmallnessest}
	\end{equation}
	In addition, %for $\varepsilon_{0}=\varepsilon_{0}(v,p)>0$ sufficiently small,
	$\tilde{v}$ is unique amongst all mild solutions (with initial data $\tilde{u_0}$ and forcing term $\div \tilde{F}$) that satisfy (\ref{mildsmallnessest}). Moreover, $\tilde{v}$ is a weak Besov solution on $Q_T$ with initial data $\tilde{u_0}$ and forcing term $\div \tilde{F}$ (in particular, it is a strong solution). Finally, $\tilde{v}$ satisfies
	\begin{equation}
		\norm{\tilde{v} - v}_{\cX_T} \leq 2 \norm{P_0(\tilde{u_0},\tilde{F}) - P_0(u_0,F)}_{\cX_T},
		\label{}
	\end{equation}
	\begin{equation}
		\norm{\tilde{v} - v}_{\cK_p(Q_T)} \leq 2 \norm{P_0(\tilde{u_0},\tilde{F}) - P_0(u_0,F)}_{\cK_p(Q_T)}.
		\label{}
	\end{equation}
\end{pro}

The method of proof is well known and goes back to the work~\cite{fujitakato} of Fujita and Kato for initial data in $H^s$, $s \geq 1/2$, as well as Kato's seminal paper~\cite{kato} concerning small-data-global-existence for initial data in $L^3$. Solutions evolving from initial data in critical Besov spaces $\dot B^{-1+\frac{3}{p}}_{p,\infty}$, $p > 3$, were investigated by Cannone~\cite{cannone} and many other authors, see, e.g., the appendix of~\cite{gallagherasymptotics} and the references in~\cite{lemarie2002}. Finally, solutions evolving from $\BMO^{-1}$ initial data were pioneered in~\cite{kochtataru} by H. Koch and D. Tataru.

\begin{proof}%[Proof of Proposition~\ref{pro:strongexist}]
	\textbf{1. Perturbations of the zero solution}.
	Let us consider the case when $u_0$ and $F$ are zero. As mentioned in Section~\ref{linearestimatessec}, there exists a constant $\kappa> 0$ such that for all $U$ and $V$ in $\cX_{T}$,
	\begin{equation}
		\norm{B(U,V)}_{\cX_T} \leq \kappa \norm{U}_{\cX_T} \norm{V}_{\cX_T}.
		\label{kochtatarubd}
		\end{equation}
		Furthermore, it is not difficult to show that there exists $\kappa_{p}>0$ such that
		\begin{equation}
			\norm{B(U,V)}_{\cK_p(Q_T)} \leq \kappa_p \min(\norm{U}_{\cX_T} \norm{V}_{\cK_p(Q_T)}, \norm{U}_{\cK_p(Q_T)} \norm{V}_{\cX_T}),
		\label{l2boundeq}
	\end{equation}
	 for all $U, V \in \cX_T\cap\cK_{p}(Q_{T})$.
	 The constants are independent of $0 <T \leq \infty$. We also use~\eqref{l2boundeq} for $p=2$. Let us write $M \geq \norm{P_0(\tilde{u_0},\tilde{F})}_{\cX_T}$ and $M_p \geq \norm{P_0(\tilde{u_0},\tilde{F})}_{\cK_p(Q_T)}$.

 \emph{1A. Existence in $\cX_T$ and $\cK_p(Q_T)$}.
 %The following method of proof is well known.%, \blue{goes back to \green{Kato and Fujita?}, etc.}
 Suppose that $M < (4 \kappa)^{-1}$. One may verify using~\eqref{kochtatarubd} that the Picard iterates $\tilde{P}_k := P_k(\tilde{u_0},\tilde{F})$ satisfy 
 \begin{equation}
	 \norm{\tilde{P}_k}_{\cX_T} \leq 2M, \quad \norm{\tilde{P}_{k+1} - \tilde{P}_{k}}_{\cX_T} \leq 4 \kappa M \norm{\tilde{P}_k - \tilde{P}_{k-1}}_{\cX_T}
	 \label{}
 \end{equation}
 for all integers $k \geq 0$.
 Hence, the sequence of Picard iterates $(\tilde{P}_k)_{k \geq 0}$ converges to a solution $\tilde{v} \in \cX_T$ of the integral equation
 \begin{equation}
	 \tilde{v}(\cdot,t) = P_0(\tilde{u}_0,\tilde{F}) - B(\tilde{v},\tilde{v}).
	 \label{}
 \end{equation}
 Observe that $\tilde{v}$ is the unique solution satisfying $\norm{\tilde{v}}_{\cX_T} < (2\kappa)^{-1}$. Now suppose that $M < (4\kappa_p)^{-1}$ is additionally satisfied. One verifies using~\eqref{l2boundeq} that for all integers $k \geq 0$, we have
\begin{equation}
	\norm{\tilde{P}_k}_{\cK_p(Q_T)} \leq 2 M_p,
	\label{}
\end{equation}
\begin{equation}
	\norm{\tilde{P}_{k+1} - \tilde{P}_{k}}_{\cK_p(Q_T)} \leq 4 \kappa_p M \norm{\tilde{P}_k - \tilde{P}_{k-1}}_{\cK_p(Q_T)}.
	\label{}
\end{equation}
The sequence $(\tilde{P}_k)_{k \geq 0}$ converges also in the space $\cK_p(Q_T)$, so $\tilde{v}$ additionally belongs to $\cK_p(Q_T)$ and satisfies $\norm{\tilde{v}}_{\cK_p(Q_T)} \leq 2M_p$.
	
%Lemma 4.3 and 4.4 in my paper.
\emph{1B. $\tilde{v}$ is a weak Besov solution}.
Recall from Lemma \ref{forcinglem} that $\tilde{P}_{k(p)+1} - \tilde{P}_{k(p)} \in \cK_2(Q_T)$. Let us further assume that $M < (4\kappa_2)^{-1}$. One may demonstrate using~\eqref{l2boundeq} that for all $k > k(p)$,
	\begin{equation}
		\norm{\tilde{P}_{k+1} - \tilde{P}_k}_{\cK_2(Q_T)} \leq 4M \kappa_2 \norm{\tilde{P}_k - \tilde{P}_{k-1}}_{\cK_2(Q_T)}.
		\label{}
	\end{equation}
	Therefore, $\tilde{u} := \tilde{v} - \tilde{P}_{k(p)}$ belongs to $\cK_2(Q_T)$, since
	\begin{equation}
		\norm{\tilde{u}}_{\cK_2(Q_T)} \leq \norm{\tilde{v} - \tilde{P}_{k(p)}}_{\cK_2(Q_T)} \leq \sum_{k = k(p)}^\infty \norm{\tilde{P}_{k+1} - \tilde{P}_k}_{\cK_2(Q_T)} < \infty.
		\label{}
	\end{equation}
	Let us now demonstrate that $\tilde{u} \in C([0,S];L^2(\R^3)) \cap L^2_t \dot H^1_x(Q_S)$ for all finite $S \in ]0,T]$. In order to show this, we use the identity
	\begin{equation}
		\tilde{u}(\cdot,t) = - B(\tilde{u},\tilde{u})(\cdot,t) - L(\tilde{P}_k \otimes \tilde{u} + \tilde{u} \otimes \tilde{P}_k + \tilde{F}_k)(\cdot,t).
		\label{}
	\end{equation}
	We then conclude using the following facts. Namely,
	\begin{equation}
	\norm{U \otimes V}_{L^{2}(Q_{S})}\leq S^{\frac{1}{4}}\norm{U}_{\cK_{2}(Q_{S})}\norm{V}_{\cX_{S}}
	\end{equation}
	($U\in\cK_{2}(Q_{S})$ and $V\in \cX_{S}$) and the fact that $\tilde{F}_k \in L^2(Q_S)$ for all $k \geq k(p)$, as observed in Lemma \ref{forcinglem}. Note also that since $\tilde{u}\in\cK_{2}(Q_{T})$, we have that 
	\begin{equation}
	\lim_{t\downarrow 0}\norm{\tilde{u}(\cdot,t)}_{L^{2}}=0.
	\end{equation}

It remains to prove the local energy inequality \eqref{vlocalenergyineq} for $\tilde{v}$ with its associated pressure $\tilde{q} := (-\Delta)^{-1} \div \div \tilde{v} \otimes \tilde{v}$. Recall that $\tilde{v} \in L^\infty(\R^3 \times ]\delta,S[)$ for all finite $S \in ]0,T]$ and $\delta \in ]0,S[$. By Calder{\'o}n-Zygmund estimates, $\tilde{q} \in L^\infty_t \BMO_x(\R^3 \times ]\delta,S[)$.
	Using these facts, the local energy inequality for $(\tilde{v},\tilde{q})$ follows by using a mollification argument in the same spirit as in \cite[p. 160-161)]{sereginnotes}. Hence, the proposition is proven with $\varepsilon_0(p) := (8\max(\kappa,\kappa_p,\kappa_2))^{-1}$ in the special case that $u_0$ and $F$ are zero.
% Moreover, by local regularity estimates for the heat equation, $\nabla \tilde{v} \in L^r_tL^q_x(B(R) \times ]\delta,S[)$ for all $r \in [1,\infty[$ and $R > 0$. Therefore, the local energy equality follows from multiplying the Navier-Stokes equations by $\varphi v$ and integrating by parts. Here, $\varphi \in C^\infty_0(Q_T)$ is arbitrary. 

	\textbf{2. Perturbations of general solutions.} Now we consider the proposition in full generality.
	
	\textit{2A. Solving the integral equation.} Our goal is to solve the following integral equation:
	\begin{equation}
		z(\cdot,t) = P_0(\tilde{u_0},\tilde{F})(\cdot,t) - P_0(u_0,F)(\cdot,t) - B(z,z)(\cdot,t) - L_v(z)(\cdot,t),
		\label{zintegraleqn}
	\end{equation}
	where $L_v(z) := B(z,v) + B(v,z)$.
	Then $\tilde{v} := z + v$ will be a mild solution of the Navier-Stokes equations. The integral equation \eqref{zintegraleqn} is equivalent to
	\begin{equation}
		z = (I+L_v)^{-1} [ P_0(\tilde{u_0},\tilde{F}) - P_0(u_0,F) ] - (I+L_v)^{-1} B(z,z),
		\label{zneweq}
	\end{equation} since $I+L_v$ is invertible on $\cX_T$ and $\cK_p(Q_T)$, see Lemma \ref{invertibility}. The existence and uniqueness theory for mild solutions of \eqref{zneweq} in $\cX_T$ and $\cK_p(Q_T)$ is similar to that of Step 1A except that one uses Picard iterates $\bar{P}_k(u_0,F,\tilde{u}_0,\tilde{F})$ defined recursively by
	\begin{equation}
		\bar{P}_0(u_0,F,\tilde{u}_0,\tilde{F}) := (I+L_v)^{-1} [P_0(\tilde{u_0},\tilde{F}) - P_0(u_0,F)],
		\label{}
	\end{equation}
	\begin{equation}
		\bar{P}_k(u_0,F,\tilde{u}_0,\tilde{F}) := \bar{P}_0 - (I+L_v)^{-1} B(\bar{P}_{k-1},\bar{P}_{k-1}), \quad k \in \N.
		\label{}
	\end{equation}
	In addition, we define
	%\begin{equation}
	%	\varepsilon_0(v,p) := (4\max(\norm{(I+L_v)^{-1}}_{\cX_T}^2 \kappa ,\norm{(I+L_v)^{-1}}_{\cK_p(Q_T)}^2 \kappa_p,\kappa,\kappa_p,\kappa_2))^{-1},
	%	\label{}
	%\end{equation}
	\begin{equation}
		\varepsilon_0(v,p) := (8\max(\norm{(I+L_v)^{-1}}^2_{\cX_T} \kappa ,\norm{(I+L_v)^{-1}}^2_{\cK_p(Q_T)} \kappa_p,\kappa,\kappa_p,\kappa_2))^{-1} / 3,
		\label{}
	\end{equation}
	which is less than $\varepsilon_{0}(p)/3$ (where $\varepsilon_{0}(p)$ is as in Step 1) when $v = 0$.
	The proof of existence and uniqueness is not difficult and follows Step 1A, so we will omit it. Let $\tilde{v}$ denote the resulting mild solution of the Navier-Stokes equations.

\textit{2B. $\tilde{u}$ has finite kinetic energy}. Since $v \in \mathring{\cX}_T$ and~\eqref{mildsmallnessest}, there exists $\tilde{T} \in ]0,T[$ such that
	%\begin{equation}
		$\norm{v}_{\cX_{\tilde{T}}} < 2\varepsilon_0(v,p) - \norm{\tilde{v} - v}_{\cX_{{\tilde{T}}}}$. %\leq2\varepsilon_0(v,p) - \norm{\tilde{v} - v}_{\cX_{\tilde{T}}}.
	%	\label{}
	%\end{equation}
		By the triangle inequality $\norm{\tilde{v}}_{\cX_{\tilde{T}}} \leq \norm{v}_{\cX_{\tilde{T}}} + \norm{\tilde{v} - v}_{\cX_{\tilde{T}}}$ and $\varepsilon_0(v,p) < \varepsilon_0(p)/3$, we obtain
 \begin{equation}
 \norm{\tilde{v}}_{\cX_{\tilde{T}}} < 2 \varepsilon_0(v,p)<2\varepsilon_0(p)/3.
  \label{smallnessmildvtilde}
 \end{equation}
Since 
% \begin{equation}
 $P_{0}(\tilde{u_0},\tilde{F})=\tilde{v}(\cdot,t)+B(\tilde{v}, \tilde{v})(\cdot,t)$,
 %\end{equation}
 we infer that
 \begin{equation}
 \norm{P_{0}(\tilde{u_0},\tilde{F})}_{X_{\tilde{T}}}\leq \norm{\tilde{v}}_{X_{\tilde{T}}}+\kappa\norm{\tilde{v}}_{X_{\tilde{T}}}^2.
 \end{equation}
 Using (\ref{smallnessmildvtilde}) and the fact that $4\kappa\varepsilon_{0}(p)<1$, we obtain that 
 \begin{equation}
	 \norm{P_{0}(\tilde{u_0},\tilde{F})}_{X_{\tilde{T}}}<\varepsilon_{0}(p).
	 \label{}
 \end{equation} So we can construct a strong solution (with initial data $\tilde{u_0}$ and forcing term $\textrm{div}\,\tilde{F}$) on $Q_{\tilde{T}}$ according to Step 1. Finally, using (\ref{smallnessmildvtilde}), $\tilde{v}$ agrees on $Q_{\tilde{T}}$ with the mild solution constructed in Step 1, and in particular, $\tilde{u} \in C([0,\tilde{T}];L^2(\R^3)) \cap \dot L^{2}_{t}H^1_x(Q_{\tilde{T}})$.
 
 To show that $\tilde{u}$ has finite energy on $Q_S$ for all finite $S \in ]0,T]$, we appeal to Lemma \ref{strongcontinuation}. Specifically, after translating in time, Lemma~\ref{strongcontinuation} says there exists $S>0$ and a solution $\bar{u} \in L^\infty(\R^3 \times ]\tilde{T},\tilde{T}+S[)$ of the integral equation
 \begin{equation}
	 \bar{u}(\cdot,t) = S(t-\tilde{T}) \tilde{u}(\cdot,\tilde{T}) -\int_{\tilde{T}}^t S(t-s-{\tilde{T}}) \bP \div F_{k}(\cdot,s)ds $$ $$ - \int_{\tilde{T}}^t S(t-s-\tilde{T}) \bP \div [(\bar{u} + P_k) \otimes \bar{u} + P_k \otimes \bar{u}](\cdot,s) \,ds.
 \end{equation}
 on $\R^3 \times ]\tilde{T},\tilde{T}+S[$. Moreover, $\bar{u}$ belongs to the energy space. Since $\bar{v} := P_k + \bar{u}$ is an $L^\infty$ mild solution of the Navier-Stokes equations on $\R^3 \times ]\tilde{T},\tilde{T}+S[$ with initial data $\tilde{v}(\cdot,\tilde{T})$ and forcing term $\div \tilde{F}$, the uniqueness of such solutions implies that $\tilde{v} \equiv \bar{v}$ on $\R^3 \times ]\tilde{T},\tilde{T}+S[$. Hence, $\tilde{u} \equiv \bar{u}$ on the same domain, so we obtain that $\tilde{u} \in C([0,\tilde{T}+S];L^2(\R^3)) \cap L^2_t \dot H^1_x(Q_{\tilde{T}+S})$. We may continue in this fashion as long as the existence time is not shrinking to zero in the iteration. In light of the lower bound \eqref{lowerboundexistencetime} on the existence time in Lemma \ref{strongcontinuation}, we conclude that $\tilde{u} \in C([0,S];L^2(\R^3)) \cap L^2_t \dot H^1_x(Q_S)$ for all finite $S \in ]0,T]$.

 \textit{2C. $\tilde{v}$ is a weak Besov solution}. The local energy inequality for $\tilde{v}$ follows from exactly the same argument as in Step~1B.
\end{proof}

\begin{lemma}[Spectrum of $L_v$]\label{invertibility}
Let $0 < T \leq \infty$ and $p \in ]3,\infty[$. Suppose that $v \in \mathring{\cX}_T$ is divergence free. Then $L_v \: \cX_T \to \cX_T$ and $L_v \: \cK_p(Q_T) \to \cK_p(Q_T)$ defined by $L_v(z) := B(z,v) + B(v,z)$ have spectrum $\{ 0 \}$.
\end{lemma}
\begin{proof}
\textit{1. $L_v$ is not invertible}. Notice that $\nabla L_v f \in L^3_\loc(Q_T)$ for all $f \in \cX_T \cup \cK_{p}(Q_T)$ due to local regularity properties of the Stokes equations.
	Of course, there exists elements $g_{1}\in\cK_{p}(Q_{T})$ and $g_{2}\in \cX_{T}$ with $\nabla g_{i}\notin L^3_\loc(Q_T)$ for $i=1,2$. Clearly, $L_v f_{1} \not= g_{1}$ for all $f_{1} \in \cK_{p}(Q_T)$ and $L_v f_{2} \not= g_{2}$ for all $f_{2} \in \cX_{T}$.
	%Therefore, there exist elements $g \in \cK_p(Q_T)$ and $\nabla g \not\in L^3_\loc(Q_T)$ such that $L_v f \not= g$ for all $f \in \cX_T$. 
	Hence, zero belongs to the spectrum of $L_v$ on $\cX_T$ and $\cK_p(Q_T)$.

	\textit{2. $\lambda I - L_v$ is invertible ($\lambda \in \mathbb{C} \setminus \{ 0 \}$)}. We omit the proof of invertibility, since it is nearly identical to the proof of \cite[Lemma 6]{auscher}, in particular, p. 684-685. The main idea is to solve the linear problem $f - L_v f = g$ on a finite number of small subintervals by a perturbation argument.
\end{proof}

\begin{lemma}[Local continuation with finite energy]\label{strongcontinuation}
Let $0 < T \leq \infty$. Assume that $a \in L^\infty(\R^3) \cap J(\R^3)$, $V \in L^\infty(Q_{T})$ is a divergence-free vector field, and $G \in L^\infty(Q_T) \cap L^2(Q_T)$ with values in $\R^{3\times3}$. There exists a finite time $S \in ]0,T]$, an absolute constant $c_0 > 0$ satisfying
\begin{equation}
	S \geq \frac{c_0}{(1+\norm{P_0(a,G)}_{L^\infty(Q_T)}+\norm{V}_{L^\infty(Q_T)})^2},
	\label{lowerboundexistencetime}
\end{equation} and a solution $u \in L^\infty(Q_S) \cap C([0,S];L^2(\R^3)) \cap L^2_t \dot H^1_x(Q_S)$ of the following integral equation:
	\begin{equation}
		u(\cdot,t) = P_0(a,G)(\cdot,t) - B(u,u)(\cdot,t) - L_V(u)(\cdot,t),
		\label{utheintegraleqn}
	\end{equation}
	 for a.e. $t \in ]0,S[$.
\end{lemma}

We omit the proof of Lemma~\ref{strongcontinuation}, since it follows known perturbation arguments similar to those in Proposition~\ref{pro:strongexist}.

\section{Applications}\label{sec:applications}
\subsection{Blow-up criteria}\label{sec:blowupBesov}

As mentioned before, the second half of this paper focuses on applications of the weak Besov solutions developed in Section~\ref{sec:weakbesovsols}. Let $\mathbb{B}$ denote
%\begin{equation}
%	\mathbb{B}^{-1}_{\infty,\infty}(\R^3) := \left\lbrace f \in \dot B^{-1}_{\infty,\infty}(\R^3) : \div f = 0, \; \lim_{\lambda \dto 0} \lambda f(\lambda \cdot) = 0 \text{ in } \cD'(\R^3) \right\rbrace.
%	\label{}
%\end{equation}
the set of all divergence-free vector fields $f \in \dot B^{-1}_{\infty,\infty}(\R^3)$ satisfying
\begin{equation}
	\lim_{\lambda \dto 0} \lambda f(\lambda (\cdot)) = 0 \text{ in } \cD'(\R^3).
	\label{}
\end{equation}
Note that $\mathbb{B}$ does not contain any non-trivial scale-invariant vector fields. We wish to prove the following theorem:

\begin{theorem}[Blow-up criteria]\label{blowupBesov}
Let $T^* > 0$, $u_0 \in L^\infty(\R^3)$ be a divergence-free vector field, and $F \in L^\infty_t L^q_x(\R^3 \times ]0,T^*[)$ for some $q \in ]3,\infty[$. Suppose that $v \in L^\infty(\R^3 \times ]0,T[)$ is a mild
solution of the Navier-Stokes equations on $\R^3 \times ]0,T[$ with initial data $u_0$ and forcing term $\div F$ for
all $T \in ]0,T^*[$. Let $p \in ]3,\infty[$ and $M > 0$. There exists a constant $\varepsilon := \varepsilon(p,q,M) > 0$ with the following properties:
	\begin{enumerate}[(i)]
	\item Suppose that
	$\norm{v(\cdot,t_1)}_{\dot B^{-1+\frac{3}{p}}_{p,\infty}(\R^3)} \leq M$ for some $t_1 \in ]0,T^*[$.\footnote{In this statement, $v(\cdot,t)$ is  well-defined for each $t \in [0,T^*]$ since $v$ belongs to $C([0,T^*];\cD'(\R^3))$. One way to argue this is as follows. First, it is known that as a mild solution, $v$ belongs to $C([0,T^*[;\cD'(\R^3))$. Second, according to Proposition~\ref{pro:strongsolbounded}, $v$ agrees on $\R^3 \times ]t_1,T^*[$ with a weak Besov solution, and such a solution belongs to $C([t_1,T^*];\cD'(\R^3))$.}
			If also
		\begin{equation}
			\norm{v(\cdot,T^*)}_{\dot B^{-1}_{\infty,\infty}(\R^3)} + \norm{F}_{\cF_q(\R^3 \times ]0,T^*[)} \leq \varepsilon,
	\label{smallatblowuptime}	
\end{equation}
then $v \in L^\infty(\R^3 \times ]0,T^*[)$.
		\item Suppose that there exists a
	sequence of times $t_n \upto T^*$ such that
	\begin{equation}\label{controlonsequence}
		\sup_{n \in \N} \norm{v(\cdot,t_n)}_{\dot B^{-1+\frac{3}{p}}_{p,\infty}(\R^3)} \leq M.
	\end{equation}
	If there exists $x^* \in \R^3$ such that $v(\cdot,T^*)$ satisfies
	\begin{equation}
		{\rm dist}(v(\cdot+x^*,T^*),\mathbb{B}) \leq \varepsilon,
		\label{blowupassumption}
	\end{equation}
	%\begin{equation}
	%	\sqrt{T^* - t_n} v(\sqrt{T^* - t_n} (\cdot - x^*),T^*) \wstar 0 \text{ in } \cD'(\R^3),
	%	\label{blowupassumption}
	%\end{equation}
	where the distance is measured in the $\dot B^{-1}_{\infty,\infty}(\R^3)$ norm,
then $v$ is regular at $(x^*,T^*)$. %\footnote{This means that $v \in L^\infty(B(x^*,R) \times ]T^*-R^2,T^*[)$ for some $R > 0$. If this criterion is not satisfied, we say that $v$ is singular at $(x^*,T^*)$.}
If \eqref{blowupassumption} is satisfied for all $x^* \in \R^3$, then $v \in L^\infty(\R^3 \times ]0,T^*[)$.
	\end{enumerate}
\end{theorem}

%\begin{remark}
%\green{It is known that a $L_{\infty}$-mild solution of the Navier-Stokes equations on $\mathbb{R}^3 \times ]0,T[$ belongs to $C^{\infty}(\mathbb{R}^3\times ]0,T[)$ and is therefore in $C(]0,T[; \mathcal{D}^{'}(\mathbb{R}^3))$. See, for example, (\textbf{cite Koch, Nadirashvili, Seregin and Sverak}). These facts allow us to make sense of the assumption in Theorem \ref{blowupBesov} that $v(\cdot,t_{n})\in \dot{B}^{s_p}_{p,\infty}(\mathbb{R}^3)$.}
%\end{remark}
\noindent Here are a few remarks concerning Theorem~\ref{blowupBesov}:
\begin{enumerate}
\item Let us mention that Escauriaza, Seregin and \v{S}verak's result\footnote{ Specifically, they prove that if a solution belongs to $L^{\infty}_{t}L^{3}_{x}$ then it is regular. See Theorems 1.3-1.4 in \cite{escauriazasereginsverak}.} was shown to hold true with the addition of certain forcing terms by Lemari{\'e}-Rieusset in \cite{lemarie2016} (specifically, Theorem 15.15, p. 527 of~\cite{lemarie2016}).
\item Previously, in \cite{choewolfyangweakl3}, Choe, Wolf, and Yang showed that a weak Leray-Hopf solution satisfying 
	\begin{equation}
		\esssup_{0<t<T*}\norm{v(\cdot,t)}_{L^{3,\infty}}\leq M
		\label{}
	\end{equation} is regular at $(x^{*},T^*)$, under certain additional assumptions on $v(\cdot,T^{*})$, which are similar in spirit to~\eqref{blowupassumption}.
\item The blow-up profiles that do not satisfy our assumption \eqref{blowupassumption} are reminiscent of the initial data conjectured by Guillod and {\u S}ver{\'a}k in~\cite{guillodsverak} to give rise to non-uniqueness. It is plausible to us that there exists a global weak Besov solution~$v$ which is singular at $T^* > 0$, $\sup_{0 < t < T^*} \norm{v(\cdot,t)}_{\dot B^{s_p}_{p,\infty}(\R^3)} < \infty$, and such that uniqueness is lost at the singular time; that is, there exists a different global weak Besov solution $\tilde{v}$ such that $v \equiv \tilde{v}$ on $Q_{T^*}$.
	\end{enumerate}

From the proof of Theorem~\ref{blowupBesov}.(i), we obtain an analogous criterion for weak Besov solutions which we will use to prove Theorem~\ref{blowupBesov}.(ii).
\begin{remark}[Blow-up criterion for weak Besov solutions]\label{blowupcriterionweakbesov}
Let $T^* > 0$, $p \in ]3,\infty[$ and $q \in ]3,p]$. Suppose that $v$ is a weak Besov solution on $Q_{T^*}$ with initial data $u_0 \in \dot B^{s_p}_{p,\infty}(\R^3)$ and forcing term $\div F$ ($F \in \cF_q(Q_{T^*})$). Finally, suppose that $\norm{u_0}_{\dot B^{s_p}_{p,\infty}(\R^3)} \leq M$. There exists a constant $\varepsilon := \varepsilon(p,q,M) > 0$ with the following property. Namely, if~\eqref{smallatblowuptime} is satisfied, then there exists an $\tilde{\varepsilon}\in ]0,T^*[$ such that $v\in L_{\infty}(\mathbb{R}^3 \times ]T^*-\tilde{\varepsilon}, T^*[).$
%$v$ is regular at $(x^*,T^*)$ for all $x^* \in \R^3$.
\end{remark}

Before we prove Theorem~\ref{blowupBesov}, we state three preliminary tools. 
The proofs of Lemma~\ref{bddatinfinity} and Proposition~\ref{backwarduniqueness} will be postponed to the end of the section. We omit the proof of Proposition~\ref{pro:strongsolbounded}, since it follows perturbation arguments similar to those in Proposition~\ref{pro:strongexist}.
%The following lemma provides qualitative control on weak Besov solutions at large distances.
\begin{lemma}[Boundedness for $|x| \gg 1$]\label{bddatinfinity}
Let $T>0$ and $q \in ]3,\infty[$. Let $v$ be a weak Besov solution (based on the $k$th Picard iterate, $0 \leq k \in \Z$) on $Q_T$ with initial data $u_0\in\BMO^{-1}(\R^3)$ and forcing term $\div F$ ($F\in \cF_{q}(Q_{T})$). There exists $R := R(v,k,T,q) > 0$ such that 
\begin{equation}
v \in L^{\infty}((\R^3\setminus B(R))\times ]T/2,T[).
\end{equation}
Moreover, if $F = 0$, we have that for all $0 \leq \alpha,\beta \in \Z$,
\begin{equation}
	\p_t^\alpha \nabla^\beta_x v \in L^{\infty}((\R^3\setminus B(R))\times ]T/2,T[).
	\label{}
\end{equation}
\end{lemma}

%Next, we have a backward uniqueness result for weak Besov solutions:
\begin{pro}[Backward uniqueness]
	\label{backwarduniqueness}
Let $T > 0$ and $v$ be a weak Besov solution on $Q_T$ with initial data $u_0\in \dot B^{s_p}_{p,\infty}(\R^3)$, where $p \in ]3,\infty[$, and zero forcing term. Furthermore, assume that $v(\cdot,T) = 0$. Then $v \equiv 0$ on $Q_T$.
\end{pro}

%Finally, the following proposition ensures that the $L^\infty$ mild solution in Theorem~\ref{blowupBesov} is compatible with the theory of weak Besov solutions.
\begin{pro}[Strong solutions with $u_0 \in L^\infty$]
	\label{pro:strongsolbounded}
Let $0 < T \leq \infty$, $u_0 \in L^\infty(\R^3) \cap \dot B^{s_p}_{p,\infty}(\R^3)$ be a divergence-free vector field, and $F \in L^\infty_t L^q_x(Q_T)$ for some $p \in ]3,\infty[$ and $q \in ]3,p]$. Suppose that $v \in L^\infty(Q_T)$ is a mild solution of the Navier-Stokes equations on $Q_T$ with initial data $u_0$ and forcing term $\div F$. Then $v$ is a weak Besov solution on $Q_T$ with the same initial data and forcing term.
\end{pro}

We now prove Theorem~\ref{blowupBesov} by following the rescaling procedure and backward uniqueness arguments of Seregin in~\cite{sereginh1/2,sereginl3}, see also the subsequent paper \cite{barkersereginhalfspace}. In turn, those arguments are adapted from the seminal work of Escauriaza, Seregin, and {\v S}ver{\'a}k in~\cite{escauriazasereginsverak}.

\begin{proof}[Proof of Theorem~\ref{blowupBesov}]

	\textit{0. Singular points}. 
	 Let us show that to prove Theorem~\ref{blowupBesov}, it is sufficient to investigate potential singularities of $v$.
Let $T^*$, $p$, $q$, $M$, $v$, $u_0$, $F$ be as in the statement of Theorem \ref{blowupBesov}, and suppose that there exists $t_1 \in ]0,T^*[$ such that $\norm{v(\cdot,t_1)}_{\dot B^{s_p}_{p,\infty}(\R^3)} < \infty$.
	We claim that $v \in L^\infty(Q_{T^*})$ provided that $v$ has no singular points at $T^*$. %For contradiction, suppose that \begin{equation}\label{assumptionvnotbdd}
%	v\notin L_{\infty}(Q_{T*}).
%	\end{equation}
	By Proposition \ref{pro:strongsolbounded}, the mild solution $v$ is also a global weak Besov solution on $\R^3\times ]t_1,T^{*}[$ with initial data $v(\cdot,t_1)$ and forcing term $\div F$. By Lemma~\ref{bddatinfinity}, there exists an $R>0$ such that 
	\begin{equation}
	v\in L^{\infty}(B(R)^c \times ]t_1+(T^*-t_1)/2, T^{*}[),
\label{}
\end{equation}
which proves the claim.

\textit{1. Proof of (i)}.
We first discuss a few simplifications. 
By Sobolev embedding for homogeneous Besov spaces, we may assume that $p \geq q$. Next, by the scaling symmetry, we may assume that $T^* = 1$. Finally, we make the following observation that allows us to assume that $t_1 = 0$ in our arguments below. For the moment, suppose that $v$ is a mild solution on $Q_1$ with forcing term $F$, as in the statement of Theorem~\ref{blowupBesov}.(i). Then \eqref{smallatblowuptime} is satisfied, and $\norm{v(\cdot,t_1)}_{\dot B^{s_p}_{p,\infty}(\R^3)} \leq M$ for some $t_1 \in ]0,1[$. Define $\lambda := \sqrt{1-t_1}$ and
	\begin{equation}
		\bar{v}(x,t) := \lambda v(\lambda x, t_1 + \lambda^2 t), \quad \bar{F}(x,t) := \lambda^2 F(\lambda x, t_1 + \lambda^2t).
		\label{}
	\end{equation}
	Then $\bar{v}$ is a mild solution on $Q_1$ with forcing term $\div \bar{F}$ also satisfying the hypotheses of Theorem~\ref{blowupBesov}.(i) with $t_1=0$ and $T^*=1$. Indeed, one may verify that $\norm{\bar{F}}_{\cF_q(Q_1)} \leq \norm{F}_{\cF_q(Q_1)}$ and
	\begin{equation}
		\norm{\bar{v}(\cdot,0)}_{\dot B^{s_p}_{p,\infty}(\R^3)} \leq M, \quad \norm{\bar{v}(\cdot,1)}_{\dot B^{-1}_{\infty,\infty}(\R^3)} + \norm{\bar{F}}_{\cF_q(Q_1)} \leq \varepsilon.
		\label{}
	\end{equation}
	If $v$ is singular at $(0,1)$, then so is $\bar{v}$.

For contradiction, suppose that Theorem~\ref{blowupBesov}.(i) is false. Then there exists
%$p \in ]3,\infty[$, $q \in ]3,p]$, $M > 0$, and 
a sequence $(v^{(n)})_{n \in \N}$ of vector fields on $Q_1$ with the following properties. First, for each $n \in \N$, $v^{(n)} \in L^\infty(Q_T)$ is a mild solution on $Q_T$ with initial data $u_0^{(n)} \in L^\infty(\R^3)$ and forcing term $F^{(n)} \in \cF_q(Q_1)$ for all $T \in ]0,1[$. 
Second,
\begin{equation}
	\sup_{n \in \N} \norm{u_0^{(n)}}_{\dot B^{s_p}_{p,\infty}(\R^3)} \leq M,
	\label{}
\end{equation}
so Proposition~\ref{pro:strongsolbounded} ensures that $v^{(n)}$ is a weak Besov solution on $Q_1$.
Third,
\begin{equation}
	\lim_{n \upto \infty} [\norm{v^{(n)}(\cdot,1)}_{\dot B^{-1}_{\infty,\infty}(\R^3)} + \norm{F^{(n)}}_{\cF_q(Q_{1})}] = 0.
	\label{besovnormgoingtozero}
\end{equation}
Finally, $v^{(n)}$ is singular at $(x^{(n)},1)$ for some $x^{(n)} \in \R^3$ which by the translation symmetry we may assume to be the origin.

By Proposition~\ref{pro:stability} concerning weak-$\ast$ stability, there exists a subsequence of $(v^{(n)})_{n \in \N}$ that converges to a weak Besov solution $\tilde{v}$ on $Q_1$ with initial data $\tilde{u}_0 \in \dot B^{s_p}_{p,\infty}(\R^3)$ and zero forcing term. Specifically,
	\begin{equation}
		u_0^{(n)} \wstar \tilde{u}_0 \text{ in } \dot B^{s_p}_{p,\infty}(\R^3),
		\label{convergencepart1}
	\end{equation}
	\begin{equation}
	v^{(n)} \wstar \tilde{v} \text{ in } (L^{\infty}_{t} L^{2}_x)_\loc (\overline{Q}_{\frac{1}{2},1}), \quad \nabla v^{(n)} \wto \nabla\tilde{v} \text{ in } L^{2}_\loc(\overline{Q}_{\frac{1}{2},1}).
	\end{equation}
	\begin{equation}
		v^{(n)} \to \tilde{v} \text{ in } L^3_\loc(\overline{Q}_{\frac{1}{2},1}), \quad q^{(n)} \wto \tilde{q} \text{ in } L^{\frac{3}{2}}_\loc(\overline{Q}_{\frac{1}{2},1}),
		\label{}
	\end{equation}
	\begin{equation}
		v^{(n)}(\cdot,1) \wstar \tilde{v}(\cdot,1) \text{ in } \cD'(\R^3),
		\label{vconvergetimeslice}
	\end{equation}
	where $q^{(n)},\tilde{q}$ denotes the pressure associated to $v^{(n)},\tilde{v}$, respectively. According to Lemma~\ref{persistenceofsingularity} in the appendix, $\tilde{v}$ also has a singular point at $(0,1)$. Furthermore,~\eqref{besovnormgoingtozero} and~\eqref{vconvergetimeslice} imply that $\tilde{v}(\cdot,1) = 0$. By Proposition~\ref{backwarduniqueness}, $\tilde{v} \equiv 0$ on $Q_1$, which contradicts that $\tilde{v}$ is singular. This completes the proof.

	\textit{2. Proof of (ii)}
	For contradiction, suppose that Theorem \ref{blowupBesov}.(ii) is false. In particular, there exist $T^*$, $p$, $q$, $M$, $v$, $u_0$, $F$, as in the statement of Theorem~\ref{blowupBesov}, satisfying \eqref{controlonsequence}--\eqref{blowupassumption}, where $\varepsilon := \varepsilon(p,q,M) > 0$ is the constant in Remark~\ref{blowupcriterionweakbesov}, and such that $v$ is singular at $(x^*,T^*)$ for some $x^* \in \R^3$. As in Step 1, we may assume that $p \geq q$, $x^* = 0$, and $T^* = 1$.
	
We now zoom in the singularity to obtain a contradiction. For each $n \in \N$, we define $\lambda_n := (1-t_n)^{\frac{1}{2}}$, and for a.e. $(x,t) \in Q_1$,
	\begin{equation}
		v^{(n)}(x,t) := \lambda_n v(\lambda_n x, t_n + \lambda_n^2 t),
		\label{}
	\end{equation}
	\begin{equation}
		F^{(n)}(x,t) := \lambda_n^2 F(\lambda_n x, t_n + \lambda_n^2 t).
		\label{}
	\end{equation}
	Proposition~\ref{pro:strongsolbounded} and \eqref{controlonsequence} imply that $v^{(n)}$ is a weak Besov solution on $Q_1$ with initial data $u_0^{(n)}:= \lambda_n v(\lambda_n \cdot, t_n)$ and forcing term $\div F^{(n)}$. Furthermore,
	\begin{equation}
		\sup_{n \in \N} \norm{u_0^{(n)}}_{\dot B^{s_p}_{p,\infty}(\R^3)} \leq M, \quad
		\lim_{n \upto \infty} \norm{F^{(n)}}_{\cF_q(Q_1)} = 0.
		\label{}
	\end{equation}
	Each velocity field $v^{(n)}$ is singular at $(0,1)$. By Proposition~\ref{pro:stability} regarding weak-$\ast$ stability, %and Remark \ref{pro:stabilityrefinement}
	there exists a divergence-free vector field $\tilde{u_0} \in \dot B^{s_p}_{p,\infty}(\R^3)$ and a subsequence of $(v^{(n)})_{n \in \N}$ converging to a weak Besov solution $\tilde{v}$ on $Q_1$ with initial data $\tilde{u}_0$, see~\eqref{convergencepart1}--\eqref{vconvergetimeslice} in Step~1. Due to Lemma~\ref{persistenceofsingularity} in the appendix, $\tilde{v}$ is singular at $(0,1)$.
	On the other hand, by~\eqref{vconvergetimeslice} and the assumption~\eqref{blowupassumption}, there exists $\Psi \in \dot B^{-1}_{\infty,\infty}(\R^3)$ with $\norm{\Psi}_{\dot B^{-1}_{\infty,\infty}(\R^3)} \leq \varepsilon$ and
	\begin{equation}
		v^{(n)}(\cdot,1) = \lambda_n v(\lambda_n \cdot, 1) \wstar \Psi = \tilde{v}(\cdot,1) \text{ in } \cD'(\R^3).
		\label{zoomintopsi}
	\end{equation}
	Since also $\norm{\tilde{u}_0}_{\dot B^{s_p}_{p,\infty}(\R^3)} \leq M$, Remark~\ref{blowupcriterionweakbesov} implies that $\tilde{v}$ is regular at $(0,1)$. This is the desired contradiction. The proof is complete.
\end{proof}

We now prove the auxiliary results Lemma~\ref{bddatinfinity} and Proposition~\ref{backwarduniqueness}. Let $Q_{S,T} := \R^3 \times ]S,T[$ when $0 < S < T \leq \infty$.
 %, and Proposition~\ref{pro:strongsolbounded}.

\begin{proof}[Proof of Lemma \ref{bddatinfinity}]
	Using the scale-invariance of the Navier-Stokes equations, we may assume without loss of generality that $T=1$. We will use the $\varepsilon$-regularity criterion for suitable weak solutions to control the equation near spatial infinity, see Proposition~\ref{pro:epsilonregularity}.% in the appendix. 
	
For $z=(x,t) \in Q_{1/2,1}$, $r \in ]0,1/2[$, $R_0 > 1/2$, and $|x|\geq R_{0}$, we have that
	 \begin{equation}
		 \frac{1}{r^2} \int_{Q(z,r)} |v|^3 \, dx' \, dt \leq \frac{c}{r^2} \int_{Q(z,r)} |u|^3 \, dx' \, dt + \frac{c}{r^2} \int_{Q(z,r)} |P_k|^3 \, dx' \, dt \leq $$ $$
		 \leq \frac{c}{r^2} \int_{\frac{1}{4}}^{1} \int_{|x|\geq R_{0}-\frac{1}{2}} |u|^3 \, dx' \, dt + cr^3 \norm{P_k}_{L^\infty(Q_{1/4,1})}^3 .
	 \end{equation}
	 Here, $Q(z,r) := B(x,r) \times ]t-r^2,t[$ denotes a parabolic ball.
 Fix $r_0 := r_{0}(\norm{P_k}_{L^\infty(Q_{1/4,1})}, \varepsilon_{\rm CKN}) > 0$ satisfying
 \begin{equation}
 cr_0^3 \norm{P_k}_{L^\infty(Q_{1/4,1})}^3\leq \frac{\varepsilon_{\rm CKN}}{8}.
 \end{equation}
 Since $v$ is a weak Besov solution on $Q_{1}$, we have that $u\in L^{\infty}_{t}L^{2}_{x} \cap L_{t}^{2}\dot{H}^{1}_{x}(Q_1)$. This implies $u\in L^{3}(Q_{1})$. Hence, there exists $R_0 := R_{0}(u, r_0, \epsilon_0) > 1/2$ such that
 \begin{equation}
 \frac{c}{r_0^2} \int_{\frac{1}{4}}^{1} \int_{|x|\geq R_{0}-\frac{1}{2}} |u|^3 \, dx' \, dt\leq \frac{\varepsilon_{\rm CKN}}{8}.
 \end{equation}
 Hence, for $z=(x,t) \in Q_{1/2,1}$ and $|x|\geq R_{0}$, we have that
 \begin{equation}\label{v3small}
		 \frac{1}{r_0^2} \int_{Q(z,r_0)} |v|^3 \, dx' \, dt\leq \frac{\varepsilon_{\rm CKN}}{4}.
		 \end{equation}
		 Similarly, after possibly adjusting $r_0$ and $R_0$, one may obtain that for $z=(x,t) \in Q_{1/2,1}$ and $|x|\geq R_{0}$,
	 \begin{equation}\label{presterm}
		 \frac{1}{r_0^2} \int_{Q(z,r_0)} |q - [q]_{x,r_0}(t)|^{\frac{3}{2}} \, dx' dt' \leq $$ $$ \leq \frac{c}{r_0^2} \int_{Q(z,r_0)} |p|^{\frac{3}{2}} \, dx' \, dt' + \frac{c}{r_0^2} \int_{Q(z,r_0)} |\pi_k - [\pi_k]_{x,r_0}(t)|^{\frac{3}{2}} \, dx' \, dt' \leq $$ $$
		 \leq \frac{c}{r_0^2} \int_{\frac{1}{4}}^{1}\int_{|x'|\geq R_{0}-\frac{1}{2}} |p_{1}|^{\frac{3}{2}} \, dx' \, dt' +\frac{c}{r_0^{\frac{3}{4}}}\Big( \int_{\frac{1}{4}}^{1}\int_{|x'|\geq R_{0}-\frac{1}{2}} |p_{2}|^{{2}} \, dx' \, dt' \Big)^{\frac{3}{4}}+ cr_0^3 \norm{\pi_k}_{L^\infty_t \BMO_x(Q_{1/4,1})}^{\frac{3}{2}} \leq \frac{\varepsilon_{\rm CKN}}{4},
	 \end{equation}
	 where $[q]_{x,r}(t') := |B(x,r)|^{-1} \int_{B(x,r)} q(x',t') \, dx'$. In~(\ref{presterm}), we have used the fact that $p\in L^{\frac{3}{2}}(Q_{1})+L^{2}_{t,\loc}L^{2}_{x}(Q_{1})$ (see the proof of Remark~\ref{pressurermk}).
	 
	  Clearly, there exists $\tilde{q} > 1$ such that
	 \begin{equation}\label{Fspacesloc}
	 F\in L_{t,\loc}^{\tilde{q}} L^{q}_{x}(Q_{1})\,\,\,\textrm{with}\,\,\,\frac{2}{\tilde{q}}+\frac{3}{q}=2-\delta\,\,\,\textrm{and}\,\,\,\delta>0.
	 \end{equation}
	 Since $\tilde{q}$ and $q$ are finite, we may adjust $R_0$ to obtain the following for $z=(x,t) \in Q_{1/2,1}$ and $|x|\geq R_{0}$. Namely,
	 \begin{equation}\label{Ftails}
	 r_0^\delta \|F\|_{L^{\tilde{q}}_{t}L^{q}_{x}(Q(z,r_0))} \leq r_0^\delta \|F\|_{L^{\tilde{q}}_{t}L^{q}_{x}(\mathbb{R}^3\setminus B(R_0-{1}/{2})\times ]1/4,1[)}\leq \frac{\varepsilon_{\rm CKN}}{2}.
	 \end{equation}
	 Using Proposition~\ref{pro:epsilonregularity}, (\ref{v3small}), (\ref{presterm}), and (\ref{Ftails}) gives the desired conclusion.%, provided that $\varepsilon_0$ is sufficiently small.
\end{proof}

\begin{proof}[Proof of Proposition \ref{backwarduniqueness}]
	\textit{0. Properties of $v$}.
	It is sufficient to show that $v\equiv0$ in $ \mathbb{R}^3 \times ]T/2,T[$. A repeated application then gives $v \equiv 0$ on $Q_{T}$.
	
  By rescaling the problem, we may assume that $T = 1$. From Definition~\ref{weaksol}, there exists an integer $k \geq 0$ and $u \in L^\infty_t L^2_x \cap L^2_t \dot H^1_x(Q_1)$ such that
	\begin{equation}
		v = P_k(u_0) + u
		\label{}
	\end{equation}
and satisfies certain additional properties, including the local energy inequality \eqref{vlocalenergyineq}. Observe that $P_k(u_0) \in L^{\infty}_{t}L^p_{x}(\R^3 \times ]\delta,1[)$
and the associated pressure $\pi_k := (-\Delta)^{-1} \div \div P_{k-1} \otimes P_{k-1} \in L^\infty_t L^{\frac{p}{2}}_x(\R^3 \times ]\delta,1[)$
for all $\delta \in ]0,1[$. Also, $u \in L^3(Q_1)$ and $p \in L^{\frac{3}{2}}(Q_1)$. Hence, the velocity field satisfies 
	\begin{equation}
	v \in L^\infty_{t}L^{p}_{x}(\R^3 \times ]\delta,1[) + L^3(\R^3 \times ]\delta,1[), \quad \delta \in ]0,1[.
		\label{}
	\end{equation}
Let $\omega := \curl v$ denote the vorticity.

\textit{1. Suffices to prove $\omega \equiv 0$}. To complete the proof, it is sufficient to prove  that $\omega \equiv 0$ on $Q_{\frac{1}{2},1} := \R^3 \times ]1/2,1[$. In such case, the velocity field $v$ is harmonic, due to the well-known identity $\Delta = \nabla \div - \curl \curl$ for the vector Laplacian. Then $\Delta v(\cdot,t) = 0$ while $v(\cdot,t) \in L^p(\R^3) + L^3(\R^3)$ for almost every $t \in ]1/2,1[$. Finally, the Liouville theorem for entire harmonic functions implies that $v \equiv 0$ on $Q_{\frac{1}{2},1}$ and finishes the proof.

\textit{2. Backward uniqueness: $\omega \equiv 0$ near spatial infinity}. From Lemma \ref{bddatinfinity}, there exists $R := R(v,P_{k}(u_0))>0$ such that for $K=B(R)$, we have $\nabla^{\ell-1}v \in L^\infty(K^c \times ]1/2,1[)$ and $\norm{\nabla^{\ell-1} v}_{L^\infty( (K^c \times ]1/2,1[)} \leq C(r_0,\ell)$ for all $\ell \in \N$. Now recall that the vorticity satisfies the equation
	 \begin{equation}
		 \p_t \omega - \Delta \omega = -\curl (u \cdot \nabla u) = \omega \cdot \nabla u - u \cdot \nabla \omega,
		 \label{}
	 \end{equation}
 from which we obtain that $\p_t \omega \in L^\infty(K^c \times ]1/2,1[)$, and
 \begin{equation}
 |\p_t \omega - \Delta \omega| \leq c(|\omega| + |\nabla \omega|) \text{ on } K^c \times ]1/2,1[.
	 \label{differentialinequality}
 \end{equation}
 Moreover, $\omega(\cdot,1) = 0$. From Theorem 5.1 in \cite{escauriazasereginsverak} concerning backward uniqueness for the differential inequality \eqref{differentialinequality}, we obtain that $\omega \equiv 0$ on $K^c \times ]1/2,1[$.

 \textit{3. Unique continuation: $\omega \equiv 0$ near the spatial origin}. The proof will be complete once we demonstrate that $w \equiv 0$ in $K \times ]1/2,1[$. For the moment, let us take for granted the following claim that we prove in Step 4:
 \begin{itemize}
	 \item[] \emph{Claim}: There exists an open set $G \subset ]0,1[$ such that $\overline{G} = [0,1]$ and $v$ is smooth on $\R^3 \times G$.
	 \end{itemize}
With the claim in hand, let $t_1 \in G \cap ]1/2,1[$ and $x_0 \in K^c$. Let $t_0 \in \R$ be such that $[t_0,t_1] \subset G$. From the smoothness of $v$, we have that $\omega, \p_t \omega, \nabla^2 \omega \in L^2(B(x_0,2R) \times [t_0,t_1])$ for any $R > 0$, and
\begin{equation}
	|\p_t \omega - \Delta \omega| \leq c(|\omega| + |\nabla \omega|) \text{ on } B(x_0,2R) \times ]t_0,t_1[.
	\label{}
\end{equation}
In addition, recall that $\omega \equiv 0$ in a neighborhood of $(x_0,t_1)$.
Hence, by Theorem 4.1 in \cite{escauriazasereginsverak} concerning unique continuation across spatial boundaries, $\omega \equiv 0$ in $B(x_0,R) \times \{t_{1}\}$. Since $t_1 \in G \cap ]1/2,1[$ and $R>0$ are arbitrary, we have that $\omega \equiv 0$ in $\R^3 \times (G \cap ]1/2,1[)$.
	Moreover, by the density of $G$ and weak-$\ast$ continuity of $\omega(\cdot,t)$ on $[0,1]$ in the sense of distributions on $\R^3$, we obtain that $\omega \equiv 0$ on $Q_{\frac{1}{2},1}$, as desired. (Another way to complete Step 3 is to use the spatial analyticity of smooth solutions of the Navier-Stokes equations.)

 \textit{4. Showing $v$ is smooth at generic times}.
 We will now prove the claim from Step 3. Let $\Pi$ denote the set of all $t_0 \in ]0,1$[ such that $u(\cdot,t) \in H^1(\R^3)$ and the global energy inequality \eqref{uglobalenergyineq} is satisfied with initial time $t_0$. The second condition ensures that
	 \begin{equation}
		 \lim_{t \dto t_0} \norm{u(\cdot,t) - u(\cdot,t_0)}_{L^2(\R^3)} = 0.
		 \label{}
	 \end{equation}
	 Notice that $|\Pi| = 1$, and in particular, $\overline{\Pi} = [0,1]$. We will prove that for each $t_0 \in \Pi$, there exists $t_1 := t_1(t_0) \in ]t_0,1[$ such that $v$ is smooth on $]t_0,t_1[$. Then $G := \cup_{t_0 \in \Pi} ]t_0,t_1(t_0)[$ will satisfy the desired properties.
From the above, we see that $u$ is a weak Leray-Hopf solution on $\R^3 \times ]t_0,t_1[$, with initial data $u(\cdot,t_0)\in H^{1}(\R^3)$ and forcing term
	\begin{equation}
		f:= -P_{k}\cdot\nabla u-u\cdot\nabla P_{k}-F_{k}.
		\label{}
	\end{equation}
One can show that $f$ belongs to $L^{2}(\mathbb{R}^3 \times ]t_0,1[)$ for all $k\geq k(p)$, where $k(p) = \lceil \frac{p}{2} \rceil - 2$. By unique solvability 
	results for weak Leray-Hopf solutions,\footnote{See Heywood's paper~\cite[Theorem 2']{heywood1980} and  Sohr's book~\cite[Theorem 1.5.1, p. 276]{sohrelementaryfunctionalanalytic}, for example.} we can conclude the following. Namely, we can find $t_1 := t_{1}(t_0,u,f)>0$ such that
\begin{equation}
u, \nabla u \in L^\infty_t L^2_x(\R^2 \times ]t_0,t_1[) \text{ and } u\in L^{\infty}_t L^6_x(\R^3 \times ]t_0,t_1[).
\end{equation}
Hence,
\begin{equation}
v\in L^{\infty}(\R^3\times ]t_0,t_1[)+L^\infty_t L^{6}_{x}(\R^3 \times ]t_0,t_1[),
\end{equation}
\begin{equation}
\nabla v \in L^{\infty}(\R^3\times ]t_0,t_1[)+L^{2}(\R^3\times ]t_0,t_1[).
\end{equation}
Using known arguments due to Serrin \cite{serrinsmoothness}, we deduce that
\begin{equation}
\nabla^{\ell}v \in L^{\infty}(\R^3\times ]t_0+\varepsilon,t_1[)
\end{equation}
for all $0<\varepsilon<t_{1}-t_0$ and all $0 \leq \ell \in \Z$.
Using known arguments (see Proposition 3.9, p. 160-162 of Seregin's book \cite{sereginnotes}, for example), we can now show that
\begin{equation}
\partial_{t}^{k}\nabla^{\ell}v \in L^{\infty}(\R^3\times ]t_0+\varepsilon,t_1[)
\end{equation}
for all $0<\varepsilon<t_{1}-t_0$ and all $0 \leq k, \ell \in \Z$.
\end{proof}

\subsection{Minimal blow-up initial data}

As discussed in Section~\ref{minimalblowupintro}, global weak Besov solutions provide a convenient framework for investigating minimal blow-up problems, even when local-in-time mild solutions are no longer guaranteed to exist.

Let $\cX$ be a critical space which continuously embeds into $\dot B^{-1+\frac{3}{p}}_{p,\infty}(\R^3)$ for some $p \in ]3,\infty[$. Here, we are using the notion of critical space in Definition~\ref{criticalspacedef}. For each $u_0 \in \cX$, we define
	\begin{itemize}
		\item[]
			$\rho_{\cX}^{u_0} := \sup( \{ 0 \} \cup \lbrace \rho > 0 : $ for all $a \in {\cX}$ satisfying $\norm{a-u_0}_{\cX} \leq \rho$, any global weak Besov solution with initial condition $a$ has no singular points$\rbrace )$.
	\end{itemize}
	%This is illustrated in Figure~\ref{fig:minimalblowup}.
	We also define $\rho_{\cX} := \rho_{\cX}^0$ as in Section~\ref{minimalblowupintro}.\footnote{It is also possible to prove minimal blow-up results with non-zero forcing terms, but the setup is not as convenient owing to the fact that many natural spaces of forcing terms do not embed into each other.}

%Blow-up occurs in the class of global weak Besov solutions with initial data belonging to $\cX$ if and only if $\rho_{\cX} < \infty$.
	
	\begin{remark}
		If $\rho_{\cX} < \infty$, the quantity $\rho_X^{u_0}$ may be zero (for example, when $u_0$ is initial data for a singular global weak Besov solution). It is guaranteed to be non-zero when additionally $u_0 \in \VMO^{-1}(\R^3)$ and there exists a global mild solution $u \in \mathring{\cX}_\infty$ with initial data $u_0$. In this scenario, small perturbations of $u_0$ also give rise to global mild solutions, see Proposition~\ref{pro:strongexist}. For example, Theorem~\ref{weakstrongintro} implies that $\rho_{\cX} > 0$.
		
	\end{remark}

	Here is our main theorem concerning minimal blow-up perturbations of global solutions, which extends Rusin's treatment in \cite{rusinminimalperturbations} for $\dot H^{\frac{1}{2}}$ initial data.\footnote{Rusin's paper~\cite{rusinminimalperturbations} is based on profile decomposition. For minimal blow-up problems, the profile decomposition approach appears to be effective in all dimensions, whereas ours is restricted to dimension $\leq 3$. The reason is that the existence theory and stability properties of suitable weak solutions are currently unknown in dimension $\geq 4$.}

	\begin{theorem}[Minimal blow-up perturbations]
	\label{thm:minimalblowupperturbations}
Let $\cX$ be a critical space which embeds into $\dot B^{-1+\frac{3}{p}}_{p,\infty}(\R^3)$ for some $p \in ]3,\infty[$. Suppose that $u_0 \in \cX$ satisfies the following property:
	\begin{itemize}
		\item[]
			If $(x_n,t_n)_{n \in \N} \subset Q_\infty$ is a sequence such that $t_n \to \infty$, $t_n \to 0$, or $|x_n| \to \infty$, then 
		\begin{equation}
			\sqrt{t_n} u_0 (\sqrt{t_n} (\cdot+x_n)) \wstar 0 \text{ in } \cD'(\R^3).
			\label{}
		\end{equation}
\end{itemize}
Suppose that $\rho_{\cX} < \infty$. Then (at least) one of the following holds:
	\begin{enumerate}[(i)]
		\item There exists a singular global weak Besov solution $v$ with initial data $a \in \cX$ such that $\norm{u_0-a}_{\cX} = \rho_{\cX}^{u_0}$.
		\item There exists a singular global weak Besov solution $v$ with initial data $a \in \cX$ such that $\norm{a}_{\cX} \leq \rho_{\cX}^{u_0}$. Hence, $\rho_{\cX} \leq \rho_{\cX}^{u_0}$.
	\end{enumerate}
	Moreover,
	\begin{enumerate}[(i')]
	\item If (ii) does not hold, then there exists a compact set $K \subset Q_\infty$ and $\varepsilon_0 > 0$ such that for all $\varepsilon \in ]0, \varepsilon_0[$, every singular global weak Besov solution with initial data $a \in \cX$ and $\norm{a-u_0}_{\cX} < \rho_{\cX}^{u_0}+\varepsilon$ has all its singularities in $K$. In this case, the set  $\{ a - u_0 : a \in {\cX} \text{ satisfies (i)} \}$ is sequentially compact in $\cX$ in the topology of distributions on $\R^3$.
	\item If $(i)$ does not hold, then for every compact set $K \subset Q_\infty$, there exists $\varepsilon_0 > 0$ such that for all $\varepsilon \in ]0,\varepsilon_0[$, every singular global weak Besov solution with initial data $a \in \cX$ and satisfying $\norm{a-u_0}_{\cX} < \rho_{\cX}^{u_0} + \varepsilon$ has all its singularities outside $K$. 
	\end{enumerate}
\end{theorem}

\begin{proof}
In order to prove the above theorem, we utilise the weak-$\ast$ stability properties of global Besov solutions, along with arguments related to those contained in~\cite{rusinsverakminimal} and~\cite{rusinminimalperturbations}.

	Assume the hypotheses of the theorem. Suppose $(u_0^{(n)})_{n \in \N} \subset \cX$ and $(v^{(n)})_{n \in \N}$ is an associated sequence of global weak Besov solutions such that
	\begin{equation}
		\norm{u_0^{(n)}-u_0}_{\cX} \dto \rho_{\cX}^{u_0}
		\label{}
	\end{equation}
	and for each $n \in \N$, $v^{(n)}$ has a singular point $(x_n,t_n) \in Q_\infty$.

	Let us consider the following two mutually exclusive cases (which also exhaust all possible cases).

	\emph{Case I}: Suppose the sequence of singular points $(x_n,t_n)_{n \in \N}$ has an accumulation point $(x',t') \in Q_\infty$. By passing to a subsequence\footnote{In this proof, we will not alter our notation when passing to a subsequence.}, %(still indexed by $n$),
we may assume that $u_0^{(n)} \wstar a$ in $\dot B^{s_p}_{p,\infty}(\R^3)$ for some $p \in ]3,\infty[$, the limit $a$ belongs to $\cX$ with norm
	%\begin{equation}
		$\norm{a-u_0}_{\cX} \leq \rho_{\cX}^{u_0}$,
		%\label{}
	%\end{equation}
		and $(v^{(n)})_{n \in \N}$ converge in the sense described in Proposition~\ref{pro:stability} to a global weak Besov solution $v$ with initial data $a$. By Lemma~\ref{persistenceofsingularity}, $v$ has a singularity at $(x',t')$. According to the definition of $\rho_X^{u_0}$, we must have $\norm{a-u_0}_{\cX} = \rho_{\cX}^{u_0}$, which verifies (i).

	\emph{Case II}: Suppose the sequence of singular points $(x_n,t_n)_{n \in \N}$ has no accumulation point in $Q_\infty$. Then there exists a subsequence
	%(still denoted by $n$)
	such that $t_n \to 0$, $t_n \to \infty$, or $|x_n| \to \infty$. We define a sequence of singular global weak Besov solutions $(\tilde{v}^{(n)})_{n \in \N}$ associated to a sequence of initial data $(\tilde{u_0}^{(n)})_{n \in \N}$ by the following translation and rescaling:
	\begin{equation}
		\tilde{v}^{(n)}(x,t) := \sqrt{t_n} v(\sqrt{t_n} (x+x_n),t_n t),
		\label{}
	\end{equation}
	\begin{equation}
		\tilde{u}_0^{(n)}(x) := \sqrt{t_n} u_0^{(n)}(\sqrt{t_n} (x+x_n)).
		\label{}
	\end{equation}
The solutions $\tilde{v}^{(n)}$ have singularities at the spatial origin and time $T=1$. By passing to a further subsequence, we may assume that $\tilde{u}_0^{(n)} \wstar a$ in $\dot B^{s_p}_{p,\infty}(\R^3)$ for some $p \in ]3,\infty[$ and $a \in \cX$ and that $(\tilde{v}^{(n)})_{n \in \N}$ converges to a singular global weak Besov solution with singularity at $(x',t') = (0,1)$. Furthermore, by the assumption on $u_0$ in the statement of Theorem~\ref{thm:minimalblowupperturbations}, we must have
	\begin{equation}
		\tilde{u_0}^{(n)} - \sqrt{t_n} u_0(\sqrt{t_n} (\cdot+x_n)) \wstar a \text{ in } \cD'(\R^3),
		\label{}
	\end{equation}
	so that $a$ satisfies $\norm{a}_{\cX} \leq \rho_{\cX}^{u_0}$. This verifies (ii).

	The proof is completed by noting that if (i) does not hold, then Case I cannot occur for any minimizing sequence of initial data, and similarly, if (ii) does not hold, then Case II cannot occur.
\end{proof}

Corollary~\ref{cor:minimalblowupdataintro} corresponds to the case $u_0 = 0$.

\begin{remark}[Interpretation]
	Suppose that Case II occurs and consider the behavior of the sequence $(x_n,t_n)_{n \in \N}$. One might interpret the situation 
	\begin{equation}
		|x_n| \to \infty \text{ and } \inf_{n \in \N} t_n > \varepsilon
		\label{}
	\end{equation}
	as meaning that $u_0$ has certain nice properties which cause the singularities to disappear at spatial infinity as the initial data approaches the sphere of radius $\rho_{\cX}^{u_0}$ centered on $u_0$, and similarly for $t_n \to \infty$. Since $\rho_{\cX} \leq \rho^{u_0}_{\cX}$, one is tempted to say that, in terms of its ability to ``prevent'' the blow-up of nearby solutions, $u_0$ is at least as good as zero initial data. The case $t_n \to 0$ is perhaps not as clear. It is tempting to interpret the occurrence of singularities very close to the initial time as ill-posedness, but this conflicts with the idea that $u_0$ is at least as good as zero. If $u_0$ is ``singular,'' as is the case for $-1$--homogeneous initial data giving rise to a self-similar solution, then the case $t_n \to 0$ may not be surprising.\footnote{It is interesting to note that scale-invariant solutions in quite general spaces are smooth regardless of the size of their initial data, see~\cite{jiasverakselfsim}.}
	%Remarks on minimal blow-up perturbations in the spaces $\cB_p$.
\end{remark}

%\begin{remark}
%	The sequence $(x_n,t_n)_{n \in \N}$ arising from a minimizing sequence of blow-up initial data is expected to depend heavily on the properties of the initial data $u_0$. However, it is an interesting question to determine the behavior of such a sequence in certain examples, such as in a half-space and with zero initial data, see forthcoming work of T. Pham for results in this direction.
%\end{remark}

%Obtained in \cite{gallagherkochplanchonbesov} using profiles
%As a Corollary, we obtain the

%\begin{theorem}[Minimal blow-up data in subcritical spaces]
%\end{theorem}
%Can mention how to obtain\dots and what about spaces $\dot B^{-\epsilon}_{\infty,\infty}? What's the best way to deal with those? Or simply the space $L^\infty$?

%\begin{remark}[Subcritical minimal blow-up data]
%	\rednote{Very, very brief note on subcritical problems}
%\end{remark}

\subsection{Forward self-similar solutions}\label{sssection}

Finally, we will prove Theorem~\ref{ssexistintro}. As mentioned in Section~\ref{selfsimintro}, we will obtain self-similar solutions evolving from rough initial data as limits of self-similar solutions evolving from $L^{3,\infty}$ initial data. The existence of such solutions was established in \cite{bradshawtsaiII} by Galerkin approximation:
\begin{pro}\textnormal{(\cite[Theorems 1.2--1.3]{bradshawtsaiII})}
	\label{bradshawthm}
	The conclusions of Theorem~\ref{ssexistintro} are valid under the additional assumption that $u_0 \in L^{3,\infty}(\R^3)$.
\end{pro}
While the results in \cite{bradshawtsaiII} are stated for local Leray solutions $v$ satisfying the additional property that $\norm{v(\cdot,t) - S(t)u_0}_{L^2(\R^3)} \leq C t^{\frac{1}{4}}$ for all $t > 0$, it is clear from their construction that $u := v - Su_0$ belongs to the energy class. This fact, combined with the local energy inequality~\eqref{vlocalenergyineq} satisfied by local Leray solutions, implies that the (discretely) self-similar solutions constructed in~\cite{bradshawtsaiII} are global weak Besov solutions.
	
	%We will briefly recall the proof of Proposition \ref{bradshawthm} below in order to reassure the reader that the solutions constructed in \cite{bradshawtsaiII} are indeed weak Besov solutions.

Next, we require the following approximation lemma proven in~\cite{bradshawtsaibesov}.\footnote{One may also approximate by (discretely) self-similar vector fields that are smooth away from the origin.}
	\begin{lemma}\textnormal{(\cite[Lemmas 2.2 and 5.2]{bradshawtsaibesov})}
		\label{approxlem}
	Assume the hypotheses of Theorem~\ref{ssexistintro}. If $u_0$ is $\lambda$-DSS, then there exists a sequence $(u_0^{(n)})_{n \in \N} \subset L^{3,\infty}(\R^3)$ of divergence-free $\lambda$-DSS vector fields such that $u^{(n)}_0 \to u_0$ in $\dot B^{s_p}_{p,\infty}(\R^3)$. If $u_0$ is scale-invariant, then there exists a sequence $(u_0^{(n)})_{n \in \N} \subset L^{3,\infty}(\R^3)$ of scale-invariant divergence-free vector fields such that $u^{(n)}_0 \to u_0$ in $\dot B^{s_p}_{p,\infty}(\R^3)$.
\end{lemma}

With these useful facts in hand, we now prove Theorem~\ref{ssexistintro}.

\begin{proof}[Proof of Theorem~\ref{ssexistintro}]
	Let $u_0$ be the $\lambda$-DSS (resp. scale-invariant) initial data from the statement of Theorem~\ref{ssexistintro}.
	According to Lemma \ref{approxlem}, there exists a sequence $(u_0^{(n)})_{n \in \N} \subset L^{3,\infty}(\R^3)$ of divergence-free $\lambda$-DSS (resp. scale-invariant) vector fields such that $u_0^{(n)} \to u_0$ in $\dot B^{s_p}_{p,\infty}(\R^3)$. By Proposition~\ref{bradshawthm}, given such a sequence, there exist $\lambda$-DSS (resp. scale-invariant) global weak Besov solutions $v^{(n)}$, $n \in \N$, with initial data $u_0^{(n)}$. Proposition~\ref{pro:stability} allows us to extract a subsequence of $(v^{(n)})_{n \in \N}$ converging in $L^3_\loc(\R^3 \times \R_+)$ to a global weak Besov solution $v$ with initial data $u_0$.\footnote{Note that weak-$\ast$ convergence $u^{(n)}_0 \wstar u_0$ in $\dot B^{s_p}_{p,\infty}(\R^3)$ would be sufficient to apply Proposition~\ref{pro:stability}.} Since the approximating solutions $v^{(n)}$ are $\lambda$-DSS (resp. scale-invariant), the limit solution $v$ is $\lambda$-DSS (resp. scale-invariant) as well. This completes the proof.
\end{proof}

\section{Appendix: Splitting lemmas}
\label{sec:appendixsplitting}

In this appendix, we prove several splitting lemmas, including Lemma~\ref{splitlem} from the introduction. %and Lemma~\ref{splitforcing}.

To illustrate the key points, we consider the following simple situation. Let $1 \leq p_0 \leq p \leq p_1 \leq \infty$. For each measurable function $f \: \R \to \R$ with $\norm{f}_{L^p} < \infty$ and $N > 0$, we may write $f = f^N_+ + f^N_-$, where
\begin{equation}
	 f^N_+ := f\chi_{ \{ |f| > N \norm{f}_{L^p}\} }, \quad f^N_- := f\chi_{ \{ |f| \leq N \norm{f}_{L^p} \} }.
\end{equation}
Then, by elementary arguments,
\begin{equation}
\label{fNsplittingest}
	\norm{f^N_+}_{L^{p_0}} \leq N^{1-\frac{p}{p_0}} \norm{f}_{L^p}, \quad \norm{f^N_-}_{L^{p_1}}^{p_1} \leq N^{1-\frac{p}{p_1}} \norm{f}_{L^p}.
\end{equation}
This splitting has the desirable property that it is ``uniform," in the sense that $\norm{f^N_+}_{L^{p_0}}$ can be made small without making $\norm{f^N_-}_{L^{p_1}}$ too large, and vice versa.\footnote{This is also a ``non-dimensionalized" splitting. If one uses
\begin{equation}
	 f^N_+ := f\chi_{ \{ |f| > N \} }, \quad f^L_- := f\chi_{ \{ |f| \leq N \} },
\end{equation}
instead, then $N$ has the same dimensions as $f$, and when $p_1 < \infty$,
\begin{equation}
	\norm{f^L_+}_{L^{p_0}}^{p_0} \leq N^{p_0-p} \norm{f}_{L^p}^p, \quad \norm{f^L_-}_{L^{p_1}}^{p_1} \leq N^{p_1-p} \norm{f}_{L^p}^p.
\end{equation}
} Moreover, it satisfies the obvious estimate
\begin{equation}
	\label{fNpersistencyest}
	\norm{f^N_+}_{L^p}, \norm{f^N_-}_{L^p} \leq \norm{f}_{L^p},
\end{equation}
as can be seen from taking $p_0 = p_1 = p$ in~\eqref{fNsplittingest}. This is known as the ``persistency property."

It is well known that uniform splittings such as~\eqref{fNsplittingest} can be readily obtained from abstract interpolation theory. Our main reference is \cite[Chapters 3-4]{berghlofstrom}. Let $A_0, A_1$ be Banach spaces embedded in a Hausdorff topological vector space $U$. The \emph{$K$-functional} is defined by
\begin{equation}
	K(t,a) = \inf_{a = a_0+a_1} \norm{a_0}_{A_0} + t \norm{a_1}_{A_1}, \quad (t,a) \in \R_+ \times (A_0 + A_1),
\end{equation}
where $a_0, a_1$ are required to belong to $A_0,A_1$, respectively. The function $t \mapsto K(t,a)$ is continuous. For $0 < \theta < 1$ and $1 \leq q \leq \infty$, $K_{\theta,q}$ is defined as the Banach space consisting of all $a \in \Sigma$ satisfying
\begin{equation}
	\norm{a}_{K_{\theta,q}} := \left\lVert t^{-\theta} K(t,a) \right\rVert_{L^q(\R_+,\frac{dt}{t})}  < \infty.
\end{equation}
By definition, for each $\varepsilon > 0$, $a \in K_{\theta,\infty}$, and $N > 0$, there exist $a_0 \in A_0$ and $a_1 \in A_1$ such that
\begin{equation}
	\norm{a_0}_{A_0} \leq N^{-\theta} \big( \norm{a}_{K_{\theta,\infty}} + \varepsilon \big), \quad \norm{a_1}_{A_1} \leq N^{1-\theta} \big( \norm{a}_{K_{\theta,\infty}} + \varepsilon \big).
\end{equation}
This uniform splitting property is analogous to~\eqref{fNsplittingest}. Furthermore, every Banach space $\tilde{A} \subset A_0 + A_1$ satisfying the uniform splitting property~\eqref{fNsplittingest} with $\norm{a}_{\tilde{A}}$ in place of $\norm{a}_{K_{\theta,\infty}}$ must embed continuously into $K_{\theta,\infty}$.\footnote{In contrast, $J_{\theta,1}$ must continuously embed into every Banach space $\tilde{A} \supset A_0 \cap A_1$ satisfying
\begin{equation}
	\norm{a}_{\tilde{A}} \leq C \norm{a_0}_{A_{0}}^{1-\theta} \norm{a_1}_{A_{1}}^{\theta}
\end{equation}
for a constant $C>0$ independent of $a \in \tilde{A}$. See \cite[Theorem 3.5.2 ]{berghlofstrom}, for example.} For example, the spaces $[A_0,A_1]_\theta$ obtained from the complex interpolation method embed continuously into $K_{\theta,\infty}$.

In the sequel, \emph{we are interested in splittings of homogeneous Besov spaces that are uniform and satisfy the persistency property}. Since the persistency property does not appear to obviously follow from the abstract real interpolation theory, our approach will be to construct such splittings explicitly.

To begin, we present the homogeneous Besov spaces as spaces of distributions modulo polynomials.

Let $d,m \in \N$. Let $\cS'$ denote the space of tempered distributions on $\R^d$ with values in $\R^m$. Let $\mathcal{P} \subset \cS'$ denote the closed subspace consisting of polynomials on $\R^d$ with values in $\R^m$. Then $\cS'/\mathcal{P}$ denotes the space of tempered distributions \emph{modulo polynomials} on $\R^d$ with values in $\R^m$.

Recall the operators $\dot \Delta_j$, $j \in \Z$, defined in Section~\ref{sec:functionspaces}.

For $s \in \R$ and $0 < p,q \leq \infty$, we define the homogeneous Besov space $\dot B^{s}_{p,q}$ as the space of tempered distributions (modulo polynomials) $u \in \cS'/\mathcal{P}$ satisfying
\begin{equation}
  \norm{u}_{\dot B^s_{p,q}} := \left\lVert \big( 2^{js} \norm{\dot \Delta_j u}_{L^p} \big)_{j\in\Z} \right\rVert_{\ell^q} < \infty.
\end{equation}
Note that $\sum_{j \geq J} \dot \Delta_j u \to u$ in $\cS'/\mathcal{P}$ as $J \to -\infty$. Moreover, when $s < d/p$ (or $s=d/p$, $q \leq 1$), the sum converges in $\cS'$ and determines a unique tempered distribution $u$. Hence, this definition of $\dot B^s_{p,q}$ is equivalent to the one given in Section~\ref{sec:functionspaces}.

Let $R$ be the \emph{retraction} from $\cS'/\mathcal{P}$ to the space of $\cS'$-valued sequences over $\Z$:
\begin{equation}
  Ru = \big(\dot \Delta_j u \big)_{j \in \Z}.
\end{equation}
Let $S$ be the \emph{co-retraction} from the space of $\cS'$-valued sequences over $\Z$ to $\cS'/\mathcal{P}$:
\begin{equation}
  S (u_j)_{j \in \Z} = \sum_{j \in \Z} \tilde{\Delta}_j u_j,
\end{equation}
where $\tilde{\Delta}_j = \dot \Delta_{j-1} + \dot \Delta_{j} + \dot \Delta_{j+1}$. Then $SR = I$ is the identity map on $\cS'/\mathcal{P}$. Let $\ell^s_q L^p$ denote the space of $L^p$-valued sequences $(u_j)_{j \in \Z}$ over $Z$ satisfying
\begin{equation}
  \left\lVert (u_j)_{j \in \Z} \right\rVert_{\ell^s_q L^p} := \left\lVert \big( 2^{js} \norm{u_j}_{L^p} \big)_{j \in \Z} \right\rVert_{\ell^q} < \infty.
\end{equation} Then $R \: \dot B^s_{p,q} \to \ell^s_q L^p$ and $S \: \ell^s_q L^p \to \dot B^s_{p,q}$ are continuous maps.

%We now splitting properties in $L^p$ measure space $\Omega$ into Banach space $X$. More general, $A$ between $\tilde{A}$ and $\bar{A}$, with persistency property (or without -- any space with the splitting). Repress the notation

The retraction/co-retraction technology allows one to ``transfer" splittings in sequence spaces to splittings in Besov spaces. Therefore, we begin with two splitting lemmas in sequence spaces:

\begin{lemma}[Horizontal splitting]
\label{horizontalsplittinglem}
Let $s,s_0,s_1 \in \R$ be distinct real numbers and $p \in ]0,\infty]$. For all $u \in \ell^{s}_{\infty} L^p$ and $K > 0$, there exist $f^K \in \ell^{s_0}_{1} L^p$ and $g^K \in \ell^{s_1}_{1} L^p$
 such that
\begin{equation}
  u=f^K + g^K,
\end{equation}
\begin{equation}
  \norm{f^K}_{\ell^{s_0}_1 L^p}\leq \frac{K^{s_0-s}}{1-2^{|s_0-s|}} \norm{u}_{\ell^s_\infty L^p},
\end{equation}
\begin{equation}
  \norm{g^K}_{\ell^{s_1}_1 L^p} \leq \frac{K^{s_1-s}}{1-2^{|s_1-s|}} \norm{u}_{\ell^s_\infty L^p}.
\end{equation}
Moreover,
\begin{equation}
\label{horizontalpersistency}
  \norm{f^{K}}_{\ell^s_\infty L^p}, \norm{g^{K}}_{l^s_\infty L^p} \leq \norm{u}_{\ell^s_\infty L^p}.
\end{equation}
\end{lemma}
\begin{proof}
Without loss of generality, we assume that $s_0 < s < s_1$. For $\kappa = \lfloor \log_2 K \rfloor$, we define
\begin{equation}
\label{horizontalsplitdef}
  f^K_j = \begin{cases} u_j & j > \kappa \\ 0 & \text{otherwise} \end{cases}, \quad g^K = u - f^K.
\end{equation}
Hence,
\begin{equation}
  \sum_{j \in \Z} 2^{j s_0} \norm{f^K_j}_{L^p} \leq \sum_{j > \kappa} 2^{j(s_0-s)} \times \norm{u}_{\ell^s_\infty L^p} \leq \frac{K^{s_0 - s}}{1-2^{s-s_0}} \norm{u}_{\ell^s_\infty L^p},
\end{equation}
\begin{equation}
   \sum_{j \in \Z} 2^{j s_1} \norm{g^K_j}_{L^p} \leq  \sum_{j \leq \kappa} 2^{j(s_1-s)} \times \norm{u}_{\ell^s_\infty L^p} \leq \frac{K^{s_1-s}}{1-2^{s_1-s}} \norm{u}_{\ell^s_\infty L^p},
\end{equation}
and the persistency property~\eqref{horizontalpersistency} is valid due to~\eqref{horizontalsplitdef}.
\end{proof}

\begin{lemma}[Diagonal splitting]
\label{diagonalsplittinglem}
Let $\sigma,\tilde{s},\bar{s} \in \R$, $0 < \tilde{p} < p < \bar{p} \leq \infty$, and $q,\tilde{q},\bar{q} \in (0,\infty]$ such that $(\sigma,p,q)$ belongs to the open segment connecting $(\tilde{s},\tilde{p},\tilde{q})$ and $(\bar{s},\bar{p},\bar{q})$.
Then for all $g \in \ell^{\sigma}_{q} L^p$ and $N > 0$, there exist $\tilde{g}^{N} \in \ell^{\tilde{s}}_{\tilde{q}} L^{\tilde{p}}$ and $\bar{g}^{N} \in \ell^{\bar{s}}_{\bar{q}} L^{\bar{p}}$ such that
\begin{equation}
  g = \tilde{g}^N + \bar{g}^N,
\end{equation}
\begin{equation}
\label{diagest1}
  \norm{\tilde{g}^N}_{\ell^{\tilde{s}}_{\tilde{q}} L^{\tilde{p}}} \leq N^{1-\frac{p}{\tilde{p}}} \norm{g}_{\ell^{\sigma}_{q} L^p}
\end{equation}
\begin{equation}
\label{diagest2}
  \norm{\bar{g}^N}_{\ell^{\bar{s}}_{\bar{q}} L^{\bar{p}}} \leq N^{1-\frac{p}{\bar{p}}} \norm{g}_{\ell^{\sigma}_{q} L^p}
\end{equation}
Moreover, for all $j \in \Z$,
\begin{equation}
\label{diagpersistency}
  \norm{\tilde{g}_j^N}_{L^p}, \norm{\bar{g}_j^N}_{L^p} \leq \norm{g_j}_{L^p}.
\end{equation}
\end{lemma}
\begin{proof}
There exists $\theta \in ]0,1[$ such that
\begin{equation}
  \sigma = \theta \tilde{s} + (1-\theta) \bar{s}, \quad \frac{1}{p} = \frac{\theta}{\tilde{p}} + \frac{1-\theta}{\bar{p}}, \quad \frac{1}{q} = \frac{\theta}{\tilde{q}} + \frac{1-\theta}{\bar{q}}.
\end{equation}
For $c, \lambda_j > 0$ to be specified below, we define
\begin{equation}
\label{diaggdef}
  \tilde{g}^N_j = g_j \chi_{\{|g_j| > cN \lambda_j \norm{g_j}_{L^p}\}}, \quad \bar{g}^N_j = g_j - \tilde{g}^N_j,
\end{equation}
which gives
\begin{equation}
  \norm{\tilde{g}^N_j}_{L^{\tilde{p}}} \leq (cN)^{1-\frac{p}{\tilde{p}}} \lambda_j^{1-\frac{p}{\tilde{p}}} \norm{g_j}_{L^p}, \quad
  \norm{\bar{g}^N_j}_{L^{\bar{p}}} \leq (cN)^{1-\frac{p}{\bar{p}}} \lambda_j^{1-\frac{p}{\bar{p}}} \norm{g_j}_{L^p}.
\end{equation}
By elementary manipulations,
\begin{equation}
\label{splittingrel1}
  2^{{j}\tilde{s}\tilde{q}} \norm{\tilde{g}^N_j}_{L^{\tilde{p}}}^{\tilde{q}} \leq (cN)^{(1-\frac{p}{\tilde{p}})\tilde{q}} \lambda_j^{(1-\frac{p}{\tilde{p}})\tilde{q}} 2^{{j}\tilde{s}\tilde{q}} \norm{g_j}_{L^p}^{\tilde{q}}
\end{equation}
\begin{equation}
\label{splittingrel2}
  2^{{j}\bar{s}\bar{q}} \norm{\bar{g}^N_j}_{L^{\bar{p}}}^{\bar{q}} \leq (cN)^{(1-\frac{p}{\bar{p}})\bar{q}} \lambda_j^{(1-\frac{p}{\bar{p}})\bar{q}} 2^{{j}\bar{s}\bar{q}}  \norm{g_j}_{L^p}^{\bar{q}}
\end{equation}
Let us only deal with values $j \in \Z$ such that $\norm{g_j}_{L^p} > 0$. We define $\lambda_j > 0$ by the following (equivalent) equations:
\begin{equation}
\label{splittingrel3}
  \lambda_j^{(1-\frac{p}{\tilde{p}})\tilde{q}} = 2^{j(\sigma q-\tilde{s}\tilde{q})} \norm{g_j}^{q-\tilde{q}}, \quad
  \lambda_j^{(1-\frac{p}{\bar{p}})\bar{q}} = 2^{j(\sigma q-\bar{s}\bar{q})} \norm{g_j}^{q-\bar{q}},
\end{equation}
whose equivalence will be justified below. Substituting~\eqref{splittingrel3} into \eqref{splittingrel1}-\eqref{splittingrel2} and summing over $j \in \Z$ gives
\begin{equation}
  \norm{\tilde{g}^N}_{\ell^{\tilde{s}}_{\tilde{q}} L^{\tilde{p}}} \leq (cN)^{1-\frac{p}{\tilde{p}}} \norm{g}_{\ell^{\sigma}_q L^p}^{\frac{q}{\tilde{q}}}, \quad
  \norm{\bar{g}^N}_{\ell^{\bar{s}}_{\bar{q}} L^{\bar{p}}} \leq (cN)^{1-\frac{p}{\bar{p}}} \norm{g}_{\ell^{\sigma}_q L^p}^{\frac{q}{\bar{q}}}
\end{equation}
Then choose $c>0$ satisfying the following (equivalent) equations
\begin{equation}
  c^{1-\frac{p}{\bar{p}}} = \norm{g}_{\ell^{\sigma}_q L^p}^{1-\frac{q}{\bar{q}}}, \quad c^{1-\frac{p}{\tilde{p}}} = \norm{g}_{\ell^{\sigma}_q L^p}^{1-\frac{q}{\tilde{q}}},
\end{equation}
and the estimates~\eqref{diagest1}-\eqref{diagest2} are proven. The persistency property~\eqref{diagpersistency} follows from the definition~\eqref{diaggdef} of $\tilde{g}^N$ and $\bar{g}^N$.

Finally, we argue that the two equations in~\eqref{splittingrel3} are equivalent. By comparing exponents, they will be equivalent as long as
\begin{equation}
\label{diagsplitequiv}
	\frac{\sigma\frac{q}{\tilde{q}} - \tilde{s}}{1-\frac{p}{\tilde{p}}} = \frac{\sigma\frac{q}{\bar{q}} - \bar{s}}{1-\frac{p}{\bar{p}}} \quad \text{ and } \quad
	\frac{\frac{q}{\tilde{q}}-1}{\frac{q}{\bar{q}}-1} = \frac{\frac{p}{\tilde{p}}-1}{\frac{p}{\bar{p}}-1}.
\end{equation}
Let us assume that $\tilde{q} \neq \bar{q}$ and $\tilde{s} \neq \bar{s}$ (otherwise, the proof simplifies). Thus,
\begin{equation}
	\theta = \frac{\sigma-\bar{s}}{\tilde{s}-\bar{s}} = \frac{ \frac{1}{p} - \frac{1}{\bar{p}}}{ \frac{1}{\tilde{p}} - \frac{1}{\bar{p}}} = \frac{ \frac{1}{q} - \frac{1}{\bar{q}}}{ \frac{1}{\tilde{q}} - \frac{1}{\bar{q}}}.
\end{equation}
The second equation in \eqref{diagsplitequiv} readily reduces to $(\theta-1)/\theta$ on each side, so let us only deal with the first equation, which is equivalent to
\begin{equation}
\label{diagsplitthingtobeverified}
	\frac{\sigma\frac{q}{\tilde{q}} - \tilde{s}}{\sigma\frac{q}{\bar{q}} - \bar{s}} = \frac{\theta-1}{\theta}.
\end{equation}
Substituting $\sigma = \theta \tilde{s} + (1-\theta) \bar{s}$ and employing the relationship $\frac{1}{q} = \frac{\theta}{\tilde{q}} + \frac{1-\theta}{\bar{q}}$ in the numerator and denominator verifies
\eqref{diagsplitthingtobeverified}. The proof is complete.
\end{proof}

Lemma~\ref{horizontalsplittinglem} is related to the characterization $\ell^s_\infty L^p = (\ell^{s_1}_1 L^p,\ell^{s_0}_1 L^p)_{\theta,\infty}$ (real interpolation, with $s = \theta s_0 + (1-\theta) s_1$), while Lemma~\ref{diagonalsplittinglem} is related to the characterization of $\ell^{\sigma}_q L^p = [ \ell^{\bar{s}}_{\bar{q}} L^{\bar{p}}, \ell^{\tilde{s}}_{\tilde{q}} L^{\tilde{p}}  ]_{\theta}$ (complex interpolation, with $\theta$ as in the proof). See \cite[Chapter 5]{berghlofstrom}.
%\textbf{Remark: Statement is also true when $p$'s are equal. If the $q$'s are not equal, then do the splitting like... certain amount above, certain amount below (to be checked). If the $q$'s are also equal, then do the splitting as in the first lemma.}

\begin{remark}[Generalizations]\label{rmkgeneraldiagonal}
\begin{enumerate}
\item In Lemma~\ref{diagonalsplittinglem}, one may replace $L^p(\R^d;\R^m)$ by $L^p(\Omega;X)$, where $\Omega$ is a measure space and $X$ is a Banach space, at no additional cost.
\item In the proof of Lemma~\ref{splitforcing}, we will require an analogous diagonal splitting lemma in spaces of sequences (over $\Z_{j \leq J}$, for fixed $J \in \Z$) with values in $L^\infty_t L^p_x$. One may verify that same proof works with almost no alteration.
\item One could replace the $L^p$ spaces with Banach spaces $X_{\alpha}$ satisfying analogous splitting properties. In this way, one could iterate the proof of Lemma~\ref{diagonalsplittinglem} to handle a variety of mixed spaces combining $L^p$ and $\ell^s_q$. One could also allow the function spaces to depend on $j \in \Z$.
\end{enumerate}
\end{remark}

Combining the previous two lemmas, we obtain
\begin{pro}[Non-diagonal splitting]
\label{nondiagonalsplittinglem}
Let $s,\tilde{s},\bar{s} \in \R$ and $0 < \tilde{p} < p < \bar{p} \leq \infty$ such that $(s,1/p)$, $(\tilde{s},1/\tilde{p})$, and $(\bar{s},1/\bar{p})$ are not colinear. There exists a unique $s_1 \in \R$ such that $(s_1,1/p)$ belongs to the closed segment connecting $(\tilde{s},1/\tilde{p})$ and $(\bar{s},1/\bar{p})$. Let $s_0 \in \R$ such that $s$ belongs to the open segment connecting $s_0$ and $s_1$.

For all $u \in \ell^s_\infty L^p$ and $K,N > 0$, there exist $f^K \in \ell^{s_0}_{1} L^p$, $\tilde{g}^{K,N} \in \ell^{\tilde{s}}_1 L^{\tilde{p}}$, and $\bar{g}^{K,N} \in \ell^{\bar{s}}_1 L^{\bar{p}}$ such that
\begin{equation}
  u = f^K + \tilde{g}^{K,N} + \bar{g}^{K,N},
\end{equation}
\begin{equation}
  \norm{f^K}_{\ell^{s_0}_1 L^p} \leq \frac{K^{s_0-s}}{1-2^{|s_0-s|}} \norm{u}_{\ell^s_\infty L^p}
\end{equation}
\begin{equation}
  \norm{\tilde{g}^{K,N}}_{\ell^{\tilde{s}}_1 L^{\tilde{p}}} \leq \frac{K^{s_1-s}}{1-2^{|s_1-s|}} N^{1-\frac{p}{\tilde{p}}} \norm{u}_{\ell^s_\infty L^p}
\end{equation}
\begin{equation}
  \norm{\bar{g}^{K,N}}_{\ell^{\bar{s}}_1 L^{\bar{p}}} \leq \frac{K^{s_1-s}}{1-2^{|s_1-s|}} N^{1-\frac{p}{\bar{p}}} \norm{u}_{\ell^s_\infty L^p}
\end{equation}
Moreover, for all $j \in \Z$,
\begin{equation}
\label{nondiagonalpersistency}
  \norm{f^K_j}_{L^p}, \norm{\tilde{g}^{K,N}_j}_{L^p}, \norm{\bar{g}^{K,N}_j}_{L^p} \leq \norm{u_j}_{L^p}.
\end{equation}
%Furthermore, $\{ j \in \Z : f^K_j \neq 0 \}$ is disjoint from $\{ j \in \Z : \tilde{g}^{N,K}_j \neq 0 \text{ or } \bar{g}^{N,K}_j \neq 0 \}$.
\end{pro}
\begin{proof}
First, apply Lemma~\ref{horizontalsplittinglem} to obtain $u = f^K + g^K$. Next, apply Lemma~\ref{diagonalsplittinglem} with $\sigma=s_1$ and $q=\tilde{q}=\bar{q}=1$ to obtain $g^K = \tilde{g}^{K,N} + \bar{g}^{K,N}$.
\end{proof}
Let all indices be as in Proposition~\ref{nondiagonalsplittinglem}.
Note that there exist $\theta, \phi\in ]0,1[$ such that
\begin{equation}
	\ell^{{s}}_\infty L^{{p}}=([ \ell^{\bar{s}}_{1} L^{\bar{p}},\ell^{\tilde{s}}_{1} L^{\tilde{p}}]_{\theta}, \ell^{s_{0}}_{1} L^{{p}})_{\phi,\infty}.
\end{equation} Furthermore,
\begin{equation}
	[ \ell^{\bar{s}}_{1} L^{\bar{p}},\ell^{\tilde{s}}_{1} L^{\tilde{p}}]_{\theta}\hookrightarrow (\ell^{\bar{s}}_{1} L^{\bar{p}},\ell^{\tilde{s}}_{1} L^{\tilde{p}})_{\theta,\infty}.
\end{equation}
Thus, Proposition~\ref{nondiagonalsplittinglem} (without the persistency property) can be obtained via the abstract interpolation theory.

%Note that since Lemma~\ref{horizontalsplittinglem} and Lemma~\ref{diagonalsplittinglem} may be obtained (without the persistency property) via the abstract interpolation theory, so may Proposition~\ref{nondiagonalsplittinglem}.

\begin{remark}[Non-diagonal splitting, Besov version]
\label{nondiagonalsplittingrmk}
The non-diagonal splitting in Lemma~\ref{nondiagonalsplittinglem} is applicable to Besov functions in the following way. Let $s,\tilde{s},\bar{s},p,\tilde{p},\bar{p}$ be as in Lemma~\ref{nondiagonalsplittinglem}. Given $u \in \dot B^{s}_{p,\infty}$ and $K,N > 0$, we apply Lemma~\ref{nondiagonalsplittinglem} to the retraction $Ru$, which belongs to $\ell^s_\infty L^p$, and obtain $Ru = f^K + \tilde{g}^{K,N} + \bar{g}^{K,N}$, satisfying the estimates in Lemma~\ref{nondiagonalsplittinglem} with $Ru$ replacing $u$. The Besov splitting is obtained by applying the co-retraction: $u = SRu = S(f^K) + S(\tilde{g}^{K,N}) + S(\bar{g}^{K,N})$. For the moment, we abuse notation by writing $f^K$ instead of $S(f^K)$, etc. In summary, we have
\begin{equation}
  u = f^K + \tilde{g}^{K,N} + \bar{g}^{K,N},
\end{equation}
\begin{equation}
  \norm{f^K}_{\dot B^{s_0}_{p,1}} \leq C \frac{K^{s_0-s}}{1-2^{|s_0-s|}} \norm{u}_{\dot B^s_{p,\infty}}
\end{equation}
\begin{equation}
  \norm{\tilde{g}^{K,N}}_{\dot B^{\tilde{s}}_{\tilde{p},1}} \leq C\frac{K^{s_1-s}}{1-2^{|s_1-s|}} N^{1-\frac{p}{\tilde{p}}} \norm{u}_{\dot B^s_{p,\infty}}
\end{equation}
\begin{equation}
  \norm{\bar{g}^{K,N}}_{\dot B^{\bar{s}}_{\bar{p},1}} \leq C\frac{K^{s_1-s}}{1-2^{|s_1-s|}} N^{1-\frac{p}{\bar{p}}} \norm{u}_{\dot B^s_{p,\infty}},
\end{equation}
and finally, the persistency property,
\begin{equation}
  \norm{f^K}_{\dot B^s_{p,\infty}}, \norm{\tilde{g}^{K,N}}_{\dot B^s_{p,\infty}}, \norm{\bar{g}^{K,N}}_{\dot B^s_{p,\infty}} \leq C\norm{u}_{\dot B^s_{p,\infty}}.
\end{equation}
The constant $C > 0$ only appears when estimating the co-retraction operator and depends continuously on the parameters $s,\tilde{s},\bar{s}$.
\end{remark}

The following splitting is obtained by combining Remark~\ref{nondiagonalsplittingrmk} and Sobolev embedding.
%Replace with 0?
\begin{pro}[Splittings in Besov spaces]
\label{besovsplittingpro}
Let $1 \leq \tilde{p} < p < \bar{p} \leq \infty$, and $s, \tilde{s} \in \R$.

Let $\alpha$ denote the line through $(s,1/p)$ and $(\tilde{s},1/\tilde{p})$, $\beta$ the line through $(s,1/p)$ of slope $1/d$, and $\gamma$ the horizontal line through the origin. Assume $\alpha \neq \beta$. Let $D \subset \R^2$ denote the interior of the compact region enclosed by $\alpha$, $\beta$, and $\gamma$.% Let $E = D \setminus (\alpha \cup \beta)$. Technically, can get guys with \infty, as well

Let $\bar{s} \in \R$ such that $(\bar{s},1/\bar{p}) \in D$. %Mention omit dependence on $d,m$ in the norms
Let $s_1 \in \R$ be the unique value such that $(s_1,1/p)$ belongs to the open segment connecting $(\tilde{s},1/\tilde{p})$ and $(\bar{s},1/\bar{p})$. Let $s_0 = \bar{s} + \frac{d}{p} - \frac{d}{\bar{p}}$.

There exists a constant $C>0$ depending continuously on the above parameters and satisfying the following properties:

For all $u \in \dot B^{s}_{p,\infty}$ and $N>0$, there exist $\tilde{u}^N \in \dot B^{\tilde{s}}_{\tilde{p},1}$ and $\bar{u}^N \in \dot B^{\bar{s}}_{\bar{p},1}$ such that
\begin{equation}
  u = \tilde{u}^N + \bar{u}^N,
\end{equation}
\begin{equation}
  \norm{\tilde{u}^N}_{\dot B^{\tilde{s}}_{\tilde{p},1}} \leq C N^{\frac{s_1-s}{s_0-s_1}  (1-\frac{p}{\bar{p}}) + (1-\frac{p}{\tilde{p}})} \norm{u}_{\dot B^{s}_{p,\infty}},
\end{equation}
\begin{equation}
  \norm{\bar{u}^N}_{\dot B^{\bar{s}}_{\bar{p},1}} \leq C N^{\frac{s_0-s}{s_0-s_1} (1-\frac{p}{\bar{p}})} \norm{u}_{\dot B^{s}_{p,\infty}}.
\end{equation}
Moreover,
\begin{equation}
\label{besovsplittingpersistencyproperty}
  \norm{\tilde{u}^N}_{\dot B^{s}_{p,\infty}}, \norm{\bar{u}^N}_{\dot B^{s}_{p,\infty}} \leq C \norm{u}_{\dot B^{s}_{p,\infty}}.
\end{equation}
\end{pro}

\setcounter{figure}{\value{theorem}}
%%%%%%%%%%%%%%%%%%%%%%%%%%%%%%%%%%%%%%%%%%%%%%%%%%%%%%%%%%%%%%%%%%%%%%%%%%%%
	\begin{figure}[h!]
		\centering
		\includegraphics[width=4in]{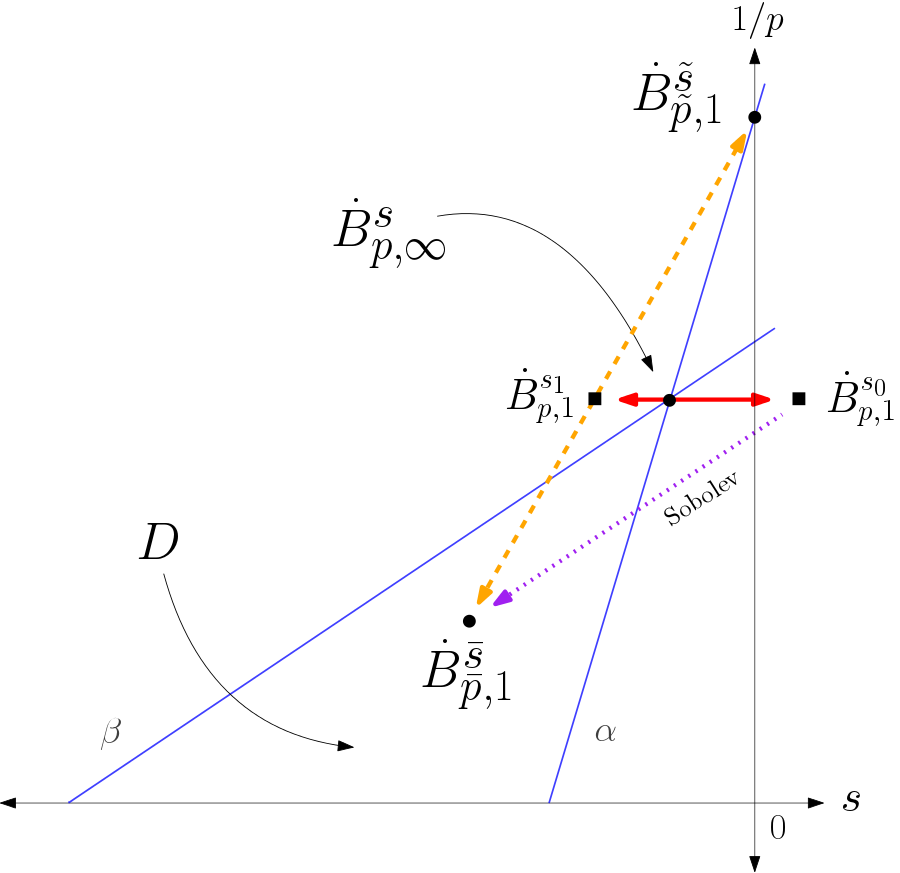}
	\caption{Illustration of Proposition~\ref{besovsplittingpro}. The original function $u \in \dot B^{s}_{p,\infty}$ is split horizontally into $f^K \in \dot B^{s_0}_{p,1}$ and $g^K \in \dot B^{s_1}_{p,1}$ along the solid red line. Next, $g^K$ is split diagonally along the orange dashed line into $\tilde{u}^N = \tilde{g}^{K,N} \in \dot B^{\tilde{s}}_{\tilde{p},1}$ and $\bar{g}^{K,N} \in \dot B^{\bar{s}}_{\bar{p},1}$. Finally, using Sobolev embedding along the dotted purple line, $f^K$ and $\bar{g}^{K,N}$ are combined to form $\bar{u}^N$.}
\label{fig:splittingfigureappendix}
\end{figure}
\addtocounter{theorem}{1}

\begin{proof}
We claim that $s,s_0,s_1$ are distinct and that $s$ lies on the open segment connecting $s_0$ and $s_1$. If $s_0$ were equal to $s_1$, then $\alpha$ would equal $\beta$. Moreover, $s$ is strictly between $s_0$ and $s_1$ because the slope of the line between $(s,1/p)$ and $(\bar{s},1/\bar{p})$ is strictly between the slope of $\alpha$ and the slope of $\beta$. See Figure~\ref{fig:splittingfigureappendix}. Hence, the hypotheses of Remark~\ref{nondiagonalsplittingrmk} are satisfied.

Let $K>0$, to be specified below. By Remark~\ref{nondiagonalsplittingrmk}, we obtain $u=f^{K} + \tilde{g}^{K,N} + \bar{g}^{K,N}$ satisfying the properties described in the remark. Notice that $f^K$ also belongs to $\dot B^{\bar{s}}_{\bar{p},1}$ due to Sobolev embedding and our particular choice of $s_0$. Moreover, \begin{equation}\label{sobolevembeddingsplitting}
\norm{f^{K}}_{\dot{B}^{\bar{s}}_{\bar{p},1}}\leq C(\bar{s}, s_{0},\bar{p}, p_{0})\norm{f^{K}}_{\dot{B}^{{s_0}}_{{p_0},1}}.
\end{equation}
 We define 
\begin{equation}
\quad \tilde{u}^N = \tilde{g}^{K,N}, \quad \bar{u}^{N} = f^K + \bar{g}^{K,N}.
 \end{equation}
 By the triangle inequality, the estimates in Remark~\ref{nondiagonalsplittingrmk} and (\ref{sobolevembeddingsplitting}),
 \begin{equation}
  \norm{\tilde{u}^N}_{\dot B^{\tilde{s}}_{\tilde{p},1}} \leq \frac{C}{1-2^{|s_1-s|}}  K^{s_1-s} N^{1-\frac{p}{\tilde{p}}} \norm{u}_{\dot B^s_{p,\infty}},
 \end{equation}
 \begin{equation}
  \norm{\bar{u}^N}_{\dot B^{\bar{s}}_{\bar{p},1}} \leq C \left(\frac{1}{1-2^{|s_0-s|}} + \frac{1}{1-2^{|s_1-s|}} \right) \left( K^{s_0-s} + K^{s_1-s} N^{1-\frac{p}{\bar{p}}} \right) \norm{u}_{\dot B^s_{p,\infty}}.
 \end{equation}
 Substituting $K = N^{(1-\frac{p}{\bar{p}})/(s_0-s_1)}$ gives the desired estimates. The persistency property~\eqref{besovsplittingpersistencyproperty} also follows from Remark~\ref{nondiagonalsplittingrmk} and the triangle inequality.
 \end{proof}

\begin{proof}[Proof of Lemma~\ref{splitlem}]
Let $d = m \geq 3$, $p > d$, $s = s_p := -1+\frac{d}{p}$, $\tilde{p}=2$, $\tilde{s} = 0$, $\bar{p}=2p$, and $\bar{s} = (s_{2p} + \dot{s})/2$ in
Proposition~\ref{besovsplittingpro}. Here, $\dot{s}$ is defined such that $(\dot{s},1/\bar{p}) \in \alpha$:
\begin{equation}
	\dot{s} = \frac{\frac{1}{2p} - \frac{1}{2}}{\frac{1}{p}-\frac{1}{2}} s.
\end{equation}
Hence, $(\bar{s},1/\bar{p}) \in D$, and we may apply Proposition~\ref{besovsplittingpro} to obtain the desired splitting. The proof is completed by applying the Leray projector $\bP$ onto divergence-free vector fields. Recall that $\bP$ is continuous on homogeneous Besov spaces, see~\cite[Proposition 2.30]{bahourichemindanchin}.
\end{proof}

We now state and prove an analogous splitting lemma for the forcing term. %The proof is easier than that of Lemma~\ref{splitlem} because %the functions are only measured on a finite time interval. Hence,
%the function spaces in question are defined only for $T < \infty$. In particular, they are not homogeneous.  %One may simply lower the logarithmic $q$ index, which here is related to the time integration, at the cost of also lowering the regularity $s$ index. 
%\blue{Not certain about $\delta_3 < -s_{p_3}'$.}

\begin{lemma}[Splitting of forcing term]\label{splitforcing}
Let $T > 0$ and $p \in ]3,\infty[$. There exist $p_3 \in ]3,\infty[$,
%$\delta_3 \in ]0,- s_{p_3}'[$
$\delta_3>0$, and $C > 0$, each depending only on $p$, such that for each $F \in \cF_{p}(Q_T)$ and $N > 0$, there exist $\bar{F}^N \in \cF_{p_3}^{s_{p_3}'+\delta_3}(Q_T) \cap \cF_{p}(Q_T)$ and $\tilde{F}^N \in L^3_t L^2_x(Q_T) \cap \cF_{p}(Q_T)$ with the following properties:
	\begin{equation}
		F =  \tilde{F}^N + \bar{F}^N,
		\label{}
	\end{equation}
		\begin{equation}
		\norm{\tilde{F}^N}_{L^3_t L^2_x(Q_T)} \leq C T^{-\frac{1}{12}} N^{1-\frac{p}{2}} \norm{F}_{\cF_{p}(Q_T)},
		\label{}
	\end{equation}
	\begin{equation}\label{eq:Fbarnorm}
		\norm{\bar{F}^N}_{\cF_{p_3}^{s_{p_3}'+\delta_3}(Q_T)} \leq C T^{\frac{\delta_3}{2}} N^{\frac{1}{2}} \norm{F}_{\cF_{p}(Q_T)}.
	\end{equation}
	Furthermore,
	\begin{equation}
		\norm{\tilde{F}^{N}}_{\cF_{p}(Q_T)}, \norm{\bar{F}^{N}}_{\cF_{p}(Q_T)} \leq \norm{F}_{\cF_{p}(Q_T)}.
		\label{eq:barFleqM}
	\end{equation}
	%\begin{equation}
	%	 \leq M.
	%	\label{eq:tildeFleqM}
	%\end{equation}
\end{lemma}

The proof of the splitting lemma for the forcing term is easier because %the functions are only measured on a finite time interval. Hence,
the function spaces in question are not homogeneous.  %One may simply lower the logarithmic $q$ index, which here is related to the time integration, at the cost of also lowering the regularity $s$ index. 
%\blue{Not certain about $\delta_3 < -s_{p_3}'$.}

\begin{proof}[Proof of Lemma~\ref{splitforcing}]
By a scaling argument, we need only to consider the case $T=2$.

We define a retraction $R$ from the space of measurable tensor fields on $Q_T = \R^3 \times ]0,T[$ to the space of sequences (over $\Z_{\leq 0}$) of measurable tensor fields on $\R^3 \times ]1,2[$:
\begin{equation}
	RG = (G_j)_{j \leq 0}, \quad G_j(\cdot,t) := G(\cdot,2^j t) \chi_{]2^{j},2^{j+1}[}(2^jt), \quad t \in ]1,2[, \; j \leq 0.
\end{equation}
The retraction $R$ is invertible, and the co-retraction $S$ is its inverse. Namely, 
\begin{equation}
S(G_{j})(\cdot,t) = \sum_{j\in\mathbb{Z}_{\leq 0}} G_{j}(\cdot, 2^{-j}t)\chi_{]2^{j},2^{j+1}[}(t).
\end{equation}
For $p \in [1,\infty]$ and $s \in \R$, we consider the space $\ell^s_\infty (L^\infty_t L^p_x)$ consisting of sequences $(G_j)_{j \leq 0}$ of locally integrable tensor fields $G_j$ on $\R^3 \times ]1,2[$ such that
\begin{equation}
\norm{(G_j)_{j \leq 0}}_{\ell^s_\infty (L^\infty_t L^p_x)} := \sup_{j \leq 0} 2^{-\frac{js}{2}} \norm{G_j}_{L^\infty_t L^p_x(\R^3 \times ]1,2[)} < \infty.
 \end{equation}
Note the $-1/2$ in the exponent.
Then $R \: \cF^s_p(Q_T) \to \ell^s_\infty (L^\infty_t L^p_x)$ and $S \: \ell^s_\infty (L^\infty_t L^p_x) \to \cF^s_p(Q_T)$ are continuous maps with norms depending only on $s$.

Let $p > 3$, $\sigma=s_p' := -2+3/p$, $\tilde{p}=2$, $\bar{p} = 2p$, and $q=\tilde{q}=\bar{q}=\infty$. Let $\tilde{s} = -7/12$ and $\bar{s} \in \R$ such that $(s_p',1/p)$, $(\tilde{s},1/2)$ and $(\bar{s},1/\bar{p})$ are colinear. Note that $\bar{s} > s_{\bar{p}}'$, since the slope of the segment from $(s_p',1/p)$ to $(-7/12,1/2)$ is less than $1/3$.\footnote{For comparison, one may verify that the slope of the segment from $(s_p',1/p)$ to $(-1/2,1/2)$ is $1/3$.}

By Remark~\ref{rmkgeneraldiagonal}.2, we may apply the diagonal splitting in $\ell^s_\infty (L^\infty_t L^p_x)$ to $RF = (F_j)_{j \leq 0}$ and obtain 
\begin{equation}
	F_j = \tilde{F}_j^N + \bar{F}_j^N, \quad j \leq 0.
\end{equation} Denote $\tilde{F}^N = S(\tilde{F}^N_j)_{j \leq 0}$ and $\bar{F}^N = S(\bar{F}^N_j)_{j \leq 0}$. Then $F = \tilde{F}^N + \bar{F}^N$, and
\begin{equation}
	\int_0^T \norm{\tilde{F}^N(\cdot,t)}_{L^2(\R^3)}^3 \, dt \leq \int_0^T t^{\frac{3 \tilde{s}}{2}} \, dt \times \norm{\tilde{F}^N}_{\cF^{\tilde{s}}_2(Q_T)}^3 \overset{\tilde{s}=-\frac{7}{12}}{\leq} C N^{1-\frac{p}{2}} \norm{F}_{\cF_p(Q_T)}^3,
	\end{equation}
	\begin{equation}
	\norm{\bar{F}^{N}}_{\cF^{\bar{s}}_{\bar{p}}(Q_T)} \leq C N^{\frac{1}{2}} \norm{F}_{\cF_p(Q_T)}.
\end{equation}
Finally, we have the persistency property:
\begin{equation}
	\norm{\tilde{F}^N}_{\cF_p(Q_T)}, \norm{\bar{F}^N}_{\cF_p(Q_T)} \leq \norm{F}_{\cF_p(Q_T)}.
\end{equation}
We define $p_3 := \bar{p}$ and $\delta_3 := \bar{s} - s_{p_3}' > 0$ to complete the proof.
\end{proof}

\section{Appendix: $\varepsilon$-regularity}

	In this section, we recall an $\varepsilon$-regularity criterion for the three-dimensional Navier-Stokes equations and some of its important consequences, following \cite{ckn,lin,ladyzhenskayaseregin,escauriazasereginsverak,rusinsverakminimal}. In particular, we will state without proof certain results with forcing terms which we could not find in the literature and indicate what modifications are necessary to prove them.

	Our main definition is adapted from the one in F. H. Lin's paper~\cite{lin}.
\begin{definition}[Suitable weak solution]
Let $Q$ denote a parabolic ball
\begin{equation}
Q(z_0,r) := B(x_0,r) \times ]t_0-r^2,t_0[
\label{}
\end{equation}
for some $z_0 = (x_0,t_0) \in \R^{3+1}$ and $r > 0$. Suppose that $v \in L^\infty_t L^2_x \cap L^2_t H^1_x(Q)$, $q, f \in L^{\frac{3}{2}}_\loc(Q)$, and $F \in L^2_\loc(Q)$.

	We say that $(v,q)$ is a \emph{suitable weak solution} of the Navier-Stokes equations on $Q$ with forcing term $f + \div F$ if
	\begin{equation}
		\left.
		\begin{aligned}
		\p_t v - \Delta v + v \cdot \nabla v &= - \nabla q + f + \div F \\
		\div v &= 0
	\end{aligned}
	\right\rbrace \text{ in } Q
		\label{}
	\end{equation}
	in the sense of distributions, and the local energy inequality
	\begin{equation}
		\label{localenergyineqappendix}
		\int_{B(x_0,r)} |v(x,t)|^2 \varphi \, dx + 2 \int_Q |\nabla v|^2 \varphi \, dx \,dt \leq $$ $$ \leq \int_Q (\p_t \varphi + \Delta \varphi) |v|^2 +  (|v|^2+2p) v \cdot \nabla \varphi + 2f\cdot (\varphi v) - 2F : \nabla (\varphi v) \, dx \,dt
	\end{equation}
is satisfied for a.e. $t \in ]t_0-r^2,t_0[$ and all $0 \leq \varphi \in C^\infty_0(Q)$.
\end{definition}
%$0 \leq \varphi \in C^\infty_0(Q \cup (B(x_0,r) \times \{ t_0 \}))$.

The following $\varepsilon$-regularity criterion for suitable weak solutions may be proven by copying the scheme of Ladyzhenskaya and Seregin in~\cite{ladyzhenskayaseregin}. Higher regularity with zero forcing term was demonstrated in~\cite{necas} according to the arguments in Serrin's paper~\cite{serrinsmoothness}.

\begin{pro}[$\varepsilon$-regularity]\label{pro:epsilonregularity}
Let $\delta > 0$ and $p_1,p_2,q_1,q_2 \in ]1,\infty]$ satisfying
\begin{equation}
	\frac{2}{q_1} + \frac{3}{p_1} = 3-\delta, \quad \frac{2}{q_2} + \frac{3}{p_2} = 2-\delta.
\label{}
\end{equation} There exist constants $\varepsilon_{\rm CKN}, c_0 > 0$ depending on $p_1,p_2,q_1,q_2$ such that for all $z \in \R^{3+1}$, $R > 0$, and suitable weak solutions $(v,q)$ on $Q(z,R)$ with forcing term $f + \div F$, $f \in L^{q_1}_t L^{p_1}_x(Q(z,R))$, $F \in L^{q_2}_t L^{p_2}_x(Q(z,R))$, the condition
	\begin{equation}
		\frac{1}{R^2} \int_{Q(z,R)} |v|^3 + |q|^{\frac{3}{2}} \, dx' \, dt' + R^\delta \norm{f}_{L^{q_1}_t L^{p_1}_x(Q(z,R))} + R^\delta \norm{F}_{L^{q_2}_t L^{p_2}_x(Q(z,R))} < \varepsilon_{\rm CKN}
		\label{}
	\end{equation}
	implies that
	$v \in C^\alpha_{\rm par}(Q(z,R/2))$, and
	\begin{equation}
		\norm{v}_{L^\infty(Q(z,R/2))} + R^{\alpha} [v]_{C^\alpha_{\rm par}(Q(z,R/2)} \leq \frac{c_0}{R}.
		\label{}
	\end{equation}
	If the condition is satisfied and $f$, $F$ are zero, then $\nabla^{\ell} v \in C^{\alpha}_{\rm par}(Q(z,R/2))$ for all $\ell \in \N$, and there exist absolute constants $c_{0,\ell} > 0$, $\ell \in \N$, such that
	\begin{equation}
		\norm{\nabla^{\ell} v}_{L^\infty(Q(z,R/2))} + R^{\alpha} [\nabla^{\ell} v]_{C^\alpha_{\rm par}(Q(z,R/2)} \leq \frac{c_{0,\ell}}{R^{\ell+1}}.
		\label{}
	\end{equation}
\end{pro}

The following lemma was proven without forcing terms by F. H. Lin in~\cite[Theorem 2.2]{lin}. In our situation, the local energy inequality~\eqref{localenergyineqappendix} for the limit solution must be obtained in a slightly indirect way which is similar to the proof of Proposition~\ref{pro:stability}, see below.
\begin{lemma}[Weak--$\ast$ stability for suitable weak solutions]\label{suitablestability}
Let $(v^{(n)},q^{(n)})_{n \in \N}$ be a sequence of suitable weak solutions on $Q := Q(z,r)$ with respective forcing terms $f^{(n)} + \div F^{(n)}$, $n \in \N$, for some $z \in \R^{3+1}$ and $r > 0$. Furthermore, suppose that
	\begin{equation}
		v^{(n)} \wstar v \text{ in } L^\infty_t L^2_x(Q), \quad \nabla v^{(n)} \wto \nabla v \text{ in } L^2(Q),
		\label{}
	\end{equation}
	\begin{equation}
		v^{(n)} \to v \text{ in } L^3(Q), \quad q^{(n)} \wto q \text{ in } L^{\frac{3}{2}}(Q),
		\label{}
	\end{equation}
	\begin{equation}
		f^{(n)} \wto f \text{ in } L^{\frac{3}{2}}(Q), \quad F^{(n)} \wto F \text{ in } L^2(Q).
		\label{}
	\end{equation}
	Let $p_1,q_1,p_2,q_2 \in ]1,\infty]$ such that
	\begin{equation}
		\frac{2}{q_1} + \frac{3}{p_1} < 3, \quad \frac{2}{q_2} + \frac{3}{p_2} < 2.
		\label{fFspaces}
	\end{equation}
	Finally, suppose that
	\begin{equation}
		f^{(n)} \wstar f \text{ in } L^{q_1}_t L^{p_1}_x(Q), \quad F^{(n)} \wstar F \text{ in } L^{q_2}_t L^{p_2}_x(Q).
		\label{}
	\end{equation}
	Then there exists $q \in L^{\frac{3}{2}}_\loc(Q)$ such that $(v,q)$ is a suitable weak solution on $Q$ with forcing term $f + \div F$.
\end{lemma}
We may assume that $Q \subset \R^3 \times \R_+$. To prove Lemma~\ref{suitablestability}, one extends $f^{(n)}, F^{(n)}$ by zero to the whole space and defines
\begin{equation}
	V^{(n)}(\cdot,t) := \int_0^t S(t-s) \bP f^{(n)}(\cdot,s) + S(t-s) \bP \div F^{(n)}(\cdot,s) \, ds
	\label{linearevolutionappendix}
\end{equation}
for each $n \in \N$.
Each suitable weak solution is decomposed into the linear solution above and a correction term: $v^{(n)} = V^{(n)} + u^{(n)}$. 
Similarly, one writes $v = V + u$. Next, one ``transfers'' the local energy inequality~\eqref{localenergyineqappendix} from the velocity field $v^{(n)}$ to obtain a local energy inequality for the correction $u^{(n)}$ on $Q$ with lower order terms and a forcing term which converges locally strongly. This is described in the proof of Proposition~\ref{pro:vlocalglobal}. The local energy inequality for the corrections $u^{(n)}$ is stable under the limiting procedure, see the proof of Proposition~\ref{pro:stability}. Finally, one transfers the local energy inequality from the limit correction $u$ to obtain~\eqref{localenergyineqappendix} for the velocity field $v$, see Remark~\ref{otherdirectionlocalenergy}.

The final proposition may be proved as in~\cite[Lemma 2.1]{rusinsverakminimal}.

\begin{pro}[Persistence of singularity]\label{persistenceofsingularity}
	Assume the hypotheses of Lemma~\ref{suitablestability}.
	% \blue{with the subcriticality condition
	%	\begin{equation}
	%	\frac{2}{q_1} + \frac{3}{p_1} < 3, \quad \frac{2}{q_2} + \frac{3}{p_2} < 2
	%	\label{}
	%\end{equation}
	%instead of~\eqref{fFspaces}.}
	Moreover, assume that $z$ is a singular point of $v^{(n)}$ for each $n \in \N$.
	Then $z$ is a singular point of $v$.
\end{pro}

\subsubsection*{Acknowledgments}
The first author was partially supported by the NDSEG Graduate Fellowship. He also thanks his advisor, Vladim{\'i}r {\v S}ver{\'a}k, as well as Simon Bortz and Raghavendra Venkatraman for helpful suggestions. The second author was supported by an EPSRC Doctoral Prize award. We also thank P.~G. Lemari{\'e}-Rieusset for pointing out the reference~\cite{krepkogorskii}.

\bibliographystyle{plain}
%\bibliography{besov}
\bibliography{arxivsubmission}

\end{document}